\documentclass[11pt]{amsart}
\usepackage{graphicx}
\usepackage[pagebackref, colorlinks=true, linkcolor=black, citecolor=black, urlcolor=black]{hyperref}

\usepackage{xargs}                      
\usepackage[colorinlistoftodos,prependcaption 
]{todonotes}                                                                                                                 

\newcommandx{\todoN}[2][1=]{\todo[linecolor=red,backgroundcolor=white,bordercolor=red,#1]{#2}}			

\usepackage{comment}
\usepackage{calc}

\usepackage{amsmath, amsthm, amsfonts, amssymb}
\usepackage{enumerate}
\usepackage{bbm}
\usepackage{mathrsfs}
\usepackage{amssymb}
\usepackage{mathtools}
\usepackage[all]{xy}
\usepackage{tikz}
\usepackage{tikz-cd}
\usepackage{adjustbox}

\usepackage{caption}
\usetikzlibrary{matrix,arrows,decorations.pathmorphing}
\usepackage{xcolor}
\usepackage{bm}
	\usetikzlibrary{calc}
	\usetikzlibrary{trees}
	\tikzstyle{every picture}=[scale=.35,inner sep=0]
	\usepackage{color}

\usepackage{stmaryrd} 

\usepackage[normalem]{ulem} 

\newtheorem{thm}{Theorem}

\theoremstyle{definition}

\theoremstyle{theorem}
\newtheorem{theorem}{Theorem}[section]
\newtheorem{lemma}[theorem]{Lemma}

\newtheorem{corollary}[theorem]{Corollary}

\theoremstyle{definition}

\theoremstyle{remark}
\newtheorem{remark}[theorem]{Remark}

\newtheorem{example}[theorem]{Example}

\numberwithin{equation}{section}

\makeatletter
\@namedef{subjclassname@2020}{\textup{MSC2020}}
\makeatother

\begin{document}

\title[Heisenberg vertex algebras and abelian varieties]{Heisenberg vertex algebras\\ and abelian varieties}

\author[N.~Tarasca]{Nicola Tarasca}
\address{Nicola Tarasca 
\newline \indent Department of Mathematics \& Statistics
\newline \indent Virginia Commonwealth University, Richmond, VA 23284}
\email{tarascan@vcu.edu}


\begin{abstract}
Here we construct spaces of coinvariants for Heisenberg vertex algebras on abelian varieties and show that these globalize to twisted $\mathscr{D}$-modules on moduli of abelian varieties. We recover the standard construction of these sheaves over the Torelli locus. As an example, in the case of commutative Heisenberg vertex algebras we show that the realization of coinvariants on curves as polynomial functions on one-forms extends to arbitrary abelian varieties.
Similarly, we construct twisted coinvariants on the universal abelian variety and show that these recover the standard construction over the universal Jacobian.
\end{abstract}

\vspace*{-1.5pc}

\maketitle

\vspace{-1.5pc}

\section*{Introduction}

Vertex operator algebras (VOAs) have a natural geometric realization on algebraic curves via \textit{spaces of coinvariants.} 
Specifically, for a VOA $V$, these spaces are the quotients of $V$ by the action of \textit{Lie algebras of outgoing states} determined by algebraic curves. Such constructions globalize to twisted $\mathscr{D}$-modules on the moduli space of curves $\mathcal{M}_g$ \cite{tuy, bfm, bd, bzf}.

As the Torelli theorem places $\mathcal{M}_g$ inside the broader moduli space $\mathcal{A}_g$ of \textit{principally polarized abelian varieties} (PPAVs), a natural question arises: are spaces of coinvariants intrinsic only to curves, or are they the restriction to the Torelli locus of a more general construction valid for all PPAVs?
In this work, we show that for \textit{Heisenberg vertex algebras,} the spaces of coinvariants can indeed be defined for arbitrary PPAVs, naturally recovering the standard construction over the Torelli locus.

Our approach applies the Arbarello--De Concini moduli space $\widehat{\mathcal{A}}_g$ of \textit{extended PPAVs} \cite{adc91}.
For a Heisenberg vertex algebra $V$ and an extended PPAV $a$, we define a space of coinvariants $\widehat{\mathbb{V}}(a,V)$ as the quotient of $V$ by the action of a Lie algebra of outgoing states determined by both $a$ and $V$ --- see \eqref{eq:VhataV}.
We then extend the classical curve-based framework to this broader PPAV setting, and thus obtain the following.
Let $\Lambda$ be the Hodge line bundle on $\mathcal{A}_g$.

\begin{thm}
\label{thm:mainintroAg}
For a Heisenberg vertex algebra $V$ constructed from a non-degenerate lattice $\Gamma$ of rank $d$:
\begin{enumerate}[(i)]
\item The spaces $\widehat{\mathbb{V}}(a,V)$ give rise to a quasi-coherent sheaf $\mathbb{V}(V)$ on $\mathcal{A}_g$ whose restriction to the Torelli locus recovers the standard sheaf of coinvariants on $\mathcal{M}_g$. 

\item For non-zero level, $\mathbb{V}(V)$ has a twisted $\mathscr{D}$-module structure
induced by an action of the Atiyah algebra $\frac{d}{2}\,\mathscr{F}_\Lambda$.
\end{enumerate}
\end{thm}

Restricting to the Torelli locus, part (ii) is compatible with the results from \cite{dgt} on the Atiyah algebra acting on sheaves of coinvariants on~$\mathcal{M}_g$.

Sheaves of coinvariants have also been constructed in Frenkel--Ben Zvi \cite[\S18.1.15]{bzf}
on  relative Picard bundles on $\mathcal{M}_g$ by incorporating line bundles on curves into the construction. 
Also in this context, we show that such sheaves are in fact the restriction of sheaves defined more generally on the universal abelian variety $\mathcal{X}_g \rightarrow \mathcal{A}_g$. 
Here $\mathcal X_g\to\mathcal A_g$ is the family whose fiber over $(X,\Theta)$
parametrizes the line bundles on $X$ inducing the polarization. It is a torsor under
the genuine universal abelian variety over $\mathcal A_g$ and becomes isomorphic to it after a
finite \'etale base change. 
Over the Torelli locus, $\mathcal X_g$ restricts to the relative Picard variety of degree $g - 1$.
By abuse of terminology, we refer to $\mathcal X_g$ as the universal abelian variety. 

We start by first defining a twisted space $\widehat{\mathbb{V}}(x,V,N)$ for a point $x$ of the universal extended abelian variety $\widehat{\mathcal{X}}_g$, a Heisenberg vertex algebra $V$ of non-zero level constructed from a non-degenerate lattice $\Gamma$ of rank $d$, and a choice of $N\in \Gamma\otimes_{\mathbb{Z}}\mathbb{C}$ such that $(N,N)=d$, see \eqref{eq:VhatxVN}. 
The choice of such $N$ is the choice of an embedding in $V$ of the rank-one Heisenberg vertex algebra of level $d$.
Let $\Xi=\mathscr{O}_{\mathcal{X}_g}(2\theta)$, where $\theta$ is the universal symmetric theta divisor on $\mathcal{X}_g$ (\S\ref{sec:univPPAV}).

\begin{thm}
\label{thm:mainintroXg}
For a Heisenberg vertex algebra $V$ of non-zero level constructed from a non-degenerate lattice $\Gamma$ of rank $d$ and a choice of $N\in \Gamma\otimes_{\mathbb{Z}}\mathbb{C}$ with \mbox{$(N,N)=d$:}
\begin{enumerate}[(i)]
\item The spaces $\widehat{\mathbb{V}}(x,V,N)$  give rise to a sheaf $\mathbb{V}(V,N)$ on $\mathcal{X}_g$ which restricts to the standard sheaf of twisted coinvariants on the relative Picard variety of degree $g - 1$. 

\item $\mathbb{V}(V,N)$ has a twisted $\mathscr{D}$-module structure
induced by an action of the Atiyah algebra $\frac{d}{2}\,\mathscr{F}_\Xi$.
\end{enumerate}
\end{thm}

The twisted $\mathscr D$-module structures in Theorems \ref{thm:mainintroAg} and \ref{thm:mainintroXg} originate in
\cite{avva}. There, sheaves of coinvariants on $\mathcal A_g$ and $\mathcal X_g$ were
constructed for certain representations of the metaplectic algebra, using Lie
algebras of outgoing states that depend only on the points of
$\widehat{\mathcal A}_g$ and $\widehat{\mathcal X}_g$. That construction makes no
reference to curves, and its comparison with the classical coinvariants was missing. 
Here we carry out this comparison for Heisenberg vertex algebras.
We describe the coinvariants through Lie algebras of outgoing states that depend also
on the vertex algebra and that restrict to the classical ones over the Torelli locus.
For non-zero level, the two descriptions yield the same coinvariants
(Remarks~\ref{rmk:avvavsavha} and \ref{rmk:avvavsavha2}). Hence the sheaves of \cite{avva} for
Heisenberg vertex algebras restrict to the standard sheaves of coinvariants on
$\mathcal M_g$ and on relative Picard varieties. The description given here applies
also at level zero, which is not covered by \cite{avva}. 

\smallskip

Finally, we present an explicit realization of coinvariants for the commutative Heisenberg vertex algebra $\pi^0$ of level zero and rank one. 
In this case, it is shown in \cite[\S 9.4.3]{bzf} building on Beilinson--Drinfeld \cite{beilinson1991quantization} that the space of coinvariants $\mathbb{V}\left(C,P, \pi^0\right)$ assigned to a pointed curve $(C,P)$ admits a canonical isomorphism
\begin{equation}
\label{eq:VCPpizeroiso}
\mathbb{V}\left(C,P, \pi^0\right) \cong \mathrm{Sym}\left(H^0\left(C, \Omega^1_C\right)^*\right).
\end{equation}
Equivalently, 
\[
\mathbb{V}\left(C,P, \pi^0\right) \cong \mathrm{Fun}\left(H^0(C, \Omega^1_C)\right),
\]
the ring of polynomial functions on the space of one-forms $H^0(C, \Omega^1_C)$. Since the spaces $\mathbb{V}\left(C,P, \pi^0\right)$ obtained for various $P$ in $C$ are all canonically isomorphic, they determine a space $\mathbb{V}\left(C, \pi^0\right)$ assigned to $C$ alone. Moreover, since 
\[
H^0\left(C, \Omega^1_C\right) \cong H^0\left(\mathrm{Jac}(C), \Omega^1_{\mathrm{Jac}(C)}\right), 
\]
the isomorphism \eqref{eq:VCPpizeroiso} can be rephrased in terms of the Jacobian $\mathrm{Jac}(C)$ as
\[
\mathbb{V}\left(C, \pi^0\right) \cong \mathrm{Sym}\left(H^0\left(\mathrm{Jac}(C), \Omega^1_{\mathrm{Jac}(C)}\right)^*\right).
\]
We show that in fact this realization extends beyond the Torelli locus:

\begin{thm}
\label{thm:Vpi0}
For a principally polarized abelian variety $(X,\Theta)$, 
the space of coinvariants $\mathbb{V}\left(X,\Theta, \pi^0\right)$ admits a canonical isomorphism
\[
\mathbb{V}\left(X,\Theta, \pi^0\right) \cong \mathrm{Sym}\left(H^0(X, \Omega^1_X)^*\right).
\]
Equivalently, $\mathbb{V}\left(X,\Theta, \pi^0\right) \cong \mathrm{Fun}\left(H^0(X, \Omega^1_X)\right)$.
\end{thm}

Passing to the dual spaces, this implies a canonical isomorphism of the space of conformal blocks
\[
\mathbb{V}\left(X,\Theta, \pi^0\right)^\vee \cong \widehat{\mathrm{Sym}}\, H^0(X, \Omega^1_X),
\]
the completed symmetric algebra of the space of one-forms on $X$, extending the similar isomorphism in the curve case from \cite[\S10.5.11]{bzf}.
We also show that an analogous statement holds for commutative Heisenberg vertex algebras $\pi^0_\Gamma$ of arbitrary rank, see
Theorem \ref{thm:Vpi0Gamma}.

\smallskip

Looking ahead, it would be interesting to find a realization of coinvariants on PPAVs also in the case of the Heisenberg vertex algebra $\pi$ of level $1$. In the curve case, such a realization is known in terms of those polydifferentials on a disk that extend to the curve --- see \cite{beilinson1994affine} and \cite[\S 10.5]{bzf}.

On a more general note, it would be interesting to extend the results of this work to the case of lattice vertex algebras.
In the curve case, positive-definite even lattice VOAs yield vector bundles of coinvariants of \textit{finite rank} over $\mathcal{M}_g$ \cite{dgt2}.
As a lattice vertex algebra contains a Heisenberg vertex subalgebra, one could consider the Lie algebra of outgoing states used here in that case as well. However, the resulting construction does not recover the standard one on curves, as the Lie algebra of outgoing states for lattice vertex algebras on curves also incorporates the additional vertex operators of the lattice vectors. 
Thus it remains to determine whether the Lie algebra of outgoing states admits an extension to PPAVs also for lattice vertex algebras. We plan to return to this question in future work.

\smallskip

The paper is structured as follows. We review coordinatized curves, extended abelian varieties, and their moduli spaces in \S\ref{sec:backg}. In \S\ref{sec:HVOA} we review Heisenberg vertex algebras constructed from a lattice and the corresponding metaplectic action. 
We define Lie algebras of outgoing states and the resulting spaces of coinvariants and conformal blocks at extended PPAVs in \S\ref{sec:spaces}.
In \S\ref{sec:homogspaces} we discuss the space $\mathrm{Sp}(X,\Theta)$ of all extended PPAVs 
together with an identification of their underlying PPAV with $(X,\Theta)$. We show how this is a homogeneous space for the action of a group $\mathbb{S}\mathrm{p}^+\left(H' \right)$ with nontrivial stabilizers. We apply this in \S\ref{sec:VPPAV}, where we construct spaces of coinvariants assigned to $(X,\Theta)$ by descent along $\mathrm{Sp}(X,\Theta)$.
In \S\ref{sec:symmetagpsch} we treat the invariance of isotropy sheaves.
We use this in the construction of sheaves of coinvariants over $\widehat{\mathcal{A}}_g$ in \S\ref{sec:sheavesaghatag} and over $\widehat{\mathcal{X}}_g$ in \S\ref{sec:sheavesaghatxg}.
In \S\ref{sec:comp} we show the compatibility with the twisted coinvariants on relative Picard bundles from \cite{bzf}.
Theorem \ref{thm:mainintroAg} is proven in \S\ref{sec:sheavesaghatag} and Theorem \ref{thm:mainintroXg} in \S\ref{sec:comp}.
Throughout, we work over $\mathbb{C}$.

\section{Background}
\label{sec:backg}

Here we review extended PPAVs and their moduli spaces from \cite{adc91}. We refer to \cite[\S\S 1--2]{avva} for an alternative review of these concepts.

\subsection{The Heisenberg algebra} 
The \textit{Heisenberg algebra} is $H:=\mathbb{C}((t))$ with Lie bracket given by
\begin{equation}
\label{eq:symplectic}
\langle f,h\rangle := - \mathop{\mathrm{Res}}_{t=0} \, f \, dh \qquad \mbox{for $f,h\in H$}.
\end{equation}
The underlying vector space $\mathbb{C}((t))$ is a complete topological vector space with respect to the $t$-adic topology and is topologically generated by $b_i:=t^i$ for $i\in\mathbb{Z}$. As the Lie bracket \eqref{eq:symplectic} is continuous with respect to this topology, $H$ is a complete topological Lie algebra. 

Let $H':=H/\mathbb{C}\,b_0$. 
The alternating bilinear form \eqref{eq:symplectic} is nondegenerate on $H'$, hence it endows $H'$ with a symplectic structure.
The algebra $H$ can be viewed as the central extension of $H'$ defined by the two-cocycle \eqref{eq:symplectic}. Below we will also use the following notation:
\[
H_- := t^{-1}\mathbb{C}\left[ t^{-1}\right],\qquad
H'_+ := t\mathbb{C}\llbracket t \rrbracket.
\]

\subsection{Extended PPAVs}
An \textit{extended PPAV} of dimension $g$ is a triple $(Z, F, L)$ where 
\begin{align*}
&\mbox{$Z$ is a Lagrangian subspace of $H'$,} \\
&\mbox{$F$ is a codimension  $g$ subspace of $Z$, and} \\
&\mbox{$L$ is a rank $2g$ lattice in $F^\perp/F$,}
\end{align*}
satisfying the following four suitable conditions. 
The first condition is
\begin{equation}
\label{eq:cond1}
Z \cap H'_+ = 0.
\end{equation}
This implies the decompositions into maximal isotropic subspaces
\[
H'= Z\oplus H'_+ \quad\mbox{and}\quad F^\perp/F = Z/F \oplus F^\perp_+ \quad \mbox{where} \quad F^\perp_+:= F^\perp \cap H'_+.
\]
It follows that $F^\perp/F$ has dimension $2g$.
Here $F^\perp/F$ is equipped with the nondegenerate symplectic form induced from the form $\langle \, , \, \rangle$ on $H'$ from \eqref{eq:symplectic}.

The second condition is
\begin{equation}
\label{eq:cond2}
 L_\mathbb{R} \cap  F^\perp_+ =0 \quad\mbox{in $F^\perp/F$}, \quad\mbox{where}\quad L_\mathbb{R} := L\otimes_{\mathbb{Z}} \mathbb{R}.
\end{equation}
This implies: the projection $F^\perp/F \rightarrow Z/F$ induces a real isomorphism $L_\mathbb{R} \cong Z/F$; the real structure on $F^\perp/F$ induced by $L_\mathbb{R}$ identifies $F^\perp_+$ with the conjugate of $Z/F$; and the form on $Z/F$ given by
\[
B(u,v):= \frac{1}{\pi}\, \langle \overline{u} , v \rangle \quad \mbox{for $u,v\in Z/F$}
\]
is Hermitian. 

The third and fourth conditions on the triple $(Z, F, L)$ are:
\begin{align}
\label{eq:cond3}
& \frac{1}{2\pi i}\, \langle \, , \, \rangle \,\, \mbox{is unimodular on $L$;}\\[5pt]
\label{eq:cond4}
&\mbox{the form $B$ on $Z/F$ is positive definite.}
\end{align}

An extended PPAV $(Z, F, L)$ determines a PPAV $(X,\Theta)$ of dimension $g$ as follows: $X$ is the quotient of the $g$-dimensional space $Z/F$ by the image of $L$ under the projection \mbox{$p\colon F^\perp/F \rightarrow Z/F$,} and the polarization $\Theta$ is induced by $B$, which is principal  by  \eqref{eq:cond3}.
In particular, this construction yields an isomorphism
\begin{equation}
\label{eq:Z/F}
Z\, / \, F \cong H^0\left(X, \Omega^1_X \right)^*.
\end{equation}
Also, the extended PPAV $(Z, F, L)$  determines the isomorphism class of an extension
\begin{equation}
\label{eq:H_+extX}
0\rightarrow H'_+ \rightarrow H'/K \rightarrow X \rightarrow 0
\end{equation}
where $K$ is the preimage of $L$ under the projection $F^\perp \rightarrow F^\perp/F$.

\subsection{The moduli space of extended PPAVs}
\label{sec:moduliAghat}
Let $\widehat{\mathcal{A}}_g$ be the moduli space of extended PPAVs of dimension $g$ from \cite{adc91}.
This is an infinite-di\-men\-sion\-al complex manifold rational homotopy equivalent to the separated, smooth Deligne-Mumford stack $\mathcal{A}_g$ via the map
\[
\widehat{\mathcal{A}}_g \rightarrow \mathcal{A}_g
\]
given by the above assignment $(Z,F,L)\mapsto (X,\Theta)$. 
The space $\widehat{\mathcal{A}}_g$ is constructed as the quotient of an extended Siegel upper half-space $\widehat{\mathcal{H}}_g$ by a free and properly discontinuous action of the symplectic group $\mathrm{Sp}(2g, \mathbb{Z})$.

We review this construction as it will be needed in \S\ref{sec:homogspaces}. The \textit{extended Siegel upper half-space} $\widehat{\mathcal{H}}_g$ of degree $g$ is the infinite-dimensional complex manifold
\begin{equation*}
\widehat{\mathcal{H}}_g := \widetilde{S}^2(H'_+) \times \mathcal{B}_g\left(H'_+\right) \times \mathcal{H}_g
\end{equation*}
where
\begin{align*}
\widetilde{S}^2\left(H'_+\right) &:= \left\{ \varphi \colon H_- \rightarrow H'_+ \, \Bigg| \,  
\begin{array}{l}
\mbox{$\varphi$ is linear and symmetric,} \\[5pt]
\mbox{i.e., $\langle a, \varphi(b)\rangle = \langle b, \varphi(a)\rangle$  for all  $a,b\in H_-$} 
\end{array}
\right\},\\[5pt]
\mathcal{B}_g\left(H'_+\right) &:= \mbox{the manifold of frames $\bm{h}$ of $(H'_+)^g$,}\\ 
&\,\,\,\,\quad\mbox{i.e., $\bm{h}=(h_1, \dots, h_g)\in(H'_+)^g$ for linearly independent $\{h_i\}_i$,}\\[5pt]
\mathcal{H}_g &:= \mbox{the Siegel upper half-space of degree $g$.}
\end{align*}
There is a natural map 
\begin{equation}
\label{eq:HghattoAghat}
\widehat{\mathcal{H}}_g\rightarrow \widehat{\mathcal{A}}_g, \qquad (\varphi, \bm{h}, \Omega)\mapsto (Z,F,L)
\end{equation}
with $Z$ equal to the graph of $\varphi$, the isotropic space $F\subset Z$ defined by the vanishing of the linear forms $h_1=\cdots=h_g=0$, and the lattice $L$ constructed from $\Omega$.

The symplectic group $\mathrm{Sp}(2g, \mathbb{Z})$ acts transitively on the fibers of \eqref{eq:HghattoAghat}
 as follows: for
\[
\sigma=\left(
\begin{array}{cc}
\alpha & \beta \\
\gamma & \delta
\end{array}
\right)\in \mathrm{Sp}(2g, \mathbb{Z})
\]
one has
\begin{equation}
\label{eq:Sp2gZaction}
\sigma \left(\varphi, \bm{h}, \Omega\right) := \left(\varphi,\, \bm{h}\left(\gamma \Omega + \delta \right)^{-1}, \,
\left(\alpha\Omega+\beta\right)\left(\gamma \Omega + \delta \right)^{-1}\right).
\end{equation}

\subsection{The universal extended PPAV}
\label{sec:univPPAV}
The moduli space $\widehat{\mathcal{A}}_g$ admits a universal family $\widehat{\mathcal{X}}_g \rightarrow \widehat{\mathcal{A}}_g$ which is rational homotopy equivalent to the universal abelian torsor
$\mathcal X_g\to\mathcal A_g$ \cite{adc91}, whose fiber over $(X,\Theta)$ is the torsor
under $X$ parametrizing the line bundles on $X$ inducing the polarization
\cite[\S 2.4]{avva}; it carries the universal symmetric divisor $\theta$, the locus of line bundles whose effective divisor contains the origin.
The fiber over a point $(Z,F,L)$ in $\widehat{\mathcal{A}}_g$ is the contracted product
\begin{equation}
\label{H'KtimesGQ}
\left(H' / K \right) \times_{G} \mathscr{Q}
\end{equation}
where $G$ is the group of points of order $2$ in $H'/K$, and $\mathscr{Q}$ is the $G$-torsor of integral quadratic forms $q$ on $L$ taken modulo $2$ satisfying 
\begin{equation}
\label{eq:quad}
q(a) + q (b) - q(a+b) \equiv \frac{1}{2\pi i} \, \langle a, b \rangle \quad \mbox{mod $2$.}
\end{equation}
In particular, the fiber over $(Z,F,L)$ is non-canonically isomorphic to $H'/K$ from \eqref{eq:H_+extX}, and the map $\widehat{\mathcal{X}}_g\rightarrow \widehat{\mathcal{A}}_g$ has no zero-section.
Points of $\widehat{\mathcal{X}}_g$ are denoted as $(Z, F, L, \overline{h}, q)$.

\subsection{The Torelli map}
\label{sec:Torelli}
The usual Torelli map $\mathcal{M}_g\hookrightarrow \mathcal{A}_g$ extends in the following way.
A \textit{coordinatized curve} of genus $g$ is a triple $(C,P,t)$ consisting of an algebraic curve $C$ of genus $g$, a point $P\in C$, and a local coordinate $t$ at $P$.
The moduli space $\widehat{\mathcal{M}}_g$ of coordinatized curves of genus $g$ is an infinite-dimensional complex manifold rational homotopy equivalent to the separated, smooth Deligne-Mumford stack $\mathcal{M}_g$ via the forgetful map \mbox{$\widehat{\mathcal{M}}_g\rightarrow \mathcal{M}_g$ \cite{adkp, besh}.}

For a coordinatized curve $(C,P,t)$, its \textit{extended Jacobian} is the extended PPAV $a=(Z,F,L)$ defined in \cite[\S 2]{adc91} as follows.
Let $\mathscr{O}(C\setminus P)$ be the space of regular functions on $C\setminus P$. One has an inclusion of commutative Lie algebras \mbox{$\mathscr{O}(C\setminus P)\hookrightarrow \mathbb{C}((t))$} given by the Laurent series expansion at the point $P$ in the coordinate $t$. Let
\begin{equation}
\label{eq:FonTorellilocus}
F:=\mathrm{Im}\left(\mathscr{O}(C\setminus P)\hookrightarrow \mathbb{C}((t)) \rightarrow H'\right).
\end{equation}
Moreover, let $D_P:=\mathrm{Spec}\left( \mathscr{O}_P\right)$ be the disk around the point $P$ and $D^\times_P:=\mathrm{Spec}\left( \mathscr{K}_P\right)$ the punctured disk around $P$. Let
\begin{equation*}
Z:=\mathrm{Im}\left(\mathscr{O}\left(D^\times_P\right)/\mathscr{O}\left(D_P\right)\hookrightarrow H'\right).
\end{equation*}
The Mayer-Vietoris sequence for the cover $\left(D_P, C\setminus P\right)$ of $C$ yields an isomorphism
\begin{equation*}
Z\, / \, F \cong H^1\left(C, \mathscr{O}_C \right).
\end{equation*}
By Serre duality, this recovers \eqref{eq:Z/F} for Jacobians. Under this identification, let
\[
L \cong H^1\left(C, \mathbb{Z} \right).
\]
The above construction defines an injective map of moduli spaces
\begin{equation}
\label{eq:extTormap}
\widehat{\mathcal{M}}_g \hookrightarrow \widehat{\mathcal{A}}_g
\end{equation}
called the \textit{extended Torelli map}. One has a commutative diagram
\[
\begin{tikzcd}
\widehat{\mathcal{M}}_g \arrow[hookrightarrow]{r}\arrow{d} &\widehat{\mathcal{A}}_g \arrow{d}\\
\mathcal{M}_g \arrow[hookrightarrow]{r} &\mathcal{A}_g
\end{tikzcd}
\]
where the bottom horizontal map is the standard Torelli map.

More generally, the \textit{extended relative degree-$(g-1)$ Picard bundle} $\widehat{\mathcal{P}}_g$ is the moduli space of tuples $(C, P, t, \mathscr{L} , \varphi)$
where $(C, P, t)$ is a coordinatized curve of genus $g$, $\mathscr{L}$ is a line bundle of degree $g-1$ on $C$, and $\varphi$ is a trivialization of the restriction of $\mathscr{L}$ to the disk $D_P$ up to a nonzero multiplicative constant. 
One has a forgetful map $\widehat{\mathcal{P}}_g \rightarrow \mathcal{P}_g$ to the relative degree-$(g-1)$ Picard bundle 
$\mathcal{P}_g\rightarrow \mathcal{M}_g$
 given by $(C, P, t, \mathscr{L} , \varphi)\mapsto (C,\mathscr{L})$.
From \cite{adc91}, one has a natural inclusion 
\[
\widehat{\mathcal{P}}_g \hookrightarrow \widehat{\mathcal{X}}_g
\]
fitting into the commutative diagram in Figure \ref{fig:Pg}. This map is built using the fact that for a curve $C$, a quadratic form $q$ satisfying \eqref{eq:quad}
is equivalent to a theta-characteristic $\eta$ on $C$ (that is, a line bundle whose square is the canonical bundle of $C$) by Riemann's theorem, and a theta-characteristic $\eta$ determines an isomorphism
\[
\mathrm{Pic}^{g-1}(C)\rightarrow \mathrm{Pic}^{0}(C), \qquad \mathscr{L}\mapsto \mathscr{L}\otimes \eta^{-1}.
\]
Moreover, a trivialization $\varphi$ of $\mathscr{L}$ induces the trivialization $\varphi\,dt^{-1/2}$, up to a nonzero multiplicative constant. The isomorphism class of $(\mathscr{L}\otimes \eta^{-1}, \varphi\,dt^{-1/2})$  determines a unique element $\overline{h}\in H'/K$.  A different choice of the theta-characteristic $\eta$ is a change of $\eta$ by an element of the group $G$ of $2$-torsion points in $H'/K$ and moves $\overline{h}$ and $q$ accordingly by an element of $G$. Thus $(\mathscr{L}, \varphi)$ determines a unique element of \eqref{H'KtimesGQ}.

\begin{figure}[t]
\[
\begin{tikzcd}
&\widehat{\mathcal{X}}_g \arrow{rr}\arrow{dd} &&\widehat{\mathcal{A}}_g \arrow{dd}\\
\widehat{\mathcal{P}}_g \arrow[crossing over]{rr}\arrow{dd}\arrow[hookrightarrow]{ru} &&\widehat{\mathcal{M}}_g \arrow[hookrightarrow]{ru}\\
&\mathcal{X}_g \arrow{rr} &&\mathcal{A}_g\\
\mathcal{P}_g \arrow{rr}\arrow[hookrightarrow]{ru} &&\mathcal{M}_g.\arrow[hookrightarrow]{ru}\arrow[leftarrow, crossing over]{uu}
\end{tikzcd}
\]
\caption{The relative Picard variety of degree $g-1$ and related maps}
\label{fig:Pg}
\end{figure}

\subsection{Uniformizations}
Here we review the uniformizations of the moduli spaces at play.
For the curve case, one needs the following Lie algebras.
The $\mathrm{Witt}$ Lie algebra is $\mathrm{Witt}=\mathbb{C}((t))\partial_t$. This is topologically generated by $L_p:=-t^{p+1}\partial_t$ with relations $[L_p, L_q]=(p-q)L_{p+q}$ for $p,q\in\mathbb{Z}$. Also, the Virasoro Lie algebra $\mathrm{Vir}$ is the universal central extension 
\begin{equation}
\label{eq:Vir}
0\rightarrow \mathbb{C}\,\bm{1}\rightarrow\mathrm{Vir}\rightarrow\mathrm{Witt} \rightarrow 0
\end{equation}
with Lie bracket $[\bm{1}, L_p]=0$ and 
\[
[L_p, L_q]=(p-q)L_{p+q} +\frac{1}{12}\left(p^3-p\right)\delta_{p,-q} \,\bm{1}
\]
where $\delta_{p,-q}=1$ when $p=-q$, and $\delta_{p,-q}=0$ otherwise.

Let $\Lambda$ be the Hodge line bundle on $\mathcal{A}_g$ and $\Xi:=\mathscr{O}_{\mathcal{X}_g}(2\theta)$. We omit notation for the pullback via $\widehat{\mathcal{M}}_g\rightarrow\mathcal{M}_g\rightarrow\mathcal{A}_g$ and $\widehat{\mathcal{P}}_g\rightarrow\mathcal{P}_g\rightarrow\mathcal{X}_g$.

\begin{theorem}[\cite{adkp, besh, kvir}]
\label{thm:WittVirunif}
\begin{enumerate}[(i)]
\item The Lie algebra $\mathrm{Vir}$ acts on the line bundle $\Lambda$ on $\widehat{\mathcal{M}}_g$ by first-order differential operators extending a transitive action of $\mathrm{Witt}$ on  $\widehat{\mathcal{M}}_g$ and with the central element $\bm{1} \in\mathrm{Vir}$ acting as multiplication by $2$ on the fibers of $\Lambda\rightarrow \widehat{\mathcal{M}}_g$.

\smallskip

\item There exists a central extension $\mathfrak{D}$ of $\mathrm{Witt}\ltimes H'$ acting on the line bundle $\Xi$ on $\widehat{\mathcal{P}}_g$ by first-order differential operators. This action extends a transitive action of $\mathrm{Witt}\ltimes H'$ on  $\widehat{\mathcal{P}}_g$, and the central element $\bm{1}\in \mathfrak{D}$ acts as multiplication by $2$ on the fibers of $\Xi\rightarrow \widehat{\mathcal{P}}_g$.
\end{enumerate}
\end{theorem}

We refer to \cite[(4.7)]{adc91} for the exact description of $\mathfrak{D}$ (see also~\mbox{\cite[\S 1.7]{avva}}).

The uniformizations of Theorem \ref{thm:WittVirunif} extend beyond the Torelli locus in the following way.
The $\mathrm{Witt}$ algebra is a Lie subalgebra of the \textit{symplectic  algebra} $\mathfrak{sp}\left(H' \right)$ defined as
\begin{equation}
\label{eq:sp}
\mathfrak{sp}\left(H' \right) := \left\{ X\in \mathfrak{gl}(H') \, | \, \langle X a, b \rangle + \langle a,X b \rangle =0, \mbox{ for all } a,b\in H' \right\}.
\end{equation}
Here $\mathfrak{gl}(H')$ is the Lie algebra of continuous endomorphisms of $H'$.
For $(Z, F, L)$ in $\widehat{\mathcal{A}}_g$, its isotropy algebra will be determined by
\begin{equation}
\label{eq:spF}
\mathfrak{sp}_{F} \left( H' \right):= \left\{ X\in \mathfrak{sp}  \left( H' \right) \,\, | \,\, X\left( F^\perp \right) \subseteq F  \right\}.
\end{equation}
This is the Lie subalgebra of $\mathfrak{sp}\left( H'\right)$ topologically generated by $fb_i$ with $f\in F$ and $i\in\mathbb{Z}\setminus\{0\}$.  For $x=(Z, F, L, \overline{h}, q)$ in $\widehat{\mathcal{X}}_g$, select $h\in \overline{h}+K$ and let
\begin{equation}
\label{eq:ix}
\mathfrak{i}(x):= \{X-X(h)+a\in \mathfrak{sp}  \left( H' \right) \ltimes H' \, | \, X\in  \mathfrak{sp}_{F} \left( H' \right) \mbox{ and } a\in F \}.
\end{equation}
The choice of $h\in \overline{h}+K$ determines an isomorphism 
\[
 \mathfrak{sp}_{F} \left( H' \right) \ltimes F
 \xrightarrow{\cong}
\mathfrak{i}(x), \qquad
 (X,a)\mapsto X-X(h)+a
\]
but $\mathfrak{i}(x)$ is independent of this choice.

\begin{theorem}[\cite{adc91}]
\label{thm:spspHunif}
\begin{enumerate}[(i)]
\item The Lie algebra \sloppy\mbox{$\mathfrak{sp} \left( H' \right)$} acts transitively on  $\widehat{\mathcal{A}}_g$, and $\mathrm{Witt}\hookrightarrow \mathfrak{sp} \left( H' \right)$ is the differential of $\widehat{\mathcal{M}}_g\hookrightarrow\widehat{\mathcal{A}}_g$.
For $a=(Z, F, L)$ in $\widehat{\mathcal{A}}_g$, one has
\[
T_{a} \left( \widehat{\mathcal{A}}_g \right) \cong \mathfrak{sp}  \left( H' \right) \,/\, \mathfrak{sp}_{F} \left( H' \right).
\]

\smallskip

\item The Lie algebra \sloppy\mbox{$\mathfrak{sp} \left( H' \right)\ltimes H'$} acts transitively on  $\widehat{\mathcal{X}}_g$, and for $x=(Z, F, L, \overline{h}, q)$ in $\widehat{\mathcal{X}}_g$, one has
\[
T_{x} \left( \widehat{\mathcal{X}}_g \right) \cong \left(\mathfrak{sp}  \left( H' \right) \ltimes H' \right)/\,  \mathfrak{i}(x).
\]
\end{enumerate}
\end{theorem}

Furthermore, the uniformizations of Theorem \ref{thm:spspHunif} extend to suitable line bundles as follows.
First we define the required Lie algebra extensions.
Let $\mathscr{H}$ be the \textit{Weyl algebra}. This is the completion of the universal enveloping algebra of the Heisenberg algebra $H$ with respect to the topology in which the basis of open neighborhoods of $0$ is formed by the left ideals 
generated by $t^N\mathbb{C}\llbracket t \rrbracket$, for $N\in \mathbb{Z}$,
modulo the two-sided ideal generated by $(1-\bm{1})$. Here $1$ is the multiplicative identity element, and $\bm{1}$ is the central element of $H$. Elements of $\mathscr{H}$ are (possibly infinite) series of the form 
\[
C^0 + \sum_{i\geq 1} C^i \,b_i
\]
with $C^i$ for $i \geq 0$ equal to a finite linear combination of finite tensor products of elements $b_j$ for $j\in\mathbb{Z}$.

The Weyl algebra has a natural filtration by subspaces \mbox{$\mathscr{H}_0\subset \mathscr{H}_1 \subset \cdots$,} where $\mathscr{H}_n$ is the subspace of tensors of degree at most $n$. E.g, 
\[
\mathscr{H}_0\cong\mathbb{C}\,\bm{1}, \qquad \mathscr{H}_1/\mathscr{H}_0\cong H', \qquad \mathscr{H}_2/\mathscr{H}_1\cong \mathfrak{sp} \left( H' \right).
\]
One has \sloppy{$[\mathscr{H}_n, \mathscr{H}_m]\subseteq \mathscr{H}_{n+m-2}$.} In particular, $\mathscr{H}_2$ is a Lie subalgebra of $\mathscr{H}$.
This is a central extension 
\begin{equation}
\label{eq:H2}
0 \rightarrow \mathbb{C}\,\bm{1}\rightarrow \mathscr{H}_2 \rightarrow \mathfrak{sp} \left( H' \right)\ltimes H' \rightarrow 0.
\end{equation}
Restricting over $\mathfrak{sp} \left( H' \right)$, one has a universal central extension
\begin{equation}
\label{eq:mp}
0 \rightarrow \mathbb{C}\,\bm{1}\rightarrow \mathfrak{mp} \left( H' \right) \rightarrow \mathfrak{sp} \left( H' \right)\rightarrow 0
\end{equation}
which restricts over $\mathrm{Witt}\subset  \mathfrak{sp} \left( H' \right)$ to the Virasoro extension $\mathrm{Vir}$ from \eqref{eq:Vir}. The universality follows from $H^2(\mathfrak{sp} \left( H' \right),\mathbb{C})\cong \mathbb{C}$ \cite[Prop.~5.1]{avva}.

We omit notation for the pullback of $\Lambda$ via $\widehat{\mathcal{A}}_g\rightarrow\mathcal{A}_g$ and of 
$\Xi$ via \mbox{$\widehat{\mathcal{X}}_g\rightarrow\mathcal{X}_g$.}

\begin{theorem}[\cite{avva}]
\begin{enumerate}[(i)]
\item The Lie algebra $\mathfrak{mp} \left( H' \right)$ acts on the line bundle $\Lambda$ on $\widehat{\mathcal{A}}_g$ by first-order differential operators extending the transitive action of \sloppy\mbox{$\mathfrak{sp} \left( H' \right)$} on  $\widehat{\mathcal{A}}_g$ and with the central element $\bm{1}$ in $\mathfrak{mp} \left( H' \right)$ acting as multiplication by $2$ on the fibers of $\Lambda\rightarrow \widehat{\mathcal{A}}_g$.

\smallskip

\item The Lie algebra $\mathscr{H}_{2}(H)$ acts on the line bundle $\Xi$ on $\widehat{\mathcal{X}}_g$ by first-order differential operators  extending the transitive action of \sloppy\mbox{$\mathfrak{sp} \left( H' \right)\ltimes H'$} on  $\widehat{\mathcal{X}}_g$ and with the central element $\bm{1}\in \mathscr{H}_2$ acting as multiplication by $2$ on the fibers of $\Xi\rightarrow \widehat{\mathcal{X}}_g$.
\end{enumerate}
\end{theorem}

\begin{figure}[t]
\begin{tikzcd}[column sep=0.5em]
&&&\mathscr{O}_{\widehat{\mathcal{A}}_g} \arrow[hookrightarrow]{dd} \arrow[rightarrow]{rr}[near start]{2} && \mathscr{O}_{\widehat{\mathcal{A}}_g} \arrow[hookrightarrow]{dd}\\
&&\mathscr{O}_{\widehat{\mathcal{M}}_g} \arrow[hookrightarrow]{ur} \arrow[rightarrow, crossing over]{rr}[near start]{2} \arrow[hook]{dd}&&\mathscr{O}_{\widehat{\mathcal{M}}_g} \arrow[hookrightarrow]{ur} \arrow[hookrightarrow]{dd}&\\
&\mathfrak{sp}^{\mathrm{out}} \left( H'\right) \arrow[rightarrow]{dd}[above, near start, rotate=-90]{=} \arrow[hookrightarrow]{rr} && \mathfrak{mp} \left( H' \right)\widehat{\otimes}_\mathbb{C}\, \mathscr{O}_{\widehat{\mathcal{A}}_g} \arrow[->>]{dd} \arrow[->>]{rr} && \mathscr{F}_{\Lambda} \arrow[->>]{dd}\\
\mathrm{Ker}(\alpha) \arrow[hookrightarrow]{ur} \arrow[rightarrow]{dd}[above, near start, rotate=-90]{=} \arrow[hookrightarrow, crossing over]{rr}&& \mathrm{Vir} \,\widehat{\otimes}_\mathbb{C}\, \mathscr{O}_{\widehat{\mathcal{M}}_g} \arrow[hookrightarrow]{ur} \arrow[crossing over,dash,shorten <= 3mm,shorten >= 3mm]{uu} \arrow[->>, crossing over]{rr} && \mathscr{F}_{\Lambda|\widehat{\mathcal{M}}_g} \arrow[hookrightarrow]{ur} \arrow[crossing over,dash,shorten <= 3mm,shorten >= 3mm]{uu}&\\
&\mathfrak{sp}^{\mathrm{out}} \left( H' \right) \arrow[hookrightarrow]{rr} && \mathfrak{sp} \left( H' \right) \widehat{\otimes}_\mathbb{C}\,\mathscr{O}_{\widehat{\mathcal{A}}_g} \arrow[->>]{rr} && \mathscr{T}_{\widehat{\mathcal{A}}_g}\\
\mathrm{Ker}(\alpha) \arrow[hookrightarrow]{ur} \arrow[hookrightarrow]{rr} && \mathrm{Witt} \,\widehat{\otimes}_\mathbb{C}\,\mathscr{O}_{\widehat{\mathcal{M}}_g} \arrow[hookrightarrow]{ur} \arrow[<<-, crossing over]{uu} \arrow[->>]{rr}{\alpha} &&\mathscr{T}_{\widehat{\mathcal{M}}_g}. \arrow[hookrightarrow]{ur} \arrow[<<-, crossing over]{uu}&
\end{tikzcd}
\caption{The Atiyah algebra of $\Lambda$ on $\widehat{\mathcal{M}}_g$ (front) and $\widehat{\mathcal{A}}_g$ (back) with corresponding exact sequences.}
\label{fig:bigAtiyahLambdadiag}
\end{figure}

\begin{figure}[t]
\begin{adjustbox}{max width=\textwidth}
\begin{tikzcd}
&&&\mathscr{O}_{\widehat{\mathcal{X}}_g} \arrow[hookrightarrow]{dd} \arrow[rightarrow]{rr}[near start]{2} && \mathscr{O}_{\widehat{\mathcal{X}}_g} \arrow[hookrightarrow]{dd}\\
&&\mathscr{O}_{\widehat{\mathcal{P}}_g} \arrow[hookrightarrow]{ur} \arrow[rightarrow, crossing over]{rr}[near start]{2} \arrow[hook]{dd}&&\mathscr{O}_{\widehat{\mathcal{P}}_g} \arrow[hookrightarrow]{ur} \arrow[hookrightarrow]{dd}&\\
&\mathfrak{i} \arrow[rightarrow]{dd}[above, near start, rotate=-90]{=} \arrow[hookrightarrow]{rr} && \mathscr{H}_2 \,\widehat{\otimes}_\mathbb{C}\, \mathscr{O}_{\widehat{\mathcal{X}}_g} \arrow[->>]{dd} \arrow[->>]{rr} && \mathscr{F}_{\Xi} \arrow[->>]{dd}\\
\mathrm{Ker}(\beta) \arrow[hookrightarrow]{ur} \arrow[rightarrow]{dd}[above, near start, rotate=-90]{=} \arrow[hookrightarrow, crossing over]{rr}&& \mathfrak{D} \,\widehat{\otimes}_\mathbb{C}\, \mathscr{O}_{\widehat{\mathcal{P}}_g}\arrow[hookrightarrow]{ur} \arrow[crossing over,dash,shorten <= 3mm,shorten >= 3mm]{uu} \arrow[->>, crossing over]{rr} && \mathscr{F}_{\Xi|\widehat{\mathcal{P}}_g} \arrow[hookrightarrow]{ur} \arrow[crossing over,dash,shorten <= 3mm,shorten >= 3mm]{uu}&\\
&\mathfrak{i} \arrow[hookrightarrow]{rr} && \left(\mathfrak{sp} \left( H' \right) \ltimes H'\right)\widehat{\otimes}_\mathbb{C}\, \mathscr{O}_{\widehat{\mathcal{X}}_g} \arrow[->>]{rr} && \mathscr{T}_{\widehat{\mathcal{X}}_g}\\
\mathrm{Ker}(\beta) \arrow[hookrightarrow]{ur} \arrow[hookrightarrow]{rr} && \left(\mathrm{Witt} \ltimes H'\right) \widehat{\otimes}_\mathbb{C}\,\mathscr{O}_{\widehat{\mathcal{P}}_g} \arrow[hookrightarrow]{ur} \arrow[<<-, crossing over]{uu} \arrow[->>]{rr}{\beta} &&\mathscr{T}_{\widehat{\mathcal{P}}_g}. \arrow[hookrightarrow]{ur} \arrow[<<-, crossing over]{uu}&
\end{tikzcd}
\end{adjustbox}
\caption{The Atiyah algebra of $\Xi$ on $\widehat{\mathcal{P}}_g$ (front) and $\widehat{\mathcal{X}}_g$ (back) with corresponding exact sequences.}
\label{fig:bigAtiyahXidiag}
\end{figure}

The uniformizations reviewed here are summarized by the commutative diagrams 
in Figures \ref{fig:bigAtiyahLambdadiag} and \ref{fig:bigAtiyahXidiag} 
involving the Atiyah algebras $\mathscr{F}_\Lambda$ and $\mathscr{F}_\Xi$ of the line bundles $\Lambda$ and $\Xi$, respectively.
We use there the following notation:
\begin{equation}
\label{eq:spout}
\mathfrak{sp}^{\mathrm{out}} \left( H'\right) \subset \mathfrak{sp}\left( H'\right) \,\widehat{\otimes}_\mathbb{C}\, \mathscr{O}_{\widehat{\mathcal{A}}_g}
\end{equation}
is the \textit{sheaf of isotropy Lie subalgebras} whose fiber at $(Z,F,L)$ in $\widehat{\mathcal{A}}_g$ is $\mathfrak{sp}_F\left( H'\right)$. This sheaf can be constructed as follows.
Let 
\begin{equation}
\label{eq:universalF}
\mathbb{F}\rightarrow\widehat{\mathcal{A}}_g
\end{equation}
be the tautological sheaf whose fiber at $(Z,F,L)$ is $F$. Then
\begin{equation}
\label{eq:spoutrel}
\mathfrak{sp}^{\mathrm{out}} \left( H'\right) = \mathrm{Im}\left(\mathbb{F} \,\widehat{\otimes}_{\mathbb{C}}\, H' \longrightarrow
\mathfrak{sp}(H')\,\widehat\otimes_{\mathbb C}\,\mathscr O_{\widehat{\mathcal A}_g},
\qquad f\otimes\ell\mapsto f\ell
\right).
\end{equation}
From \cite[Prop.~5.4]{avva}, one has a unique continuous splitting
\[
\mathfrak{sp}^{\mathrm{out}} \left( H'\right) \hookrightarrow \mathfrak{mp} \left( H' \right)\widehat{\otimes}_\mathbb{C}\, \mathscr{O}_{\widehat{\mathcal{A}}_g}, \qquad f\ell \mapsto f\ell
\]
for a local section $f\in\mathbb{F}$ and $\ell\in H'$.
Here the target is the product with the $F$-factor on the left.
Similarly,  
\[
\mathfrak{i} \subset \left(\mathfrak{sp} \left( H' \right) \ltimes H'\right)\widehat{\otimes}_\mathbb{C}\, \mathscr{O}_{\widehat{\mathcal{X}}_g}
\]
is the \textit{sheaf of isotropy Lie subalgebras} whose fiber at $x\in\widehat{\mathcal{X}}_g$ is $\mathfrak{i}(x)$ from \eqref{eq:ix}, and by \cite[Prop.~5.4]{avva} one has a unique continuous splitting 
\begin{align}
\begin{split}
\label{eq:isplitting}
\mathfrak{i} &\hookrightarrow \mathscr{H}_2 \,\widehat{\otimes}_\mathbb{C}\, \mathscr{O}_{\widehat{\mathcal{X}}_g},\\
f\ell - (f\ell)(h) &\mapsto \left(f-\langle f,h\rangle\bm{1} \right)\left(\ell-\langle \ell,h\rangle\bm{1} \right),\\
f &\mapsto f-\langle f,h\rangle\bm{1},
\end{split}
\end{align}
for $f\in \mathbb{F}$, $\ell\in H'$, and some $h\in \overline{h}+K$ (this is independent of the choice of such $h$).
Equivalently,
\begin{align*}
X - X(h) &\mapsto X - X(h) + \frac{1}{2}\langle X(h),h\rangle\bm{1},\\
f &\mapsto f-\langle f,h\rangle\bm{1},
\end{align*}
for $X\in \mathfrak{sp}^{\mathrm{out}} \left( H'\right)$ and $f\in \mathbb{F}$.

\section{Heisenberg vertex algebras}
\label{sec:HVOA}

After reviewing Heisenberg vertex algebras following \cite[\S 3.6]{KacBeginners} and \cite[\S\S 2, 5.4.1]{bzf}, we define an action of the metaplectic algebra on them.

\subsection{Heisenberg algebras}
Let $\Gamma$ be a lattice, that is, a free abelian group equipped with a symmetric bilinear form $(,) \colon \Gamma\times \Gamma\rightarrow \mathbb{Z}$. Assume that $\Gamma$ has {finite rank} and is non-degenerate.

Let $\mathfrak{h}_\Gamma:=\Gamma\otimes_{\mathbb{Z}}\mathbb{C}$ and extend the bilinear form of $\Gamma$ to $\mathfrak{h}_\Gamma$ linearly.
Consider the loop space $\mathfrak{h}_\Gamma((t)):= \mathfrak{h}_\Gamma\otimes_{\mathbb{C}}\mathbb{C}((t))$ with trivial Lie bracket. 
The \textit{Heisenberg algebra} $\widehat{\mathfrak{h}}_\Gamma$ is the central extension 
\[
0\rightarrow \mathbb{C}\bm{1} \rightarrow \widehat{\mathfrak{h}}_\Gamma \rightarrow \mathfrak{h}_\Gamma((t)) \rightarrow 0
\]
defined by the Lie bracket 
\[
\left[ A\otimes f(t), B\otimes g(t)\right] = -(A,B) \left(\mathop{\mathrm{Res}}_{t=0} f dg\right) \bm{1}
\]
for $A\otimes f(t), B\otimes g(t)\in \mathfrak{h}_\Gamma((t))$. Explicitly, for $A_n:= A\otimes t^n$ and $B_m:= B\otimes t^m$ in $\mathfrak{h}_\Gamma((t))$, one has
\[
\left[A_n, B_m\right] = (A,B)\,n\,\delta_{n,-m}\,\bm{1}
\]
where $\delta_{n,-m}=1$ when $n=-m$, and $\delta_{n,-m}=0$ otherwise.

\subsection{Weyl algebras}
\label{sec:Weyl}
Let $\mathscr{H}_\Gamma$ be the \textit{Weyl algebra} assigned to the lattice $\Gamma$. That is, $\mathscr{H}_\Gamma$ is the completion of the universal enveloping algebra of $\widehat{\mathfrak{h}}_\Gamma$ with respect to the topology in which the basis of open neighborhoods of $0$ is formed by the left ideals generated by
\[
\mathfrak{h}_\Gamma\,\widehat{\otimes}_{\mathbb{C}}\, t^N\mathbb{C}\llbracket t \rrbracket
 \qquad \mbox{for $N\in \mathbb{Z}$}
\]
modulo the two-sided ideal generated by $(1-\bm{1})$. Here $1$ is the multiplicative identity element, and $\bm{1}$ is the central element of $\widehat{\mathfrak{h}}_\Gamma$.
Elements of $\mathscr{H}_\Gamma$ are (possibly infinite) series of the form 
\[
C^0 + \sum_{i\geq 1} C^i A^i \qquad \mbox{with} \qquad
A^i=B^i\otimes t^i\in \mathfrak{h}_\Gamma\otimes_{\mathbb{C}}\mathbb{C}[ t]
\]
and $C^i$ for $i \geq 0$ equal to a finite linear combination of finite tensor products of elements in $\mathfrak{h}_\Gamma\otimes_{\mathbb{C}}\mathbb{C}[ t, t^{-1}]$.

\subsection{Heisenberg vertex operator algebras}
\label{sec:Hvoa}
Let $\mathscr{H}_\Gamma^+$ (respectively, $\mathscr{H}_\Gamma^-$) be the commutative subalgebra of $\mathscr{H}_\Gamma$ generated by $\mathfrak{h}_\Gamma\otimes_{\mathbb{C}}\mathbb{C}\llbracket t\rrbracket$ (resp., \mbox{$\mathfrak{h}_\Gamma\otimes_{\mathbb{C}}t^{-1}\mathbb{C}[ t^{-1}]$}).
The \textit{Fock representation} $\pi_\Gamma$ of $\mathscr{H}_\Gamma$ is defined as
\[
\pi_\Gamma := \mathscr{H}_\Gamma \otimes_{\mathscr{H}_\Gamma^+} \mathbb{C}\bm{1}
\]
where $\mathscr{H}_\Gamma^+$ acts trivially on $\mathbb{C}\bm{1}$ --- that is, elements of $\mathfrak{h}_\Gamma\otimes_{\mathbb{C}}\mathbb{C}\llbracket t\rrbracket$ act by $0$ and ${1}$ acts by the identity.
 By the Poincar\'e--Birkhoff--Witt theorem, one has an isomorphism of vector spaces $\mathscr{H}_\Gamma\cong \mathscr{H}_\Gamma^-\otimes_{\mathbb{C}}\mathscr{H}_\Gamma^+$ and thus 
\[
\pi_\Gamma\cong \mathscr{H}_\Gamma^- \qquad \mbox{(as vector spaces).}
\]
The standard gradation on $\pi_\Gamma$ is the one induced by $\deg(A_n)=-n$.

The Fock representation $\pi_\Gamma$ carries a vertex algebra structure. The vertex operator of an element of type $A_{-1}\,\bm{1}$ in $\pi_\Gamma$ is 
\[
A(t):=\sum_{i\in\mathbb{Z}} A_i \, t^{-i-1}.
\]
By the Strong Reconstruction Theorem \cite[\S 4.4.1]{bzf}, this extends uniquely to a vertex algebra structure on $\pi_\Gamma$.
The resulting vertex operator of an arbitrary elements of $\pi_\Gamma$
\[
A^{i_1}\cdots A^{i_k}\,\bm{1} \in \pi_\Gamma \qquad \mbox{with}\quad A^{i_j}=B^{i_j}\otimes t^{i_j}\in \mathfrak{h}_\Gamma\otimes_{\mathbb{C}}t^{-1}\mathbb{C}\left[ t^{-1}\right]
\] 
is
\[
\frac{1}{(-i_1-1)!\dots(-i_k-1)!} :\! \partial_t^{-i_1-1} B^{i_1}(t)\cdots \partial_t^{-i_k-1} B^{i_k}(t) \!:
\]
where $:\,\, :$ is the usual normal ordering.

Moreover, there is a family of vertex operator algebra structures on $\pi_\Gamma$.
To see this, let $\{H^1, \dots, H^d\}$ be an orthonormal basis of $\mathfrak{h}_\Gamma$ with $d=\mathrm{rank}(\Gamma)$.
Then for $A\in\mathfrak{h}_\Gamma$, the vector
\[
\omega^A:= \frac{1}{2}\sum_{i=1}^d H^i_{-1}H^i_{-1}\,\bm{1} + A_{-2}\,\bm{1} 
\]
is a conformal vector with central charge $c^A:= d-12(A,A)$ endowing $\pi_\Gamma$  with a vertex operator algebra structure \cite[\S3]{chu2024classification}. 
The induced action of the Virasoro algebra is given by
\begin{equation}
\label{eq:Viraction}
\mathrm{Vir}\rightarrow \mathrm{End}(\pi_\Gamma), \qquad \bm{1}\mapsto c^A \,\mathrm{id}_{\pi_\Gamma},
\quad L_p \mapsto \frac{1}{2}\sum_{n\in \mathbb{Z}}\sum_{i=1}^d :\! H^i_n H^i_{p-n} \!: - (p+1)A_{p}.
\end{equation}
We will refer to $\omega^A$ with $A=0$ as the \textit{standard} conformal vector.

Even more conformal vectors could be considered. See for instance
\cite{matsuo1999note}, where all conformal vectors of $\pi_\Gamma$ are classified in case $\Gamma$ has rank one. The conformal vectors of type $\omega^A$ are the \textit{unique} ones that preserve the standard gradation of $\pi_\Gamma$, see \cite[Rmk 3.1]{matsuo1999note}, \cite[Lemma 3.1]{chu2024classification}.

We will also consider the \textit{Heisenberg vertex algebra of level $0$}
\[
\pi_\Gamma^0 := \mathscr{H}^0_\Gamma \otimes_{\mathscr{H}_\Gamma^{0,+}} \mathbb{C}\bm{1}
\]
where $\mathscr H^0_\Gamma$ is 
defined as $\mathscr H_\Gamma$ in \S\ref{sec:Weyl}, 
but modulo the two-sided ideal generated by $\mathbf 1$ in place of $(1-\mathbf 1)$, and elements of
$\mathfrak h_\Gamma\otimes_{\mathbb C}\mathbb C\llbracket t\rrbracket$ act on $\mathbb C\bm{1}$
by $0$. 
This is the \textit{commutative} vertex algebra corresponding to the commutative algebra $\mathscr{H}_\Gamma^-$ \cite[\S 2.3.9]{bzf}.
That is, one has
\begin{equation}
\label{eq:piGamma0}
\pi_\Gamma^0 
\cong \mathrm{Sym}\left( \mathfrak{h}_\Gamma\otimes_{\mathbb{C}}t^{-1}\mathbb{C}[ t^{-1}] \right)
 \qquad \mbox{(as vertex algebras).}
\end{equation}
When $V=\pi^0_\Gamma$, the algebra $\mathscr H_\Gamma$ is to be read as
$\mathscr H^0_\Gamma$ throughout.
The vertex algebra $\pi_\Gamma^0$ has no conformal vector. Instead, $\pi_\Gamma^0$ is a quasi-conformal vertex algebra in the sense of \cite[\S6.3.4]{bzf}. 

One could define similarly the Heisenberg vertex algebra $\pi_\Gamma^\ell$ of level $\ell$ for arbitrary $\ell\in\mathbb{C}$. 
Thus the vertex algebra $\pi_\Gamma=\pi_\Gamma^1$ can be seen as a deformation of the commutative algebra $\pi_\Gamma^0$.
On the other hand, one has $\pi_\Gamma^\ell\cong \pi_\Gamma$ when $\ell\neq 0$ \cite[\S2.4.3]{bzf}. Thus we will only consider $\pi_\Gamma$ and $\pi_\Gamma^0$.

\subsection{Metaplectic action}
\label{sec:mpaction}
Let $V$ be a Heisenberg vertex algebra of non-zero level, i.e., $V=\pi_\Gamma$ for some lattice $\Gamma$ as above.
Let $\{H^1, \dots, H^d\}$ be an orthonormal basis of $\mathfrak{h}_\Gamma$, and select $N\in \mathfrak{h}_\Gamma$ such that $(N,N)=d$. Recall the Lie algebras $\mathscr{H}_2$
from \eqref{eq:H2} and $\mathfrak{mp}\left(H' \right)$ from \eqref{eq:mp}.

\begin{lemma}
\label{lemma:mpaction}
\begin{enumerate}[(i)]
\item The Lie algebra $\mathfrak{mp}\left(H' \right)$ acts on $V$ via the map:
\[
\mathfrak{mp}\left(H' \right) \rightarrow \mathrm{End}(V), \qquad \bm{1}\mapsto d\,\mathrm{id}_V, \quad
:\!b_mb_n\!:\, \mapsto \sum_{i=1}^d :\!H^i_mH^i_n\!:.
\]
This action is independent of the choice of orthonormal basis $\{H^i\}_i$.

\smallskip

\item This $\mathfrak{mp}\left(H' \right)$-action restricts to an action of the Virasoro algebra $\mathrm{Vir}$ on $V$ of central charge $d$:
\[
\mathrm{Vir}\rightarrow \mathrm{End}(V), \qquad \bm{1}\mapsto d \,\mathrm{id}_{V},
\quad L_p \mapsto \frac{1}{2}\sum_{n\in \mathbb{Z}}\sum_{i=1}^d :\! H^i_n H^i_{p-n} \!:.
\]

\smallskip

\item This Virasoro action coincides with the one induced by the standard conformal vector of $V$.

\smallskip

\item The Lie algebra $\mathscr{H}_2$ acts on $V$ via the map
\[
\mathscr{H}_2 \rightarrow \mathrm{End}(V), \quad 
\bm{1}\mapsto d\,\mathrm{id}_V, \quad
:\! b_mb_n\!:\, \mapsto \sum_{i=1}^d :\! H^i_mH^i_n \!:, \quad
b_k \mapsto N_k.
\]
This action is independent of the choice of orthonormal basis $\{H^i\}_i$ but depends on the choice of $N\in \mathfrak{h}_\Gamma$ with $(N,N)=d$.

\end{enumerate}
\end{lemma}

\begin{proof}
Parts (i)--(ii) and (iv) follow from a direct computation. Part (iii) follows by comparison with \eqref{eq:Viraction}.
\end{proof}

\section{Spaces of coinvariants at extended PPAVs}
\label{sec:spaces}

After reviewing the case of coordinatized curves, we construct here spaces of coinvariants associated to extended PPAVs. In the case of extended Jacobians, we show how these recover the spaces associated to the corresponding coordinatized curves.
Fix throughout a lattice $\Gamma$ of finite rank with non-degenerate bilinear form and consider the vector space $\mathfrak{h}_\Gamma$ and Weyl algebra $\mathscr{H}_\Gamma$ as in \S\ref{sec:HVOA}. 

\subsection{The Lie algebra of outgoing states: the curve case}
For a coordinatized curve $(C,P,t)$, 
let $\mathscr{O}(C\setminus P)$ be the space of regular functions on $C\setminus P$
and define the \textit{Lie algebra of outgoing states} associated to $(C,P,t)$ as
\begin{equation}
\label{eq:houtCPt}
\mathfrak{h}_\Gamma^{\mathrm{out}}(C,P,t) := \mathfrak{h}_\Gamma \otimes_{\mathbb{C}} \mathscr{O}(C\setminus P).
\end{equation}
One has an inclusion of commutative Lie algebras \mbox{$\mathscr{O}(C\setminus P)\hookrightarrow \mathbb{C}((t))$} given by the Laurent series expansion at the point $P$ in the coordinate $t$. 
From the residue theorem, one has 
\begin{equation}
\label{eq:isotropicF}
\mathop{\mathrm{Res}}_{t=0} f dh=0 \qquad \mbox{for $f,h\in \mathscr{O}(C\setminus P)$.}
\end{equation}
Thus the induced inclusion $\mathfrak{h}_\Gamma^{\mathrm{out}}(C,P,t) \hookrightarrow \mathscr{H}_\Gamma$ realizes 
 $\mathfrak{h}_\Gamma^{\mathrm{out}}(C,P,t)$ as a (commutative) Lie subalgebra of $\mathscr{H}_\Gamma$.

\subsection{The Lie algebra of outgoing states: the PPAV case}
For an extended PPAV $a=(Z,F,L)$, define the \textit{Lie algebra of outgoing states} associated to $a$ as
\[
\mathfrak{h}_\Gamma^{\mathrm{out}}(a) := \mathfrak{h}_\Gamma \otimes_{\mathbb{C}} F.
\]
The assumption that $F$ is isotropic yields a vanishing as in \eqref{eq:isotropicF} and thus the natural inclusion 
\[
\mathfrak{h}_\Gamma^{\mathrm{out}}(a) \hookrightarrow \mathscr{H}_\Gamma,
\]
where elements of $H'$ are lifted to $H$ with zero constant term,
realizes $\mathfrak{h}_\Gamma^{\mathrm{out}}(a)$ as a (commutative) Lie subalgebra of $\mathscr{H}_\Gamma$.
Hence the role of the residue theorem in the curve case is played here by the isotropicity of~$F$.

\subsection{Spaces of coinvariants for extended PPAVs}
\label{sec:VhataV}
Let $V$ be a Heisenberg vertex algebra as in \S\ref{sec:HVOA}, that is, $V$ is either the Fock representation $\pi_\Gamma$ of a Weyl algebra $\mathscr{H}_\Gamma$ or $\pi_\Gamma^0$ of $\mathscr{H}^0_\Gamma$.

For an extended PPAV $a=(Z,F,L)$, the action of $\mathscr{H}_\Gamma$ on $V$ induces an action of its Lie subalgebra $\mathfrak{h}_\Gamma^{\mathrm{out}}(a)$ on $V$.

Define the \textit{space of coinvariants} of $V$ associated to $a$ as the vector space
\begin{equation}
\label{eq:VhataV}
\widehat{\mathbb{V}}(a,V):=V\,/\, \mathfrak{h}_\Gamma^{\mathrm{out}}(a) (V).
\end{equation}
Also, define the \textit{space of conformal blocks} of $V$ associated to $a$ as the dual vector space $\widehat{\mathbb{V}}(a,V)^\vee$. This is the space of $\mathfrak{h}_\Gamma^{\mathrm{out}}(a)$-invariant functionals on~$V$:
\[
\widehat{\mathbb{V}}(a,V)^\vee = \mathrm{Hom}_{\mathfrak{h}_\Gamma^{\mathrm{out}}(a)}\left(V,\mathbb{C} \right).
\]

These definitions are motivated by the analogous ones in the curve case from \cite[\S9.1.7]{bzf}. 
Indeed, the space of coinvariants of $V$ associated to a coordinatized curve $(C,P,t)$ is defined there as
\begin{equation}
\label{eq:VhatCPTtV}
\widehat{\mathbb{V}}(C,P,t,V):=V\,/\, \mathfrak{h}_\Gamma^{\mathrm{out}}(C,P,t) (V).
\end{equation}
The definitions \eqref{eq:VhataV} and \eqref{eq:VhatCPTtV} agree on the extended Torelli locus:

\begin{lemma}
\label{lemma:VhatonTorelli}
When $a=(Z,F,L)$ is the extended Jacobian of a coordinatized curve $(C,P,t)$, one has
\[
\widehat{\mathbb{V}}(a,V) = \widehat{\mathbb{V}}(C,P,t,V).
\]
\end{lemma}

\begin{proof}
One has
\[
\mathfrak{h}_\Gamma^{\mathrm{out}}(C,P,t) =
\mathfrak{h}_\Gamma^{\mathrm{out}}(a) \oplus \mathfrak{h}_\Gamma \otimes_{\mathbb{C}} b_0.
\]
This follows from the fact that the isotropic space $F$ for extended Jacobians is defined as in \eqref{eq:FonTorellilocus}. Since $\mathfrak{h}_\Gamma \otimes_{\mathbb{C}} b_0$ acts trivially on $V$, the statement follows.
\end{proof}

\begin{remark}
It is easy to see that for an extended PPAV $a=(Z,F,L)$ with $g>0$, the space $\widehat{\mathbb{V}}(a,V)$ has infinite dimension. Indeed, $\widehat{\mathbb{V}}(a,V)$ contains linearly independent vectors of type 
\[
(A\otimes h)^{ n}\,\bm{1} \qquad \mbox{for $0\neq A\otimes h \in \mathfrak{h}_\Gamma \otimes_{\mathbb{C}} Z/F$ and $n\geq 0$.}
\]
\end{remark}

\begin{example}
\label{ex:pi0Gamma}
For $V=\pi_\Gamma^0$ and an extended PPAV $a=(Z,F,L)$, applying \eqref{eq:piGamma0} one has
\begin{equation*}
\pi_\Gamma^0 \cong \mathrm{Sym}\left( \mathfrak{h}_\Gamma\otimes_{\mathbb{C}} Z \right)
 \qquad \mbox{(as algebras)}
\end{equation*}
and thus
\begin{equation*}
\widehat{\mathbb{V}}\left(a,\pi_\Gamma^0\right) \cong \mathrm{Sym}\left( \mathfrak{h}_\Gamma\otimes_{\mathbb{C}}Z/F \right).
\end{equation*}
\end{example}

\begin{remark}
In the curve case, spaces of conformal blocks globalize to sheaves which are usually not quasi-coherent on moduli spaces of curves \cite[\S 17.3]{bzf}. On the contrary, sheaves of coinvariants are always quasi-coherent, thus the preference for coinvariants over conformal blocks in \cite{bzf}. For the same reason, we will focus on coinvariants in the following.
\end{remark}

\section{The homogeneous space of a PPAV's extensions}
\label{sec:homogspaces}

In preparation for the construction of the spaces of coinvariants assigned to PPAVs treated in the next \S\ref{sec:VPPAV}, we 
discuss here the space $\mathrm{Sp}(X,\Theta)$ of all extended PPAVs with fixed underlying PPAV $(X,\Theta)$. We show how this is a homogeneous space for the action of a group $\mathbb{S}\mathrm{p}^+\left(H' \right)$ with nontrivial stabilizers. 
We end with a consequence of the equivariance of the sheaf of isotropy algebras with respect to the action of $\mathbb{S}\mathrm{p}^+\left(H' \right)$.

\subsection{The curve case: the space $\mathrm{Aut}(C)$}
\label{sec:AutC}
Before treating the PPAV case, first we review here the curve case as treated in \cite{bzf}.

For an algebraic curve $C$ of genus $g$, let
\[
\mathrm{Aut}(C) := \left\{\mbox{all } (C,P,t) \mbox{ with fixed }C\right\}.
\]
This set has naturally the structure of an infinite-dimensional complex manifold and 
has a natural identification with the fiber of the forgetful map \mbox{$\widehat{\mathcal{M}}_g\rightarrow \mathcal{M}_g$} from \S\ref{sec:Torelli} over the moduli point corresponding to $C$. 

The fibers of the forgetful map $\mathrm{Aut}(C)\rightarrow C$ are naturally torsors of the group scheme
\begin{equation}
\label{eq:AutO}
\mathbb{A}\mathrm{ut}(\mathbb{O}) := \left\{t\mapsto a_1t + a_2t^2 + \dots \,|\, a_1\mbox{ is a unit} \right\}
\end{equation}
which acts by change of coordinates. Thus $\mathrm{Aut}(C)\rightarrow C$ is a principal $\mathbb{A}\mathrm{ut}(\mathbb{O})$-bundle.

\subsection{The PPAV case: the space $\mathrm{Sp}(X,\Theta)$}
\label{sec:SpXTheta}
For a PPAV $(X,\Theta)$ of dimension $g$, define $\mathrm{Sp}(X,\Theta)$ by the cartesian diagram 
\[
\begin{tikzcd}
 \mathrm{Sp}(X,\Theta) \arrow[hookrightarrow]{r} \arrow{d} & \widehat{\mathcal{A}}_g \arrow{d}\\
 \mathrm{Spec}\,(\mathbb{C}) \arrow[]{r}[near start]{(X,\Theta)} &   \mathcal{A}_g.
\end{tikzcd}
\]

Its points are $(Z,F,L,\varphi)$ where $(Z,F,L)$ is an extended PPAV and 
\[
\varphi\colon ((Z/F)/p(L), B)\rightarrow (X,\Theta)
\]
is an isomorphism of PPAVs. A symplectic basis of $H_1(X,\mathbb{Z})$, with period matrix $\Omega \in \mathcal{H}_g$, identifies $\mathrm{Sp}(X,\Theta)$ with the fibre 
\[
\widetilde{S}^2(H'_+) \times \mathcal{B}_g(H'_+) \times \{\Omega\} 
\]
of $\widehat{\mathcal{H}}_g \rightarrow \mathcal{H}_g$, which gives $\mathrm{Sp}(X,\Theta)$ the structure of an infinite-dimensional complex manifold; another basis changes the identification by~\eqref{eq:Sp2gZaction}.
Note that $\mathrm{Aut}(X,\Theta)$ (given by the stabilizer of the action of $\mathrm{Sp}(2g,\mathbb{Z})$ at a period matrix for $(X,\Theta)$) 
acts freely on $\mathrm{Sp}(X,\Theta)$ via \eqref{eq:Sp2gZaction}.

\subsection{The Jacobian case}
When $(X,\Theta)$ is the Jacobian of a curve $C$, one has an inclusion
\[
\mathrm{Aut}(C) \hookrightarrow \mathrm{Sp}(X,\Theta)
\]
obtained by restricting the inclusion $\widehat{\mathcal{M}}_g\hookrightarrow \widehat{\mathcal{A}}_g$ from \eqref{eq:extTormap}, and thus the commutative diagram of Figure \ref{fig:AutCSpXTheta}.

\begin{figure}[t]
\begin{tikzcd}
& \mathrm{Sp}(X,\Theta) \arrow[hookrightarrow]{rr} \arrow{dd} && \widehat{\mathcal{A}}_g \arrow{dd}\\
\mathrm{Aut}(C) \arrow[crossing over, hookrightarrow]{rr} \arrow[hookrightarrow]{ru} \arrow{dd} && \widehat{\mathcal{M}}_g \arrow[hookrightarrow]{ru}\\
& \mathrm{Spec}\,(\mathbb{C}) \arrow[]{rr}[near start]{(X,\Theta)} & &  \mathcal{A}_g\\
 \mathrm{Spec}\,(\mathbb{C}) \arrow[]{rr}[near start]{C} \arrow[]{ru} & &  \mathcal{M}_g. \arrow[hookrightarrow]{ru} \arrow[crossing over, leftarrow]{uu} 
\end{tikzcd}
\caption{The space $\mathrm{Sp}(X,\Theta)$ for a Jacobian.}
\label{fig:AutCSpXTheta}
\end{figure}

\subsection{Symplectic group ind-schemes}
\label{sec:spgp}
Here we review the group ind-schemes of symplectic automorphisms from \cite[\S3.2]{avva}.

Let $\mathbb{GL}_b\left(H'\right)$ be the group ind-scheme defined as follows: for a $\mathbb{C}$-algebra $R$, its $R$ points are the  continuous linear automorphisms $\rho$ of $H'\,\widehat{\otimes}_\mathbb{C}\,R$ admitting $k\in\mathbb{N}$ with 
\begin{equation*}
t^{n+k}H'_+\,\widehat{\otimes}_\mathbb{C}\,R \subseteq \rho\left(t^n H'_+\,\widehat{\otimes}_\mathbb{C}\,R\right)\subseteq t^{n-k}H'_+\,\widehat{\otimes}_\mathbb{C}\, R \qquad\mbox{for all $n\in \mathbb{Z}$.}
\end{equation*}
Here the image of $t^m H'_+\,\widehat{\otimes}_\mathbb{C}\, R$ in $H'\,\widehat{\otimes}_\mathbb{C}\,R$ is intended.

Let $\mathbb{S}\mathrm{p}^+\left(H' \right)$ be the group ind-scheme
\[
\mathbb{S}\mathrm{p}^+\left(H' \right) := \left\{
\rho\in\mathrm{GL}_b\left(H'\right) \, \Bigg| \, 
\begin{array}{ll}
\langle\rho(f), \rho(g) \rangle = \langle f, g \rangle & \mbox{for all $f,g\in H'$,}\\[5pt]
\rho\left(H'_+\right) = H'_+
\end{array}
\right\}.
\]
The group scheme $\mathbb{A}\mathrm{ut}(\mathbb{O})$ from \eqref{eq:AutO} is naturally a subgroup of $\mathbb{S}\mathrm{p}^+\left(H' \right)$.

Also, for an extended PPAV $(Z,F,L)$, let $\mathbb{S}\mathrm{p}^+_F\left(H' \right)$ be the subgroup ind-scheme of $\mathbb{S}\mathrm{p}^+\left(H' \right)$ defined as
\[
\mathbb{S}\mathrm{p}^+_F\left(H' \right) := \left\{
\rho\in\mathbb{S}\mathrm{p}^+\left(H' \right) \, \Bigg| \,
\begin{array}{l}
\rho(F^\perp) = F^\perp,\\[5pt]
\rho(F) = F,\\[5pt]
\rho|_{F^\perp/F} = \mathrm{id}|_{F^\perp/F}
\end{array}
\right\}.
\]
Here $\rho|_{F^\perp/F}$ stands for the linear automorphism of $F^\perp/F$ induced by $\rho$, well defined since $\rho(F^\perp) = F^\perp$ and $\rho(F) = F$.

\begin{theorem}[{\cite[Thm 3.1]{avva}}]
\label{thm:Sp+action}
For a PPAV $(X,\Theta)$: 
\begin{enumerate}[(i)]
\item The group $\mathbb{S}\mathrm{p}^+\left(H' \right)$ acts transitively on the space $\mathrm{Sp}(X,\Theta)$. 
\smallskip
\item The stabilizer of a point $(Z,F,L,\varphi)$ in $\mathrm{Sp}(X,\Theta)$ is $\mathbb{S}\mathrm{p}^+_F\left(H' \right)$.
\smallskip
\item The subgroups $\mathbb{S}\mathrm{p}^+_F\left(H' \right)$ obtained by varying $(Z, F, L)$ are pairwise conjugated inside $\mathbb{S}\mathrm{p}^+\left(H' \right)$.
\end{enumerate}
\end{theorem}

\begin{proof}
Parts (i) and (iii) are as in \cite[Thm 3.1]{avva}. For part (ii), notice that the elements of $\mathbb{S}\mathrm{p}^+\left(H' \right)$ fixing $(Z,F,L,\varphi)$ in particular fix $\varphi$, hence they are the identity on $F^\perp/F$, and thus they are in $\mathbb{S}\mathrm{p}^+_F\left(H' \right)$.
\end{proof}

\noindent Thus for a point $(Z,F,L,\varphi)$ in $\mathrm{Sp}(X,\Theta)$, one has an isomorphism of spaces
\[
\mathrm{Sp}(X,\Theta) \,\cong\, 
\mathbb{S}\mathrm{p}^+\left(H' \right) \, / \, \mathbb{S}\mathrm{p}^+_F\left(H' \right).
\]

\subsection{Symplectic algebras}
To describe the Lie algebra of $\mathbb{S}\mathrm{p}^+\left(H' \right)$, consider 
\begin{equation}
\label{eq:sp+}
\mathfrak{sp}^+\left( H'\right):= \left\{ X\in\mathfrak{sp}\left( H'\right)\,|\, X\left(H'_+ \right)\subseteq H'_+ \right\}.
\end{equation}
This is the Lie subalgebra of $\mathfrak{sp}\left( H'\right)$ topologically generated by $b_i b_j$ with $i$ and $j$ not both negative.
One checks
\[
\mathrm{Lie}\left(\mathbb{S}\mathrm{p}^+\left(H' \right) \right) = \mathfrak{sp}^+_b\left( H'\right)
\]
where 
\[
\mathfrak{sp}^+_b\left(H' \right):= \left\{ X\in\mathfrak{sp}^+\left(H' \right) \,\, | \,\,\exists k\in\mathbb{N}: X(t^n H'_+)\subseteq t^{n-k}H'_+ \quad\mbox{for all $n$}\right\}.
\]
This is a dense Lie subalgebra of $\mathfrak{sp}^+\left(H' \right)$ containing the dense subset $S^2(H')\cap \mathfrak{sp}^+\left(H' \right)$. 

Also, recall the Lie algebra $\mathfrak{sp}_F\left( H'\right)$ from \eqref{eq:spF} and define
\begin{equation}
\label{eq:sp+F}
\mathfrak{sp}^+_F\left( H'\right) := \mathfrak{sp}^+\left( H'\right)\cap \mathfrak{sp}_F\left( H'\right).
\end{equation}
This is the \textit{isotropy Lie algebra} at the point $(Z,F,L,\varphi)$ of $\mathrm{Sp}(X,\Theta)$.
It is a Lie subalgebra of $\mathfrak{sp}^+\left( H'\right)$ topologically generated by $fb_i$ with $f\in F$ and $i>0$.
One has
\[
\mathrm{Lie}\left(\mathbb{S}\mathrm{p}^+_F\left(H' \right) \right) = \mathfrak{sp}^+_F\left( H'\right)\cap \mathfrak{sp}^+_b\left( H'\right).
\]

\subsection{Invariance of isotropy sheaves}
The following consequence of Theorem \ref{thm:Sp+action} will be useful.
Consider the sheaf of Lie algebras $\mathfrak{sp}^+\left( H'\right)\,\widehat{\otimes}_\mathbb{C}\, \mathscr{O}_{\mathrm{Sp}(X,\Theta)}$
and the \textit{sheaf of isotropy Lie subalgebras}
\begin{equation*}
\mathfrak{sp}^{+, \mathrm{out}}\left(X,\Theta\right) \subset \mathfrak{sp}^+\left( H'\right) \,\widehat{\otimes}_\mathbb{C}\, \mathscr{O}_{\mathrm{Sp}(X,\Theta)}
\end{equation*}
whose fiber at $(Z,F,L)$ is $\mathfrak{sp}^+_F\left( H'\right)$.
This sheaf can be constructed as follows.
Let 
\begin{equation}
\label{eq:universalFSp}
\mathbb{F}\rightarrow\mathrm{Sp}(X,\Theta) 
\end{equation}
be the tautological sheaf whose fiber at $(Z,F,L)$ is $F$. Then
\begin{equation}
\label{eq:sp+out}
\mathfrak{sp}^{+, \mathrm{out}}\left(X,\Theta\right) = \mathbb{F} \,\widehat{\otimes}_{\mathbb{C}}\, H'_+.
\end{equation}

From Theorem \ref{thm:Sp+action} and \cite[\S3.1]{avva}, $\mathfrak{sp}^+\left( H'\right) \,\widehat{\otimes}_\mathbb{C}\, \mathscr{O}_{\mathrm{Sp}(X,\Theta)}$ is a Lie algebroid on $\mathrm{Sp}(X,\Theta)$ with surjective anchor map 
\[
\mathfrak{sp}^+\left( H'\right) \,\widehat{\otimes}_\mathbb{C}\, \mathscr{O}_{\mathrm{Sp}(X,\Theta)} \rightarrow \mathscr{T}_{\mathrm{Sp}(X,\Theta)}.
\]
The sheaf $\mathfrak{sp}^{+, \mathrm{out}}\left(X,\Theta\right)$ is by definition the kernel of this anchor map.

The actions of $\mathfrak{sp}^+\left(H'\right)$ on $H'$ and on $\mathrm{Sp}(X,\Theta)$ induce an action of $\mathfrak{sp}^+\left(H'\right)$ on $H' \,\widehat{\otimes}_\mathbb{C}\, \mathscr{O}_{\mathrm{Sp}(X,\Theta)}$:
\begin{equation}
\label{eq:sp+actiononH'sheaf}
S\left( h\otimes f\right) = S(h)\otimes f + h\otimes S(f)
\end{equation}
for $S\in \mathfrak{sp}^+\left(H'\right)$ and $h\otimes f$ a local section of $H' \,\widehat{\otimes}_\mathbb{C}\, \mathscr{O}_{\mathrm{Sp}(X,\Theta)}$.

Similarly, the actions of $\mathfrak{sp}^+\left(H'\right)$ on itself and on $\mathrm{Sp}(X,\Theta)$ induce an action of $\mathfrak{sp}^+\left(H'\right)$ on $\mathfrak{sp}^+\left(H'\right) \,\widehat{\otimes}_\mathbb{C}\, \mathscr{O}_{\mathrm{Sp}(X,\Theta)}$:
\begin{equation}
\label{eq:sp+action}
S\left( T\otimes f\right) = [S,T]\otimes f + T\otimes S(f)
\end{equation}
for $S\in \mathfrak{sp}^+\left(H'\right)$ and $T\otimes f$ a local section of $\mathfrak{sp}^+\left(H'\right) \,\widehat{\otimes}_\mathbb{C}\, \mathscr{O}_{\mathrm{Sp}(X,\Theta)}$.

\begin{corollary}
\label{cor:sp+outpreserved}
\begin{enumerate}[(i)]
\item The subsheaf 
\[
\mathbb{F} \subset
H' \,\widehat{\otimes}_\mathbb{C}\, \mathscr{O}_{\mathrm{Sp}(X,\Theta)}
\]
is preserved by the action of $\mathfrak{sp}^+\left(H'\right)$ on $H' \,\widehat{\otimes}_\mathbb{C}\, \mathscr{O}_{\mathrm{Sp}(X,\Theta)}$.

\item The subsheaf 
\[
\mathfrak{sp}^{+, \mathrm{out}}\left(X,\Theta\right) \subset
\mathfrak{sp}^+\left( H'\right) \,\widehat{\otimes}_\mathbb{C}\, \mathscr{O}_{\mathrm{Sp}(X,\Theta)}
\]
is preserved by the action of $\mathfrak{sp}^+\left(H'\right)$ on $\mathfrak{sp}^+\left( H'\right) \,\widehat{\otimes}_\mathbb{C}\, \mathscr{O}_{\mathrm{Sp}(X,\Theta)}$.
\end{enumerate}
\end{corollary}

\begin{proof}
Theorem \ref{thm:Sp+action} implies that $\mathbb{F}$ is $\mathbb{S}\mathrm{p}^+\left(H' \right)$-equivariant, hence part (i) follows for $\mathfrak{sp}^+_b\left(H'\right)$, and thus for $\mathfrak{sp}^+\left(H'\right)$ by continuity.
 Part (ii) follows from the realization \eqref{eq:sp+out} and part (i), together with the fact that $\mathfrak{sp}^+\left(H'\right)$ preserves $H'_+$ by definition \eqref{eq:sp+}.
\end{proof}

\section{Spaces of coinvariants for PPAVs}
\label{sec:VPPAV}

Here we construct spaces of coinvariants associated to PPAVs.
We start by presenting the strategy in the curve case and highlight a key difference arising in the PPAV case.
For Jacobians, we show how the spaces defined here recover the spaces assigned to the corresponding curves (Theorem \ref{thm:VonTorelli}).
We end with a generalization of Theorem \ref{thm:Vpi0} (Theorem \ref{thm:Vpi0Gamma}) and its proof.

\subsection{The curve case}
\label{sec:curvecasedescent}
Spaces of coinvariants assigned to curves are constructed in \cite[\S9.1]{bzf} as follows.
For an algebraic curve $C$ and a Heisenberg vertex algebra $V$, one starts from the trivial bundle 
\[
V\times \mathrm{Aut}(C)\rightarrow \mathrm{Aut}(C). 
\]
The quasi-conformal structure of $V$ induces an action of $\mathrm{Der}_0(\mathbb{O}):=t\mathbb{C}\llbracket t\rrbracket\partial_t$ on~$V$. One has $\mathrm{Der}_0(\mathbb{O}) = \mathrm{Lie}\left( \mathbb{A}\mathrm{ut}(\mathbb{O})\right)$, where 
$\mathbb{A}\mathrm{ut}(\mathbb{O})$ is the group scheme from \eqref{eq:AutO}.
The action of $\mathrm{Der}_0(\mathbb{O})$ exponentiates to an  action on $V$ of $\mathbb{A}\mathrm{ut}(\mathbb{O})$ \cite[\S6.3]{bzf}.
Combining with the simply transitive action of $\mathbb{A}\mathrm{ut}(\mathbb{O})$ on the fibers of $\mathrm{Aut}(C)\rightarrow C$, one has an $\mathbb{A}\mathrm{ut}(\mathbb{O})$-equivariant structure on $V\times \mathrm{Aut}(C)$. Via descent along the principal $\mathbb{A}\mathrm{ut}(\mathbb{O})$-bundle $\mathrm{Aut}(C)\rightarrow C$, this yields a bundle on $C$, denoted $\mathscr{V}\rightarrow C$.

To construct the spaces of coinvariants assigned to $C$, one then proceeds as follows: the action on $V$ of the Lie algebra of outgoing states from \eqref{eq:houtCPt} is $\mathbb{A}\mathrm{ut}(\mathbb{O})$-equivariant, thus one can construct a sheaf of Lie algebras of outgoing states on $C$ acting on $\mathscr{V}$ and then define the space of coinvariants as the fibers over $C$ of the quotient of $\mathscr{V}$ by this action.

Now, the above strategy does not carry over verbatim to the PPAV case. The group ind-scheme $\mathbb{S}\mathrm{p}^+\left(H' \right)$ does act on V, extending the action of $\mathbb{A}\mathrm{ut}(\mathbb{O})$ (see \S\ref{sec:equivstronVhatXThetaV}), so that for a PPAV $(X,\Theta)$ the trivial bundle $V\times \mathrm{Sp}(X,\Theta) \rightarrow \mathrm{Sp}(X,\Theta)$ is $\mathbb{S}\mathrm{p}^+\left(H' \right)$-equivariant and descends to the quotient stack 
\[
[\mathrm{Sp}(X,\Theta)\,/\,\mathbb{S}\mathrm{p}^+\left(H' \right)]. 
\]
However, while $\mathbb{A}\mathrm{ut}(\mathbb{O})$ acts freely on the fibers of $\mathrm{Aut}(C)\rightarrow C$, the action of $\mathbb{S}\mathrm{p}^+\left(H' \right)$ on $\mathrm{Sp}(X,\Theta)$ has nontrivial stabilizers $\mathbb{S}\mathrm{p}^+_F\left(H' \right)$ (Theorem \ref{thm:Sp+action}), and these act nontrivially on $V$. Hence $V\times \mathrm{Sp}(X,\Theta)$ does not descend further to a vector space. What is needed is a quotient of $V$ on which the stabilizers act trivially.

Thus, instead of the trivial bundle $V\times \mathrm{Sp}(X,\Theta) \rightarrow \mathrm{Sp}(X,\Theta)$, we start from the bundle on $\mathrm{Sp}(X,\Theta)$ which globalizes the spaces of coinvariants associated to extended PPAVs with underlying $(X,\Theta)$ from \S\ref{sec:VhataV}. Before discussing this in detail starting from \S\ref{sec:VhatXthetaV}, we briefly run this strategy in the curve case. The resulting spaces of coinvariants are equivalent to the outcome of the previous strategy.

Specifically, the spaces of coinvariants $\widehat{\mathbb{V}}(C,P,t,V)$ from \eqref{eq:VhatCPTtV} globalize to a quasi-coherent sheaf on $\mathrm{Aut}(C)$, which we denote as
\[
\widehat{\mathbb{V}}(C,V) \rightarrow \mathrm{Aut}(C).
\]
The action of $\mathrm{Der}_0(\mathbb{O})$ on $V$ induces an action of $\mathrm{Der}_0(\mathbb{O})$ on $\widehat{\mathbb{V}}(C,V)$ which exponentiates to an $\mathbb{A}\mathrm{ut}(\mathbb{O})$-equivariant structure on $\widehat{\mathbb{V}}(C,V)$ (this is in \cite[Thm 17.3.11]{bzf}).
Via descent along the principal $\mathbb{A}\mathrm{ut}(\mathbb{O})$-bundle $\mathrm{Aut}(C)\rightarrow C$, this yields a bundle on $C$. The fiber at the point $P$ is the vector space of coinvariants assigned to $(C,P)$, here denoted $\mathbb{V}(C,P,V)$.

Finally, as the spaces $\mathbb{V}(C,P,V)$ obtained by varying the point $P$ in $C$ are all canonically isomorphic (propagation of vacua, \cite[Thm 10.3.1]{bzf}), they determine a space assigned to $C$, denoted $\mathbb{V}(C,V)$. This is the \textit{space of coinvariants assigned to $C$.}

\subsection{The sheaf $\widehat{\mathbb{V}}(X,\Theta,V)$}
\label{sec:VhatXthetaV}
For a PPAV $(X,\Theta)$ and a Heisenberg vertex algebra $V$, the spaces of coinvariants associated to extended PPAVs with underlying $(X,\Theta)$ from \S\ref{sec:VhataV} globalize to a quasi-coherent sheaf on $\mathrm{Sp}(X,\Theta)$, which we denote as 
\[
\widehat{\mathbb{V}}(X,\Theta,V)\rightarrow\mathrm{Sp}(X,\Theta).
\]
Specifically, recall the tautological sheaf $\mathbb{F}\rightarrow\mathrm{Sp}(X,\Theta)$ from \eqref{eq:universalFSp}.
Define
\[
\mathfrak{h}_\Gamma^{\mathrm{out}}(X,\Theta):= \mathfrak{h}_\Gamma \otimes_{\mathbb{C}} \mathbb{F}.
\]
This is a subsheaf of the sheaf of Lie algebras $\mathscr{H}_\Gamma \,\widehat{\otimes}_\mathbb{C}\, \mathscr{O}_{\mathrm{Sp}(X,\Theta)}$.
The action of $\mathscr{H}_\Gamma$ on $V$ extends $\mathscr{O}_{\mathrm{Sp}(X,\Theta)}$-linearly to an action of $\mathscr{H}_\Gamma \,\widehat{\otimes}_\mathbb{C}\, \mathscr{O}_{\mathrm{Sp}(X,\Theta)}$ on $V\otimes_\mathbb{C} \mathscr{O}_{\mathrm{Sp}(X,\Theta)}$. This induces an action of $\mathfrak{h}_\Gamma^{\mathrm{out}}(X,\Theta)$ on $V\otimes_\mathbb{C} \mathscr{O}_{\mathrm{Sp}(X,\Theta)}$. Then the sheaf $\widehat{\mathbb{V}}(X,\Theta,V)$ is defined as
\[
\widehat{\mathbb{V}}(X,\Theta,V) := V\otimes_\mathbb{C} \mathscr{O}_{\mathrm{Sp}(X,\Theta)} \,/\, \mathfrak{h}_\Gamma^{\mathrm{out}}(X,\Theta) \left( V\otimes_\mathbb{C} \mathscr{O}_{\mathrm{Sp}(X,\Theta)}\right).
\]

\subsection{The positive symplectic action}
\label{sec:possympact}
First, we note that $V$ has an action of $\mathfrak{sp}^+\left(H'\right)$.
Indeed, when $V$ has non-zero level,
recall the action of the metaplectic algebra $\mathfrak{mp}\left(H'\right)$ on $V$ from \S\ref{sec:mpaction}.
It is easy to see that the universal two-cocycle on $\mathfrak{sp}\left(H'\right)$ defining $\mathfrak{mp}\left(H'\right)$ vanishes on the Lie subalgebra $\mathfrak{sp}^+\left(H'\right)$ from \eqref{eq:sp+} --- see \cite[\S4.2]{avva}. Thus one has a Lie algebra splitting
\begin{equation}
\label{eq:sp+splitting}
\mathfrak{sp}^+\left(H'\right) \hookrightarrow \mathfrak{mp}\left(H'\right), \qquad 
b_i b_j \mapsto :\! b_ib_j\!:
\end{equation}
with $i$ and $j$ not both negative.
This induces an action of $\mathfrak{sp}^+\left(H'\right)$ on $V$.
When $V$ has level zero, $\mathfrak{sp}^+(H')$ preserves $H'_+$, hence
acts on $H'/H'_+$ and by derivations on
$V\cong\operatorname{Sym}\bigl(\mathfrak h_\Gamma\otimes_{\mathbb C}H'/H'_+\bigr)$.
Explicitly, $b_mb_n$ with $m<0<n$ acts as
$\sum_{i=1}^d n\,H^i_m\,\partial/\partial H^i_{-n}$, independently of the choice of an orthonormal basis $\{H^i\}_i$, and $b_mb_n$ with $m,n>0$ acts by
zero.

On the other hand, the differential of the transitive action of $\mathbb{S}\mathrm{p}^+\left(H' \right)$ on $\mathrm{Sp}(X,\Theta)$ from
Theorem \ref{thm:Sp+action} induces an action of $\mathfrak{sp}^+\left(H'\right)$ on $\mathrm{Sp}(X,\Theta)$.

The actions of $\mathfrak{sp}^+\left(H'\right)$ on $V$ and on $\mathrm{Sp}(X,\Theta)$ induce an action of $\mathfrak{sp}^+\left(H'\right)$ on $V\otimes_\mathbb{C}\mathscr{O}_{\mathrm{Sp}(X,\Theta)}$:
\[
S\left( A\otimes f\right) = S(A)\otimes f + A\otimes S(f)
\]
for $S\in \mathfrak{sp}^+\left(H'\right)$ and $A\otimes f$ a local section of $V\otimes_\mathbb{C}\mathscr{O}_{\mathrm{Sp}(X,\Theta)}$.
For an extended PPAV $(Z,F,L)$, recall the Lie algebra $\mathfrak{sp}^+_F\left( H'\right)$ from \eqref{eq:sp+F}.

\begin{theorem}
\label{thm:houtispreserved}
\begin{enumerate}[(i)]
\item The subsheaf 
\[
\mathfrak{h}_\Gamma^{\mathrm{out}}(X,\Theta) \left( V\otimes_\mathbb{C} \mathscr{O}_{\mathrm{Sp}(X,\Theta)}\right) \subset
V\otimes_\mathbb{C} \mathscr{O}_{\mathrm{Sp}(X,\Theta)}
\]
is preserved by the action of $\mathfrak{sp}^+\left(H'\right)$ on $V\otimes_\mathbb{C} \mathscr{O}_{\mathrm{Sp}(X,\Theta)}$.
Hence one has an induced action of $\mathfrak{sp}^+\left(H'\right)$ on the quotient $\widehat{\mathbb{V}}(X,\Theta,V)$.

\smallskip

\item For each $a=(Z,F,L, \varphi)$ extending $(X,\Theta)$, the action of $\mathfrak{sp}^+\left(H'\right)$ on $\widehat{\mathbb{V}}(X,\Theta,V)$ induces a trivial action of $\mathfrak{sp}^+_F\left(H'\right)$ on the fiber $\widehat{\mathbb{V}}(a,V)$ of $\widehat{\mathbb{V}}(X,\Theta,V)$ at $a$.
\end{enumerate}
\end{theorem}

\begin{proof}
Assume that $V$ has non-zero level.
Let
\[
\mathfrak{sp}_{\Gamma}^+\left(H' \right) 
\]
be the Lie subalgebra of $\mathscr{H}_\Gamma$ topologically generated by $A_mB_n$ with \mbox{$A,B\in\mathfrak{h}_\Gamma$} and $m,n$ not both negative. 
Define
\begin{equation}
\label{eq:sp+outGamma}
\mathfrak{sp}_{\Gamma}^{+, \mathrm{out}}\left(X,\Theta\right) := 
\left(\mathfrak{h}_\Gamma \otimes_{\mathbb{C}} \mathbb{F} \right)
\widehat{\otimes}_{\mathbb{C}} 
\left(\mathfrak{h}_\Gamma \,\widehat{\otimes}_\mathbb{C}\, H'_+\right).
\end{equation}
The natural map
\begin{equation}
\label{eq:sp+outGammaintoHO}
\mathfrak{sp}_{\Gamma}^{+, \mathrm{out}}\left(X,\Theta\right)  \rightarrow \mathfrak{sp}_{\Gamma}^+\left(H' \right) \,\widehat{\otimes}_\mathbb{C}\, \mathscr{O}_{\mathrm{Sp}(X,\Theta)}, \,\,\,
(A\otimes f) (B\otimes h) \mapsto (A\otimes f) (B\otimes h)
\end{equation}
realizes $\mathfrak{sp}_{\Gamma}^{+, \mathrm{out}}\left(X,\Theta\right) $ 
as a subsheaf of Lie algebras
\begin{equation}
\label{eq:sp+XThetasubsheaf}
\mathfrak{sp}^{+, \mathrm{out}}_\Gamma\left(X,\Theta\right)
\subset \mathfrak{sp}_{\Gamma}^+\left(H' \right) \,\widehat{\otimes}_\mathbb{C}\, \mathscr{O}_{\mathrm{Sp}(X,\Theta)}. 
\end{equation}
The $\mathscr{O}_{\mathrm{Sp}(X,\Theta)}$-linear action of $\mathscr{H}_\Gamma \,\widehat{\otimes}_\mathbb{C}\, \mathscr{O}_{\mathrm{Sp}(X,\Theta)}$ on \mbox{$V\otimes_\mathbb{C} \mathscr{O}_{\mathrm{Sp}(X,\Theta)}$} induces an action of $\mathfrak{sp}_{\Gamma}^{+,\mathrm{out}}\left(X,\Theta\right)$ on $V\otimes_\mathbb{C} \mathscr{O}_{\mathrm{Sp}(X,\Theta)}$.

To show the statement, first we argue that 
\begin{equation}
\label{eq:houtissp+out}
\mathfrak{h}_\Gamma^{\mathrm{out}}(X,\Theta) \left( V\otimes_\mathbb{C} \mathscr{O}_{\mathrm{Sp}(X,\Theta)}\right) = 
\mathfrak{sp}_{\Gamma}^{+,\mathrm{out}}\left(X,\Theta\right) \left( V\otimes_\mathbb{C} \mathscr{O}_{\mathrm{Sp}(X,\Theta)}\right).
\end{equation}
Indeed, by the inclusion \eqref{eq:sp+outGammaintoHO}, one clearly has
\[
\mathfrak{h}_\Gamma^{\mathrm{out}}(X,\Theta) \left( V\otimes_\mathbb{C} \mathscr{O}_{\mathrm{Sp}(X,\Theta)}\right) \supseteq 
\mathfrak{sp}_{\Gamma}^{+,\mathrm{out}}\left(X,\Theta\right) \left( V\otimes_\mathbb{C} \mathscr{O}_{\mathrm{Sp}(X,\Theta)}\right).
\]
The opposite inclusion follows from the identity
\[
\left(\mathfrak{h}_\Gamma \,\widehat{\otimes}_\mathbb{C}\, H'_+\right) (V) = V.
\]
This proves \eqref{eq:houtissp+out}.

We proceed by showing that the subsheaf 
\begin{equation}
\label{eq:sp+GammaoutXThetasubsheaf}
\mathfrak{sp}_{\Gamma}^{+, \mathrm{out}}\left(X,\Theta\right)
\left( V\otimes_\mathbb{C} \mathscr{O}_{\mathrm{Sp}(X,\Theta)}\right) \subset
V\otimes_\mathbb{C} \mathscr{O}_{\mathrm{Sp}(X,\Theta)}
\end{equation}
is preserved by the action of $\mathfrak{sp}^+\left(H'\right)$ on $V\otimes_\mathbb{C} \mathscr{O}_{\mathrm{Sp}(X,\Theta)}$.

One has a Lie algebra inclusion
\[
\mathfrak{sp}^+\left(H' \right) \rightarrow \mathfrak{sp}_{\Gamma}^+\left(H' \right), \qquad
b_mb_n \mapsto \sum_{i=1}^d :\!H^i_mH^i_n\!:
\]
with notation as in Lemma \ref{lemma:mpaction}, and thus an action of $\mathfrak{sp}^+\left(H' \right)$ on \sloppy{$\mathfrak{sp}_{\Gamma}^+\left(H' \right) \,\widehat{\otimes}_\mathbb{C}\, \mathscr{O}_{\mathrm{Sp}(X,\Theta)}$} 
as in \eqref{eq:sp+action}. It is enough to show that this action preserves the subsheaf \eqref{eq:sp+XThetasubsheaf}.
Indeed, for $S\in \mathfrak{sp}^+\left(H' \right)$ and local sections $T\otimes f\in \mathfrak{sp}^{+, \mathrm{out}}_\Gamma\left(X,\Theta\right)$ and $v\otimes h\in V\otimes_\mathbb{C} \mathscr{O}_{\mathrm{Sp}(X,\Theta)}$, one has
\[
S((T\otimes f)v\otimes h) 
= [S, T\otimes f](v\otimes h) + (T\otimes f)S(v\otimes h)
\]
where $[S, T\otimes f]$ is the action of $S$ as in \eqref{eq:sp+action};
hence the stability of the subsheaf \eqref{eq:sp+XThetasubsheaf} implies the stability of the subsheaf \eqref{eq:sp+GammaoutXThetasubsheaf} under $\mathfrak{sp}^+\left(H'\right)$.

When $\mathrm{rank}(\Gamma) = 1$, the sheaf of Lie algebras \eqref{eq:sp+outGamma} recovers \eqref{eq:sp+out}, that is, one has 
\[
\mathfrak{sp}_{\Gamma}^{+, \mathrm{out}}\left(X,\Theta\right) = \mathfrak{sp}^{+, \mathrm{out}}\left(X,\Theta\right).
\]
In this case, the stability of the subsheaf \eqref{eq:sp+XThetasubsheaf} under $\mathfrak{sp}^+\left(H'\right)$ follows from Corollary \ref{cor:sp+outpreserved}(ii).

When $\mathrm{rank}(\Gamma) > 1$, first note that  
for $S\in \mathfrak{sp}^+\left(H' \right)$ and $A\otimes f\in \mathfrak{h}_\Gamma \,\widehat{\otimes}_\mathbb{C}\, H'$, one has
\begin{equation}
\label{eq:SAf}
[S, A\otimes f] = A\otimes [S,f].
\end{equation}
This can be verified easily by reducing to the case $S=b_mb_n$ and $f=b_k$ by continuity and linearity:
\begin{eqnarray*}
\left[ \sum_{i=1}^d H^i_mH^i_n, A_k\right] &=& \sum_i(H^i,A)\left( H^i_m\langle b_n, b_k\rangle + H^i_n\langle b_m, b_k\rangle\right)\\
&=& A_m\langle b_n, b_k\rangle + A_n\langle b_m, b_k\rangle\\
&=& A\otimes [b_mb_n, b_k].
\end{eqnarray*}
Since $[S,\,]$ is a derivation and applying \eqref{eq:SAf}, one has
\begin{equation}
\label{eq:SAfBh}
[S,(A\otimes f)(B\otimes h)]= (A\otimes [S,f])(B\otimes h) + (A\otimes f)(B\otimes [S,h])
\end{equation}
for $S\in \mathfrak{sp}^+\left(H' \right)$ and $(A\otimes f)(B\otimes h)\in \mathfrak{sp}_{\Gamma}^+\left(H' \right)$.
Since $\mathfrak{sp}^+\left(H' \right)$ preserves $H'_+$ by definition \eqref{eq:sp+}, for $h\in H'_+$ one has $[S,h]\in H'_+$.
Apply \eqref{eq:SAfBh} with $f$ replaced by a local section of $\mathbb{F}$ and interpret $[S,f]$ as the action in \eqref{eq:sp+actiononH'sheaf}.
Then the stability of the subsheaf \eqref{eq:sp+XThetasubsheaf} under $\mathfrak{sp}^+\left(H'\right)$ follows from Corollary \ref{cor:sp+outpreserved}(i), hence part~(i).

For $a=(Z,F,L,\varphi)$ extending $(X,\Theta)$,
one has that $\mathfrak{sp}^+_F\left(H'\right)$ is a Lie subalgebra of the fiber of the sheaf $\mathfrak{sp}_{\Gamma}^{+, \mathrm{out}}\left(X,\Theta\right)$ at $a$. Thus part (ii) follows from \eqref{eq:houtissp+out}.

Finally, assume $V$ has level zero. For $S\in\mathfrak{sp}^+(H')$ and
$A\otimes f\in\mathfrak h_\Gamma\widehat\otimes_{\mathbb C}H'$, one has
$[S,A\otimes f]=A\otimes S(f)$ as operators on $V$, where $A\otimes f$ acts through
its class modulo $\mathfrak h_\Gamma\,\widehat{\otimes}_\mathbb{C}\, H'_+$. Part (i) then follows from
Corollary~\ref{cor:sp+outpreserved}(i). For part (ii), an element $S\in\mathfrak{sp}^+_F(H')$ maps
$H'$ into $F+H'_+$, so it acts on $V$ by a derivation mapping the generators into
$\mathfrak h_\Gamma\otimes_{\mathbb C}F$, and thus \mbox{$S(V)\subseteq\mathfrak h^{\rm out}_\Gamma(a)(V)$.}
\end{proof}

\subsection{Equivariant structure}
\label{sec:equivstronVhatXThetaV}
We thus have an induced action of $\mathfrak{sp}^+\left(H'\right)$ on the sheaf $\widehat{\mathbb{V}}(X,\Theta,V)$. Next, we study the action of the group ind-scheme $\mathbb{S}\mathrm{p}^+\left(H' \right)$.

First, note that 
$V$ carries an action of $\mathbb{S}\mathrm{p}^+\left(H' \right)$ whose differential is the $\mathfrak{sp}^+_b\left(H'\right)$-action: 
\[
\rho(u\bm{1}) = \rho(u)\bm{1}\qquad \mbox{for $\rho\in \mathbb{S}\mathrm{p}^+\left(H' \right)$ and $u\in \mathscr{H}_\Gamma$}
\]
as in \cite[Lemma 7.3]{avva}.
Here $\rho$ acts on $\mathfrak{h}_\Gamma \,\widehat{\otimes}_{\mathbb{C}}\, H'$ as $1\otimes \rho$, as in \eqref{eq:SAf}.
This action is well-defined since $\rho(H'_+)\subseteq H'_+$.

Recall the action of $\mathbb{S}\mathrm{p}^+\left(H' \right)$ on $\mathrm{Sp}(X,\Theta)$ from Theorem \ref{thm:Sp+action}. One thus has a diagonal action of $\mathbb{S}\mathrm{p}^+\left(H' \right)$ on $V\otimes_\mathbb{C} \mathscr{O}_{\mathrm{Sp}(X,\Theta)}$.

\begin{theorem}
\label{thm:expactionSp+}
\begin{enumerate}[(i)]
\item The subsheaf 
\begin{equation}
\label{eq:hGammaoutVO}
\mathfrak{h}_\Gamma^{\mathrm{out}}(X,\Theta) \left( V\otimes_\mathbb{C} \mathscr{O}_{\mathrm{Sp}(X,\Theta)}\right) \subset
V\otimes_\mathbb{C} \mathscr{O}_{\mathrm{Sp}(X,\Theta)}
\end{equation}
is preserved by the action of $\mathbb{S}\mathrm{p}^+\left(H' \right)$ on $V\otimes_\mathbb{C} \mathscr{O}_{\mathrm{Sp}(X,\Theta)}$.
This induces an $\mathbb{S}\mathrm{p}^+\left(H' \right)$-equivariant structure on the quotient $\widehat{\mathbb{V}}(X,\Theta,V)$. Its differential is the action of $\mathfrak{sp}^+_b\left(H'\right)$.

\smallskip

\item For each $a=(Z,F,L,\varphi)$ extending $(X,\Theta)$, the $\mathbb{S}\mathrm{p}^+\left(H' \right)$-equivariant structure on $\widehat{\mathbb{V}}(X,\Theta,V)$ induces a trivial action of $\mathbb{S}\mathrm{p}^+_F\left(H' \right)$ on the fiber $\widehat{\mathbb{V}}(a,V)$ of $\widehat{\mathbb{V}}(X,\Theta,V)$ at $a$.
\end{enumerate}
\end{theorem}

\begin{proof}
For $a=(Z,F,L,\varphi)$ extending $(X,\Theta)$, 
the fiber of the subsheaf \eqref{eq:hGammaoutVO} at $a$ is $\mathfrak{h}_\Gamma^{\mathrm{out}}(a)(V)$.
One has 
\[
\rho\left( \mathfrak{h}_\Gamma^{\mathrm{out}}(a) \left( V\right)\right) = \mathfrak{h}_\Gamma^{\mathrm{out}}(\rho(a)) \left( \rho(V)\right)
= \mathfrak{h}_\Gamma^{\mathrm{out}}(\rho(a)) \left( V\right).
\]
As the right-hand side is the fiber of the subsheaf \eqref{eq:hGammaoutVO} at $\rho(a)$, the stability of the subsheaf follows.
The compatibility with the $\mathfrak{sp}^+\left(H'\right)$-action and part (ii) follow by a direct computation (this is as in \cite[Lemma 7.3]{avva} and verbatim at level zero, where all operators involved commute).
\end{proof}

\subsection{Spaces of coinvariants for PPAVs}
\label{sec:coinvonPPAVs}
From Theorem \ref{thm:expactionSp+}(i), the sheaf $\widehat{\mathbb{V}}(X,\Theta,V)\rightarrow \mathrm{Sp}(X,\Theta)$ is 
$\mathbb{S}\mathrm{p}^+\left(H' \right)$-equivariant. Thus it descends to a sheaf $[\widehat{\mathbb{V}}(X,\Theta,V)/ \mathbb{S}\mathrm{p}^+\left(H' \right) ]$ on the quotient stack $[\mathrm{Sp}(X,\Theta)/\mathbb{S}\mathrm{p}^+\left(H' \right)]$.

Moreover, from Theorem \ref{thm:expactionSp+}(ii), the stabilizer $\mathbb{S}\mathrm{p}^+_F\left(H' \right)$ of a point $(Z,F,L,\varphi)$ in $\mathrm{Sp}(X,\Theta)$ acts trivially on $\widehat{\mathbb{V}}(X,\Theta,V)$. Applying \cite[Thm 10.3]{alper2013good}, the sheaf 
$[\widehat{\mathbb{V}}(X,\Theta,V)/ \mathbb{S}\mathrm{p}^+\left(H' \right) ]$ further descends to a vector space, which we denote
\[
\mathbb{V}(X,\Theta,V).
\]
We define this to be the \textit{space of coinvariants assigned to the PPAV $(X, \Theta)$}.

To summarize, this is the diagram of descents:
\[
\begin{tikzcd}
\widehat{\mathbb{V}}(X,\Theta,V) \arrow[]{r} \arrow[]{d} & \left[\widehat{\mathbb{V}}(X,\Theta,V)/ \mathbb{S}\mathrm{p}^+\left(H' \right) \right]\arrow[]{r} \arrow[]{d} & \mathbb{V}(X,\Theta,V) \arrow[]{d}\\
\mathrm{Sp}(X,\Theta) \arrow[]{r} & \left[\mathrm{Sp}(X,\Theta)/\mathbb{S}\mathrm{p}^+\left(H' \right)\right] \arrow[]{r} & \mathrm{Spec}\,(\mathbb{C}). 
\end{tikzcd}
\]
Note that $\mathrm{Aut}(X,\Theta)$ acts on $\mathrm{Sp}(X,\Theta)$ normalizing the $\mathbb{S}\mathrm{p}^+\left(H' \right)$-action (since $\mathrm{Sp}(2g,\mathbb{Z}) \ltimes \mathbb{S}\mathrm{p}^+\left(H' \right)$ acts on $\widehat{\mathcal{H}}_g$), hence acts on $\mathbb{V}(X,\Theta,V)$. This is the action of the automorphism group of the stacky point $[X,\Theta]$. 

Next, we show how this construction of spaces of coinvariants recovers the familiar spaces over the Jacobian locus.
Recall the notation from \S\ref{sec:curvecasedescent}.

\begin{theorem}
\label{thm:VonTorelli}
When $(X,\Theta)$ is the Jacobian of a curve $C$, one has
\[
\mathbb{V}(X,\Theta,V) = \mathbb{V}(C,V).
\]
\end{theorem}

\begin{proof}
Assume $(X,\Theta)$ is the Jacobian of $C$.
For a choice of an extended PPAV $a=(Z,F,L)$ with underlying PPAV $(X,\Theta)$ and a coordinatized curve $(C,P,t)$ for the fixed curve $C$,
Lemma \ref{lemma:VhatonTorelli} yields $\widehat{\mathbb{V}}(a,V) = \widehat{\mathbb{V}}(C,P,t,V)$.
More generally, the bundle $\widehat{\mathbb{V}}(X,\Theta,V)\rightarrow \mathrm{Sp}(X,\Theta)$ restricts to the bundle
$\widehat{\mathbb{V}}(C,V)\rightarrow \mathrm{Aut}(C)$ via the inclusion $\mathrm{Aut}(C)\hookrightarrow \mathrm{Sp}(X,\Theta)$

Moreover, the $\mathbb{S}\mathrm{p}^+\left(H' \right)$-equivariant structure on
$\widehat{\mathbb V}(X,\Theta,V)$ extends the combination of the
$\mathbb{A}\mathrm{ut}(\mathbb{O})$-equivariant structure on $\widehat{\mathbb V}(C,V)$ with the
invariance of the spaces $\mathbb V(C,P,V)$ with respect to $P$ in $C$. Indeed, one has
\[
\mathrm{Der}(\mathbb O):=\mathbb C\llbracket t\rrbracket\partial_t\subset\mathfrak{sp}^+(H'),
\]
including the element $\partial_t=-L_{-1}$, since $\partial_t(t)=1\equiv 0$ in $H'$,
and thus $\partial_t$ preserves $H'_+$. The $\mathbb{S}\mathrm{p}^+\left(H' \right)$-equivariant structure is
compatible with the action of $\mathfrak{sp}^+(H')$ from Theorem~\ref{thm:houtispreserved}. This
restricts to the action of $\mathrm{Der}(\mathbb O)$, whose restriction to
$\mathrm{Der}_0(\mathbb O)=t\mathbb{C}\llbracket t\rrbracket\partial_t$ integrates to the action of
$\mathbb{A}\mathrm{ut}(\mathbb O)\subset \mathbb{S}\mathrm{p}^+\left(H' \right)$. The invariance with respect to $P$ is the
parallel transport along the flat connection on $C$ defined by $L_{-1}$. It agrees with
the $\mathbb{S}\mathrm{p}^+\left(H' \right)$-equivariant structure, since both integrate the same Lie algebra action
and $C$ is connected. Thus one has a commutative diagram as in Figure~\ref{fig:2descents}.

\begin{figure}[t]
\begin{tikzcd}
&\widehat{\mathbb{V}}(X,\Theta,V) \arrow{dd} \arrow{rr} && \mathbb{V}(X,\Theta,V) \arrow{dd}\\
\widehat{\mathbb{V}}(C,V) \arrow{ru}\arrow{dd} \arrow[crossing over]{rr}&& \mathbb{V}(C,V) \arrow{ru}\\
&\mathrm{Sp}(X,\Theta) \arrow{rr} && \mathrm{Spec}\,(\mathbb{C})\\
\mathrm{Aut}(C) \arrow[hookrightarrow]{ru} \arrow{rr}&& \mathrm{Spec}\,(\mathbb{C}) \arrow{ru} \arrow[crossing over, leftarrow]{uu}. 
\end{tikzcd}
\caption{The two descents for a Jacobian.}
\label{fig:2descents}
\end{figure}

Hence the descent of $\widehat{\mathbb{V}}(X,\Theta,V)$ along $\mathrm{Sp}(X,\Theta) \rightarrow \mathrm{Spec}\,(\mathbb{C})$ extends the descent of $\widehat{\mathbb{V}}(C,V)$ along \sloppy{$\mathrm{Aut}(C) \rightarrow \mathrm{Spec}\,(\mathbb{C})$.} The statement follows.
\end{proof}

\subsection{Functional realization}
We conclude this section with the functional realization of the spaces of coinvariants for $V=\pi^0_\Gamma$:

\begin{theorem}
\label{thm:Vpi0Gamma}
For a PPAV $(X,\Theta)$, one has
a canonical isomorphism
\[
\mathbb{V}\left(X,\Theta, \pi^0_\Gamma\right) \cong \mathrm{Sym}\left( \mathfrak{h}_\Gamma\otimes_{\mathbb{C}} H^0(X, \Omega^1_X)^*\right).
\]
Equivalently, $\mathbb{V}\left(X,\Theta, \pi^0_\Gamma\right) \cong \mathrm{Fun}\left( \mathfrak{h}_\Gamma\otimes_{\mathbb{C}} H^0(X, \Omega^1_X)\right)$.
\end{theorem}

\begin{proof}
For a choice of an extended PPAV $a=(Z,F,L)$ with underlying PPAV $(X,\Theta)$, Example \ref{ex:pi0Gamma} yields
\begin{equation*}
\widehat{\mathbb{V}}\left(a,\pi_\Gamma^0\right) \cong \mathrm{Sym}\left( \mathfrak{h}_\Gamma\otimes_{\mathbb{C}}Z/F \right).
\end{equation*}
Composing with $Z\, / \, F \cong H^0\left(X, \Omega^1_X \right)^*$ from \eqref{eq:Z/F}, we have
\[
\widehat{\mathbb{V}}\left(a,\pi_\Gamma^0\right) \cong \mathrm{Sym}\left( \mathfrak{h}_\Gamma\otimes_{\mathbb{C}} H^0(X, \Omega^1_X)^*\right).
\]
As the right-hand side is independent of the choice of $a$ in $\mathrm{Sp}(X,\Theta)$, the statement follows.
\end{proof}

\begin{proof}[Proof of Theorem \ref{thm:Vpi0}]
This is the case $\mathrm{rank}(\Gamma) = 1$ of Theorem \ref{thm:Vpi0Gamma}.
\end{proof}

\section{Invariance of isotropy sheaves}
\label{sec:symmetagpsch}

Here we treat a global analogue of the invariance of isotropy sheaves from Corollary \ref{cor:sp+outpreserved} in preparation for the construction of spaces of coinvariants over the moduli space $\widehat{\mathcal{A}}_g$.
The actions of $\mathfrak{mp}\left(H'\right)$ on $H'$ and on $\widehat{\mathcal{A}}_g$ via the projection $\mathfrak{mp}\left(H'\right)\rightarrow \mathfrak{sp}\left(H'\right)$ induce an action of $\mathfrak{mp}\left(H'\right) \,\widehat{\otimes}_\mathbb{C}\, \mathscr{O}_{\widehat{\mathcal{A}}_g}$ on $H' \,\widehat{\otimes}_\mathbb{C}\, \mathscr{O}_{\widehat{\mathcal{A}}_g}$:
\begin{equation}
\label{eq:mpactiononH'sheaf}
(S\otimes f) \left( h\otimes k\right) = S(h)\otimes fk + h\otimes (S\otimes f)(k)
\end{equation}
for local sections $S\otimes f\in \mathfrak{mp}\left(H'\right) \,\widehat{\otimes}_\mathbb{C}\, \mathscr{O}_{\widehat{\mathcal{A}}_g}$ and $h\otimes k\in H' \,\widehat{\otimes}_\mathbb{C}\, \mathscr{O}_{\widehat{\mathcal{A}}_g}$.
Here $(S\otimes f)(k)$ is the derivative of $k$ along the vector field on $\widehat{\mathcal{A}}_g$ by which $S\otimes f$ acts in Theorem \ref{thm:spspHunif}(i).

Also, the actions of $\mathfrak{mp}\left(H'\right)$ on itself and on $\widehat{\mathcal{A}}_g$ induce an action of $\mathfrak{mp}\left(H'\right) \,\widehat{\otimes}_\mathbb{C}\, \mathscr{O}_{\widehat{\mathcal{A}}_g}$ on itself:
\begin{equation}
\label{eq:mpaction}
\left(S\otimes f\right)\left( T\otimes k\right) = [S,T]\otimes fk + T\otimes \left(S\otimes f\right)(k)
\end{equation}
for local sections $S\otimes f, T\otimes k\in\mathfrak{mp}\left(H'\right) \,\widehat{\otimes}_\mathbb{C}\, \mathscr{O}_{\widehat{\mathcal{A}}_g}$. 

Similarly, one has an affine action of $\mathscr H_2\,\widehat\otimes_{\mathbb C}\,
\mathscr O_{\widehat{\mathcal X}_g}$ on $H'\,\widehat\otimes_{\mathbb C}\,
\mathscr O_{\widehat{\mathcal X}_g}$ via the projection
$\mathscr H_2\to\mathfrak{sp}\left(H'\right)\ltimes H'$:
\begin{equation}
\label{eq:H2actiononH'sheaf}
(\widehat S\otimes f)(u)=(S\otimes f)(u)+s\otimes f
\end{equation}
for local sections $u\in H' \,\widehat{\otimes}_\mathbb{C}\, \mathscr{O}_{\widehat{\mathcal{X}}_g}$
and $\widehat{S}\otimes f\in \mathscr{H}_2 \,\widehat{\otimes}_\mathbb{C}\, \mathscr{O}_{\widehat{\mathcal{X}}_g}$ with image 
\[
(S,s)\otimes f\in \left(\mathfrak{sp} \left( H' \right) \ltimes H'\right) \,\widehat{\otimes}_\mathbb{C}\, \mathscr{O}_{\widehat{\mathcal{X}}_g},
\]
and with $(S\otimes f)(u)$ as in \eqref{eq:mpactiononH'sheaf}, where the derivative is taken along
the vector field on $\widehat{\mathcal X}_g$ by which $\widehat S\otimes f$ acts in
Theorem~\ref{thm:spspHunif}(ii).
Also, $\mathscr{H}_2\,\widehat{\otimes}_\mathbb{C}\, \mathscr{O}_{\widehat{\mathcal{X}}_g}$ acts on itself as in \eqref{eq:mpaction}.

Recall the sheaves of isotropy Lie subalgebras $\mathfrak{sp}^{\mathrm{out}} \left( H'\right)$ and $\mathfrak{i}$ from Figures \ref{fig:bigAtiyahLambdadiag} and \ref{fig:bigAtiyahXidiag}. 
For an extended PPAV $a=(Z, F, L)$, recall $K$ defined as the preimage of $L$ under the projection $F^\perp \rightarrow F^\perp/F$ as in \eqref{eq:H_+extX}. Let 
\[
\mathbb{F}\rightarrow\widehat{\mathcal{A}}_g\qquad \mbox{and}\qquad \mathbb{K}\rightarrow\widehat{\mathcal{A}}_g
\]
be the tautological sheaves restricting to $F$ and $K$ over $a$. In particular, 
$\mathbb{F}\subset
H' \,\widehat{\otimes}_\mathbb{C}\, \mathscr{O}_{\widehat{\mathcal{A}}_g}$ is an $\mathscr{O}_{\widehat{\mathcal{A}}_g}$-submodule, while $\mathbb{K}$ is a sheaf of abelian groups.
Let 
\[
\overline{\bm{h}} \in\left(H'\,\widehat{\otimes}_\mathbb{C}\, \mathscr{O}_{\widehat{\mathcal{X}}_g}\right)/\mathbb{K}
\]
be the tautological section defined by $(Z, F, L, \overline{h}, q)\mapsto \overline{h}$.
The following statement extends Corollary \ref{cor:sp+outpreserved}:

\begin{corollary}
\label{cor:spoutpreserved}
\begin{enumerate}[(i)]
\item The action of $\mathfrak{mp}\left(H'\right)\widehat{\otimes}_\mathbb{C}\, \mathscr{O}_{\widehat{\mathcal{A}}_g}$ on $H' \,\widehat{\otimes}_\mathbb{C}\, \mathscr{O}_{\widehat{\mathcal{A}}_g}$ preserves the subsheaf
\[
\mathbb{F} \subset
H' \,\widehat{\otimes}_\mathbb{C}\, \mathscr{O}_{\widehat{\mathcal{A}}_g}
\]
and maps $\mathbb{K} \subset
H' \,\widehat{\otimes}_\mathbb{C}\, \mathscr{O}_{\widehat{\mathcal{A}}_g}$ into $\mathbb{F}\subset \mathbb{K}$.

\item The action of $\mathfrak{mp}\left(H'\right)\widehat{\otimes}_\mathbb{C}\, \mathscr{O}_{\widehat{\mathcal{A}}_g}$ on itself
preserves the subsheaf 
\[
\mathfrak{sp}^{\mathrm{out}}\left(H' \right) \subset
\mathfrak{mp}\left( H'\right) \,\widehat{\otimes}_\mathbb{C}\, \mathscr{O}_{\widehat{\mathcal{A}}_g}.
\] 

\item The action of $\mathscr{H}_2 \,\widehat{\otimes}_\mathbb{C}\, \mathscr{O}_{\widehat{\mathcal{X}}_g}$ on $H'\,\widehat{\otimes}_\mathbb{C}\, \mathscr{O}_{\widehat{\mathcal{X}}_g}$
maps any lift in $H'\,\widehat{\otimes}_\mathbb{C}\, \mathscr{O}_{\widehat{\mathcal{X}}_g}$ of the section $\overline{\bm{h}}$ to $\mathbb{F}$.

\item The action of $\mathscr{H}_2\,\widehat{\otimes}_\mathbb{C}\, \mathscr{O}_{\widehat{\mathcal{X}}_g}$ on itself preserves 
the subsheaf 
\[
\mathfrak{i}\subset
\mathscr{H}_2 \,\widehat{\otimes}_\mathbb{C}\, \mathscr{O}_{\widehat{\mathcal{X}}_g}.
\]
\end{enumerate}
\end{corollary}

\begin{proof}
The transitive action of $\mathfrak{sp} \left( H' \right)$ on  $\widehat{\mathcal{A}}_g$ from Theorem \ref{thm:spspHunif}(i) is induced by the natural action of $\mathfrak{sp} \left( H' \right)$ on $H'$: for $S\in \mathfrak{sp} \left( H' \right)$, the corresponding tangent vector at a point $a=(Z,F,L)$ of $\widehat{\mathcal{A}}_g$ is the $R$-point 
\[
a_\epsilon := \left((1-\epsilon S)(Z\otimes_\mathbb{C} R),(1-\epsilon S)(F\otimes_\mathbb{C} R),(1-\epsilon S)L\right)
\]
where $R:=\mathbb{C}[\epsilon]/(\epsilon^2)$ \cite{adc91}.
For a local section $u\in H' \,\widehat{\otimes}_\mathbb{C}\, \mathscr{O}_{\widehat{\mathcal{A}}_g}$, one has
\[
u(a_\epsilon) = u(a) + \epsilon \left(S(u)(a) - S(u(a)) \right)
\]
from \eqref{eq:mpactiononH'sheaf}, and thus
\begin{equation}
\label{eq:uaepsilon}
u(a_\epsilon) = (1-\epsilon S)u(a) + \epsilon \, S(u)(a).
\end{equation}
Assume $u\in \mathbb{F}$. In particular $u(a)\in F$ and 
\[
u(a_\epsilon) \in (1-\epsilon S)(F\otimes_\mathbb{C} R).
\]
This implies that 
$u(a_\epsilon) = (1-\epsilon S)f +\epsilon h$ for some $f,h\in F$. Reducing modulo $\epsilon$, one has $f=u(a)$.
Thus \eqref{eq:uaepsilon} implies that $S(u)(a)=h\in F$, hence the stability of $\mathbb{F}$. 

Similarly, assume $u\in \mathbb{K}$. The fiber of $\mathbb{K}$ over $a_\epsilon$ consists of the elements of $(1-\epsilon S)(F^\perp\otimes_\mathbb{C} R)$ whose class modulo $(1-\epsilon S)(F\otimes_\mathbb{C} R)$ lies in $(1-\epsilon S)L$. As $L$ is discrete, every $R$-point of $L$ is constant in $\epsilon$. Thus $u(a_\epsilon) = (1-\epsilon S)k +\epsilon f$ for some $k\in K$ and $f\in F$. Reducing modulo $\epsilon$, one has $k=u(a)$.
Thus \eqref{eq:uaepsilon} implies that $S(u)(a)=f\in F$, hence part (i) follows. 

Part (ii) follows from the realization \eqref{eq:spoutrel} and part (i).

As in part (i), for $(S,s)\in \mathfrak{sp} \left( H' \right) \ltimes H'$, the corresponding tangent vector at a point $x=(Z, F, L, \overline{h}, q)$ of $\widehat{\mathcal{X}}_g$ is the $R$-point 
\[
x_\epsilon := \left((1-\epsilon S)Z,(1-\epsilon S)F,(1-\epsilon S)L, (1-\epsilon S)\overline{h} -\epsilon s, q\right)
\]
\cite{adc91}.
Similarly to \eqref{eq:uaepsilon}, for a local section $u\in H' \,\widehat{\otimes}_\mathbb{C}\, \mathscr{O}_{\widehat{\mathcal{X}}_g}$, one has
\begin{equation}
\label{eq:uaepsilon2}
u(x_\epsilon) = (1-\epsilon S)u(x) -\epsilon s+ \epsilon \, \widehat{S}(u)(x)
\end{equation}
from \eqref{eq:H2actiononH'sheaf}.
Assume $u$ is a lift of the section $\overline{\bm{h}}$. 
Since $\overline{\bm{h}}(x_\epsilon)=(1-\epsilon S)\overline{h} -\epsilon s$ by definition of $x_\epsilon$, 
both $u(x_\epsilon)$ and $(1-\epsilon S)u(x) -\epsilon s$ are lifts of $\overline{\bm{h}}(x_\epsilon)$, and thus
the element
\[
u(x_\epsilon) -(1-\epsilon S)u(x) +\epsilon s
\]
lies in the fiber of $\mathbb{K}$ over $x_\epsilon$ and vanishes modulo $\epsilon$. By the description of such a fiber in part (i), it is therefore of the form $\epsilon f$ for some $f\in F$. Comparing with \eqref{eq:uaepsilon2} yields $\widehat{S}(u)(x)=f\in F$, hence part (iii).

For part (iv), recall from Figure \ref{fig:bigAtiyahXidiag} that $\mathfrak{i}$ is the kernel of the map of Lie algebroids
\[
\begin{tikzcd}
\varphi\colon \mathscr{H}_2 \,\widehat{\otimes}_\mathbb{C}\, \mathscr{O}_{\widehat{\mathcal{X}}_g} \arrow[twoheadrightarrow]{r}
&\mathscr{F}_{\Xi}.
\end{tikzcd}
\]
Thus part (iv) follows from the general fact that the kernel of a map of Lie algebroids is preserved by the action of the source on itself: for local sections $s\in \mathscr{H}_2 \,\widehat{\otimes}_\mathbb{C}\, \mathscr{O}_{\widehat{\mathcal{X}}_g}$ and $t\in \mathrm{Ker}(\varphi)$, the action \eqref{eq:mpaction} of $s$ on $t$ coincides with the Lie algebroid bracket $[s, t]$, since $t$ lies in the kernel of the anchor map, and one has $\varphi([s,t])=[\varphi(s), \varphi(t)]=0$, hence $[s,t]\in \mathrm{Ker}(\varphi)$. 
This argument can also be used to give an alternative proof of part (ii).
\end{proof}

\section{Sheaves of coinvariants on $\widehat{\mathcal{A}}_g$ and ${\mathcal{A}}_g$}
\label{sec:sheavesaghatag}

Here we show that spaces of coinvariants globalize to a quasi-coherent sheaf $\widehat{\mathbb{V}}(V)$ on $\widehat{\mathcal{A}}_g$ carrying a twisted $\mathscr{D}$-module structure (Corollary \ref{cor:twistedDmodonVhat}). We then show that $\widehat{\mathbb{V}}(V)$ descends to a quasi-coherent sheaf ${\mathbb{V}}(V)$ on ${\mathcal{A}}_g$ (Corollary \ref{cor:VonAg}) also carrying a twisted $\mathscr{D}$-module structure (Corollary \ref{cor:twistedDmodonV}).
We conclude with the proof of Theorem \ref{thm:mainintroAg}.

\subsection{Sheaves of coinvariants on $\widehat{\mathcal{A}}_g$}
\label{sec:VhatonAg}
For a Heisenberg vertex algebra $V$, the spaces of coinvariants associated to extended PPAVs from \S\ref{sec:VhataV} globalize to a quasi-coherent sheaf on $\widehat{\mathcal{A}}_g$, which we denote as 
\[
\widehat{\mathbb{V}}(V)\rightarrow\widehat{\mathcal{A}}_g.
\]
Specifically, let $\mathbb{F}\rightarrow\widehat{\mathcal{A}}_g$ be the tautological sheaf from \eqref{eq:universalF} whose fiber at $(Z,F,L)$ is $F$.
Define
\begin{equation}
\label{eq:houtGamma}
\mathfrak{h}_\Gamma^{\mathrm{out}}:= \mathfrak{h}_\Gamma \otimes_{\mathbb{C}} \mathbb{F}.
\end{equation}
This is a subsheaf of the sheaf of Lie algebras $\mathscr{H}_\Gamma \,\widehat{\otimes}_\mathbb{C}\, \mathscr{O}_{\widehat{\mathcal{A}}_g}$.
The action of $\mathscr{H}_\Gamma$ on $V$ extends $\mathscr{O}_{\widehat{\mathcal{A}}_g}$-linearly to an action of $\mathscr{H}_\Gamma \,\widehat{\otimes}_\mathbb{C}\, \mathscr{O}_{\widehat{\mathcal{A}}_g}$ on $V\otimes_\mathbb{C} \mathscr{O}_{\widehat{\mathcal{A}}_g}$. This induces an action of $\mathfrak{h}_\Gamma^{\mathrm{out}}$ on $V\otimes_\mathbb{C} \mathscr{O}_{\widehat{\mathcal{A}}_g}$. Then the sheaf $\widehat{\mathbb{V}}(V)$ is defined as
\[
\widehat{\mathbb{V}}(V) := V\otimes_\mathbb{C} \mathscr{O}_{\widehat{\mathcal{A}}_g}
\, / \,
\mathfrak{h}_\Gamma^{\mathrm{out}}\left( V\otimes_\mathbb{C} \mathscr{O}_{\widehat{\mathcal{A}}_g}\right).
\]
The restriction of $\widehat{\mathbb{V}}(V)$ to the fiber of the map $\widehat{\mathcal{A}}_g\rightarrow \mathcal{A}_g$ over a point $(X, \Theta)$ in $\mathcal{A}_g$ recovers the quasi-coherent sheaf $\widehat{\mathbb{V}}(X,\Theta,V)\rightarrow\mathrm{Sp}(X,\Theta)$ from \S\ref{sec:VhatXthetaV}. In particular, one has the Cartesian diagram in Figure \ref{fig:VV}.

\begin{figure}[t]
\begin{tikzcd}
\widehat{\mathbb{V}}(V) \arrow{d} \arrow[hookleftarrow]{r}& \widehat{\mathbb{V}}(X,\Theta,V) \arrow{d} \arrow[hookleftarrow]{r} & \widehat{\mathbb{V}}(a, V) \arrow{d}\\
\widehat{\mathcal{A}}_g \arrow{d} \arrow[hookleftarrow]{r}& \mathrm{Sp}(X,\Theta) \arrow{d} \arrow[hookleftarrow]{r}{a} &  \mathrm{Spec}\,(\mathbb{C})\\
 \mathcal{A}_g \arrow[hookleftarrow]{r}{(X,\Theta)} & \mathrm{Spec}\,(\mathbb{C}).
\end{tikzcd}
\caption{The sheaf $\widehat{\mathbb{V}}(V)$.}
\label{fig:VV}
\end{figure}

\subsection{The metaplectic action}
Assume $V$ has non-zero level.
The metaplectic algebra $\mathfrak{mp}\left(H'\right)$ acts on $V$ as in \S\ref{sec:mpaction} and acts transitively on $\widehat{\mathcal{A}}_g$ from Theorem \ref{thm:spspHunif}(i).
These two actions induce an action of $\mathfrak{mp}\left(H'\right) \,\widehat{\otimes}_\mathbb{C}\, \mathscr{O}_{\widehat{\mathcal{A}}_g}$ on $V\otimes_\mathbb{C} \mathscr{O}_{\widehat{\mathcal{A}}_g}$:
\begin{equation}
\label{eq:mpOonVO}
\left(S\otimes h\right)\left( A\otimes f\right) = S(A)\otimes hf + A\otimes \left(S\otimes h\right)(f)
\end{equation}
for local section $S\otimes h\in \mathfrak{mp}\left(H'\right)\,\widehat{\otimes}_\mathbb{C}\, \mathscr{O}_{\widehat{\mathcal{A}}_g}$ and $A\otimes f\in V\otimes_\mathbb{C} \mathscr{O}_{\widehat{\mathcal{A}}_g}$.
Recall the sheaf of isotropy Lie subalgebras $\mathfrak{sp}^{\mathrm{out}} \left( H'\right)$ from Figure \ref{fig:bigAtiyahLambdadiag}. 

\begin{theorem}
\label{thm:houtispreservedmp}
\begin{enumerate}[(i)]
\item The subsheaf 
\[
\mathfrak{h}_\Gamma^{\mathrm{out}}\left( V\otimes_\mathbb{C} \mathscr{O}_{\widehat{\mathcal{A}}_g}\right) \subset
V\otimes_\mathbb{C} \mathscr{O}_{\widehat{\mathcal{A}}_g}
\]
is preserved by the action of $\mathfrak{mp}\left(H'\right) \,\widehat{\otimes}_\mathbb{C}\, \mathscr{O}_{\widehat{\mathcal{A}}_g}$ on $V\otimes_\mathbb{C} \mathscr{O}_{\widehat{\mathcal{A}}_g}$.
Hence one has an induced action of $\mathfrak{mp}\left(H'\right) \,\widehat{\otimes}_\mathbb{C}\, \mathscr{O}_{\widehat{\mathcal{A}}_g}$ on the quotient $\widehat{\mathbb{V}}(V)$.

\smallskip

\item The action of $\mathfrak{mp}\left(H'\right) \,\widehat{\otimes}_\mathbb{C}\, \mathscr{O}_{\widehat{\mathcal{A}}_g}$ on $\widehat{\mathbb{V}}(V)$ induces a trivial action of $\mathfrak{sp}^{\mathrm{out}}\left(H' \right)$ on $\widehat{\mathbb{V}}(V)$.
\end{enumerate}
\end{theorem}

\begin{proof}
Let
\[
\mathfrak{mp}_{\Gamma}\left(H' \right) 
\]
be the Lie subalgebra of $\mathscr{H}_\Gamma$ topologically generated by $\bm{1}$ and $A_mB_n$ with $A,B\in\mathfrak{h}_\Gamma$ and $m,n\in\mathbb{Z}\setminus \{0\}$. 
Consider the natural map
\begin{align*}
\left(\mathfrak{h}_\Gamma \otimes_{\mathbb{C}} \mathbb{F} \right)
\widehat{\otimes}_{\mathbb{C}} 
\left(\mathfrak{h}_\Gamma \,\widehat{\otimes}_\mathbb{C}\, H'\right)
&\rightarrow
\mathfrak{mp}_{\Gamma}\left(H' \right) \,\widehat{\otimes}_\mathbb{C}\, \mathscr{O}_{\widehat{\mathcal{A}}_g},\\
(A\otimes f)\otimes(B\otimes h)&\mapsto(A\otimes f)(B\otimes h),
\end{align*}
and let
\begin{equation}\label{eq:spoutsubsheaf}
\mathfrak{sp}_{\Gamma}^{\mathrm{out}}\left( H' \right)\subset
\mathfrak{mp}_{\Gamma}\left(H' \right) \,\widehat{\otimes}_\mathbb{C}\, \mathscr{O}_{\widehat{\mathcal{A}}_g}
\end{equation}
be its image. This is a subsheaf of Lie algebras.
The proof strategy is similar to the one of Theorem \ref{thm:houtispreserved} upon replacing $\mathfrak{sp}_{\Gamma}^{+, \mathrm{out}}\left(X,\Theta\right)$ there with $\mathfrak{sp}_{\Gamma}^{\mathrm{out}}\left( H' \right)$ here.

In particular, \eqref{eq:houtissp+out} extends here as
\begin{equation}
\label{eq:houtisspout}
\mathfrak{h}_\Gamma^{\mathrm{out}} \left( V\otimes_\mathbb{C} \mathscr{O}_{\widehat{\mathcal{A}}_g}\right) = 
\mathfrak{sp}_{\Gamma}^{\mathrm{out}}\left(H' \right) \left( V\otimes_\mathbb{C} \mathscr{O}_{\widehat{\mathcal{A}}_g}\right),
\end{equation}
which follows similarly. Thus it remains to show that the subsheaf 
\[
\mathfrak{sp}_{\Gamma}^{\mathrm{out}}\left(H' \right)
\left( V\otimes_\mathbb{C} \mathscr{O}_{\widehat{\mathcal{A}}_g}\right) \subset
V\otimes_\mathbb{C} \mathscr{O}_{\widehat{\mathcal{A}}_g}
\]
is preserved by the action of $\mathfrak{mp}\left(H'\right) \,\widehat{\otimes}_\mathbb{C}\, \mathscr{O}_{\widehat{\mathcal{A}}_g}$ on $V\otimes_\mathbb{C} \mathscr{O}_{\widehat{\mathcal{A}}_g}$.

One has an inclusion of Lie algebras
\[
\mathfrak{mp}\left(H' \right) \rightarrow \mathfrak{mp}_{\Gamma}\left(H' \right), \qquad
\bm{1}\mapsto d\,\bm{1}, \qquad
b_mb_n \mapsto \sum_{i=1}^d H^i_mH^i_n
\]
with notation as in Lemma \ref{lemma:mpaction}, and thus an action of $\mathfrak{mp}\left(H' \right) \,\widehat{\otimes}_\mathbb{C}\, \mathscr{O}_{\widehat{\mathcal{A}}_g}$ on \sloppy{$\mathfrak{mp}_{\Gamma}\left(H' \right) \,\widehat{\otimes}_\mathbb{C}\, \mathscr{O}_{\widehat{\mathcal{A}}_g}$} 
as in \eqref{eq:mpaction}. It is enough to show that this action preserves the subsheaf of Lie algebras
\eqref{eq:spoutsubsheaf}.

When $\mathrm{rank}(\Gamma) = 1$, the sheaf of Lie algebras \eqref{eq:spoutsubsheaf} recovers \eqref{eq:spoutrel}, that is, one has 
\begin{equation}
\label{eq:spGammaoutspout}
\mathfrak{sp}_{\Gamma}^{\mathrm{out}}\left( H' \right) = \mathfrak{sp}^{\mathrm{out}}\left( H' \right).
\end{equation}
In this case, the stability of the subsheaf \eqref{eq:spoutsubsheaf} under $\mathfrak{mp}\left(H' \right) \,\widehat{\otimes}_\mathbb{C}\, \mathscr{O}_{\widehat{\mathcal{A}}_g}$ follows from Corollary \ref{cor:spoutpreserved}(ii).

When $\mathrm{rank}(\Gamma) > 1$, first note that  
for $S\in \mathfrak{mp}\left(H' \right)$ and $A\otimes f\in \mathfrak{h}_\Gamma \,\widehat{\otimes}_\mathbb{C}\, H'$, one has
\begin{equation}
\label{eq:SmpAf}
[S, A\otimes f] = A\otimes [S, f]
\end{equation}
where the action $[S,f]$ is via the projection $\mathfrak{mp}\left(H' \right) \rightarrow \mathfrak{sp}\left(H' \right)$.
This can be verified easily as for \eqref{eq:SAf}.
Since $[S,\,]$ is a derivation and applying \eqref{eq:SmpAf}, one has
\begin{equation}
\label{eq:SAfBhmp}
[S,(A\otimes f)(B\otimes h)]= (A\otimes [{S},f])(B\otimes h) + (A\otimes f)(B\otimes [{S},h])
\end{equation}
for $S\in \mathfrak{mp}\left(H' \right)$ and $(A\otimes f)(B\otimes h)\in \mathfrak{mp}_{\Gamma}\left(H' \right)$.
Apply \eqref{eq:SAfBhmp} with $f$ replaced by a local section of $\mathbb{F}$ and interpret $[S,f]$ as the action in \eqref{eq:mpactiononH'sheaf}.
Then the stability of the subsheaf \eqref{eq:spoutsubsheaf} under $\mathfrak{mp}\left(H'\right)$ follows from the stability of $\mathbb{F}$ from Corollary \ref{cor:spoutpreserved}(i), hence part~(i).

Since $\mathfrak{sp}^{\mathrm{out}}\left(H' \right)$ is a sheaf of Lie subalgebras of $\mathfrak{sp}_{\Gamma}^{\mathrm{out}}\left(H' \right)$, part (ii) follows from \eqref{eq:houtisspout}.
\end{proof}

\begin{remark}\label{rmk:avvavsavha}
The coinvariants defined in this paper coincide with the coinvariants from \cite{avva},
where the sheaf of Lie algebras of outgoing states is
$\mathfrak{sp}^{\rm out}(H')$. That is, one has
\begin{equation}\label{eq:same-coinv}
\mathfrak{h}_\Gamma^{\mathrm{out}} \left( V\otimes_\mathbb{C} \mathscr{O}_{\widehat{\mathcal{A}}_g}\right) = 
\mathfrak{sp}^{\mathrm{out}}\left(H' \right) \left( V\otimes_\mathbb{C} \mathscr{O}_{\widehat{\mathcal{A}}_g}\right).
\end{equation}
The inclusion $\supseteq$ follows from \eqref{eq:houtisspout}, since
$\mathfrak{sp}^{\rm out}(H')$ is a sheaf of Lie subalgebras of
$\mathfrak{sp}^{\rm out}_\Gamma(H')$. For the opposite inclusion, let $f$ be a local
section of $\mathbb F$, let $v\in V$, and let $H^1$ be a unit vector in
$\mathfrak h_\Gamma$, completed to an orthonormal basis $\{H^i\}_i$. Select $n>0$ such
that $v$ does not involve $H^i_{-n}$ for all $i$, and let
$w:=\tfrac1n H^1_{-n}v$. Then $H^1_nw=v$ and $H^i_nw=0$ for $i\neq 1$. Hence the
section $fb_n$ of $\mathfrak{sp}^{\rm out}(H')$ acts on $w$ as
\[
(fb_n)\,w=
\sum_{i=1}^d (H^i\otimes f)\,H^i_n\,w=(H^1\otimes f)\,v.
\]
Since unit vectors span $\mathfrak h_\Gamma$, the identity \eqref{eq:same-coinv}
follows. When $\operatorname{rank}(\Gamma)>1$, the sheaf $\mathfrak{sp}^{\rm out}(H')$
is a sheaf of proper Lie subalgebras of $\mathfrak{sp}^{\rm out}_\Gamma(H')$, but the
two have the same image on $V\otimes_{\mathbb C}\mathscr O_{\widehat{\mathcal A}_g}$.
\end{remark}

As a consequence, we have the following result. Let $d:=\mathrm{rank}(\Gamma)$.

\begin{corollary}
\label{cor:twistedDmodonVhat}
The sheaf $\widehat{\mathbb{V}}(V)$ on $\widehat{\mathcal{A}}_g$ has a twisted $\mathscr{D}$-module structure
induced by an action of the Atiyah algebra $\frac{d}{2}\,\mathscr{F}_\Lambda$.
\end{corollary}

\begin{proof}
Recall the exact sequence
\[
\begin{tikzcd}
\mathfrak{sp}^{\mathrm{out}}\left(H' \right) \arrow[hookrightarrow]{r} &
\mathfrak{mp}\left( H'\right) \,\widehat{\otimes}_\mathbb{C}\, \mathscr{O}_{\widehat{\mathcal{A}}_g} \arrow[twoheadrightarrow]{r}
&\mathscr{F}_{\Lambda}
\end{tikzcd}
\]
from Figure \ref{fig:bigAtiyahLambdadiag}. 
From Theorem \ref{thm:houtispreservedmp}(i), one has an action of $\mathfrak{mp}\left( H'\right) \,\widehat{\otimes}_\mathbb{C}\, \mathscr{O}_{\widehat{\mathcal{A}}_g}$ on $\widehat{\mathbb{V}}(V)$. From Theorem \ref{thm:houtispreservedmp}(ii), the induced action of $\mathfrak{sp}^{\mathrm{out}}\left(H' \right)$ on $\widehat{\mathbb{V}}(V)$ is trivial. Hence the action of $\mathfrak{mp}\left( H'\right) \,\widehat{\otimes}_\mathbb{C}\, \mathscr{O}_{\widehat{\mathcal{A}}_g}$
 factors to an action of the Lie algebroid $\mathfrak{mp}\left( H'\right) \,\widehat{\otimes}_\mathbb{C}\, \mathscr{O}_{\widehat{\mathcal{A}}_g}\,/\,\mathfrak{sp}^{\mathrm{out}}\left(H' \right)$, which is an extension of $\mathscr{T}_{\widehat{\mathcal{A}}_g}$ by $\mathscr{O}_{\widehat{\mathcal{A}}_g}$. Thus $\widehat{\mathbb{V}}(V)$ is a twisted $\mathscr{D}$-module on $\widehat{\mathcal{A}}_g$. 
 
 The action of the above Lie algebroid is equivalent to an action of
 the Atiyah algebra $k\,\mathscr{F}_{\Lambda}$ on $\widehat{\mathbb{V}}(V)$ for some $k\in\mathbb{C}$.
From \S\ref{sec:mpaction}, the central element $\bm{1}$ in $\mathfrak{mp}\left( H'\right)$ acts on $V$ as multiplication by $d$. 
It follows from the diagram in Figure \ref{fig:bigAtiyahLambdadiag} that $k=\frac{d}{2}$. 
\end{proof}

\subsection{Equivariant structure}
In preparation for the descent to ${\mathcal{A}}_g$, we have the following global version of Theorem \ref{thm:expactionSp+}. 
When $V$ has non-zero level, recall the splitting
\[
\mathfrak{sp}^+\left(H'\right) \hookrightarrow \mathfrak{mp}\left(H'\right)
\]
from \eqref{eq:sp+splitting}. Thus the action of $\mathfrak{mp}\left(H'\right) \,\widehat{\otimes}_\mathbb{C}\, \mathscr{O}_{\widehat{\mathcal{A}}_g}$ on $\widehat{\mathbb{V}}(V)$ from Theorem \ref{thm:houtispreservedmp} induces an action of $\mathfrak{sp}^+\left(H'\right)$ on $\widehat{\mathbb{V}}(V)$. 
At level zero, the action of $\mathfrak{sp}^+\left(H'\right)$ on $\widehat{\mathbb{V}}(V)$ is the one induced by \S\ref{sec:possympact}.
Recall  the group ind-schemes
$\mathbb{S}\mathrm{p}^+\left(H' \right)$ and $\mathbb{S}\mathrm{p}^+_F\left(H' \right)$  from \S\ref{sec:spgp}.

\begin{theorem}
\label{thm:expactionSp+indgroup}
\begin{enumerate}[(i)]
\item The subsheaf 
\[
\mathfrak{h}_\Gamma^{\mathrm{out}}\left( V\otimes_\mathbb{C} \mathscr{O}_{\widehat{\mathcal{A}}_g}\right) \subset
V\otimes_\mathbb{C} \mathscr{O}_{\widehat{\mathcal{A}}_g}
\]
is preserved by the action of $\mathbb{S}\mathrm{p}^+\left(H' \right)$ on $V\otimes_\mathbb{C} \mathscr{O}_{\widehat{\mathcal{A}}_g}$.
This induces an $\mathbb{S}\mathrm{p}^+\left(H' \right)$-equivariant structure on the quotient $\widehat{\mathbb{V}}(V)$. Its differential is the action of $\mathfrak{sp}^+_b\left(H'\right)$.

\smallskip

\item For $a=(Z,F,L)$ in $\widehat{\mathcal{A}}_g$, the $\mathbb{S}\mathrm{p}^+\left(H' \right)$-equivariant structure on $\widehat{\mathbb{V}}(V)$ induces a trivial action of $\mathbb{S}\mathrm{p}^+_F\left(H' \right)$ on the fiber $\widehat{\mathbb{V}}(a,V)$ of $\widehat{\mathbb{V}}(V)$ at $a$.
\end{enumerate}
\end{theorem}

\begin{proof}
The statement follows similarly to its local version Theorem \ref{thm:expactionSp+}. 
\end{proof}

\subsection{Descent to ${\mathcal{A}}_g$}
Finally, we show the global version of \S\ref{sec:coinvonPPAVs}:

\begin{corollary}
\label{cor:VonAg}
The sheaf $\widehat{\mathbb{V}}(V)$ on $\widehat{\mathcal{A}}_g$ descends to a sheaf ${\mathbb{V}}(V)$ on ${\mathcal{A}}_g$.
\end{corollary}

\begin{proof}
From Theorem \ref{thm:expactionSp+indgroup}(i), the sheaf $\widehat{\mathbb{V}}(V)\rightarrow \widehat{\mathcal{A}}_g$ is 
$\mathbb{S}\mathrm{p}^+\left(H' \right)$-equivariant. Thus it descends to a sheaf \sloppy{$[\widehat{\mathbb{V}}(V)/ \mathbb{S}\mathrm{p}^+\left(H' \right) ]$} on the quotient stack $[\widehat{\mathcal{A}}_g/\mathbb{S}\mathrm{p}^+\left(H' \right)]$.

Moreover, from Theorem \ref{thm:Sp+action}(ii), 
the relative stabilizer  of 
$[\widehat{\mathcal{A}}_g\,/\, \mathbb{S}\mathrm{p}^+\left(H' \right)] \rightarrow {\mathcal{A}}_g$
at a point $(Z,F,L)$ in $\widehat{\mathcal{A}}_g$
is $\mathbb{S}\mathrm{p}^+_F\left(H' \right)$, and this
 acts trivially on $\widehat{\mathbb{V}}(V)$ by Theorem \ref{thm:expactionSp+indgroup}(ii). 
 Applying \cite[Thm 10.3]{alper2013good}, the sheaf 
$[\widehat{\mathbb{V}}(V)/ \mathbb{S}\mathrm{p}^+\left(H' \right) ]$ further descends to a sheaf $\mathbb{V}(V)$ on ${\mathcal{A}}_g$.
\end{proof}

\begin{figure}[t]
\begin{tikzcd}
\widehat{\mathbb{V}}(V) \arrow[]{r} \arrow[]{d} & \left[\widehat{\mathbb{V}}(V)/ \mathbb{S}\mathrm{p}^+\left(H' \right) \right]\arrow[]{r} \arrow[]{d} & \mathbb{V}(V) \arrow[]{d}\\
\widehat{\mathcal{A}}_g \arrow[]{r} & \left[\widehat{\mathcal{A}}_g/\mathbb{S}\mathrm{p}^+\left(H' \right)\right] \arrow[]{r} & {\mathcal{A}}_g.
\end{tikzcd}
\caption{The descent of $\widehat{\mathbb{V}}(V)$.}
\label{fig:VVdes}
\end{figure}

The diagram of descents is summarized in Figure \ref{fig:VVdes}.
Next, we show:

\begin{corollary}
\label{cor:twistedDmodonV}
The sheaf ${\mathbb{V}}(V)$ on ${\mathcal{A}}_g$ has a twisted $\mathscr{D}$-module structure
induced by an action of the Atiyah algebra $\frac{d}{2}\,\mathscr{F}_\Lambda$.
\end{corollary}

\begin{proof}
From Corollary \ref{cor:twistedDmodonVhat}, the action of $\mathfrak{mp}\left( H'\right) \,\widehat{\otimes}_\mathbb{C}\, \mathscr{O}_{\widehat{\mathcal{A}}_g}$ on $\widehat{\mathbb{V}}(V)$
 factors to an action of the Atiyah algebra $\frac{d}{2}\mathscr{F}_{\Lambda}$ on $\widehat{\mathbb{V}}(V)$.
From Theorem \ref{thm:expactionSp+indgroup}, the differential of the $\mathbb{S}\mathrm{p}^+\left(H' \right)$-equivariant structure on $\widehat{\mathbb{V}}(V)$ is the $\mathfrak{sp}^+_b\left(H'\right)$-action.
Since the actions of $\mathfrak{mp}\left( H'\right) \,\widehat{\otimes}_\mathbb{C}\, \mathscr{O}_{\widehat{\mathcal{A}}_g}$ 
and $\mathbb{S}\mathrm{p}^+\left(H' \right)$ on $\widehat{\mathbb{V}}(V)$ are compatible, 
it follows that the action of $\frac{d}{2}\mathscr{F}_{\Lambda}$ on $\widehat{\mathbb{V}}(V)$ is $\mathbb{S}\mathrm{p}^+\left(H' \right)$-equivariant and thus descends to an action of $\frac{d}{2}\mathscr{F}_{\Lambda}$ on the sheaf $[\widehat{\mathbb{V}}(V)/ \mathbb{S}\mathrm{p}^+\left(H' \right) ]$ over the quotient stack $[\widehat{\mathcal{A}}_g/\mathbb{S}\mathrm{p}^+\left(H' \right)]$.

The stabilizers of the closed points of $[\widehat{\mathcal{A}}_g/\mathbb{S}\mathrm{p}^+\left(H' \right)]$ act trivially on both 
$\widehat{\mathbb{V}}(V)$ (Theorem \ref{thm:expactionSp+indgroup}(ii)) and $\frac{d}{2}\mathscr{F}_{\Lambda}$ \cite[Proof of Thms~1 and 2]{avva}. It follows that the action of $\frac{d}{2}\mathscr{F}_{\Lambda}$ 
on the sheaf $[\widehat{\mathbb{V}}(V)/ \mathbb{S}\mathrm{p}^+\left(H' \right)]$ descends to an action of $\frac{d}{2}\mathscr{F}_{\Lambda}$ on ${\mathbb{V}}(V)$.
\end{proof}

We conclude with:

\begin{proof}[Proof of Theorem \ref{thm:mainintroAg}]
Part (i) follows from Theorem \ref{thm:VonTorelli} and Corollary \ref{cor:VonAg}, and part (ii) from Corollary \ref{cor:twistedDmodonV}.
\end{proof}

\section{Sheaves of twisted coinvariants on $\widehat{\mathcal{X}}_g$ and ${\mathcal{X}}_g$}
\label{sec:sheavesaghatxg}

Here we construct the sheaf of twisted coinvariants $\widehat{\mathbb{V}}(V,N)$ on $\widehat{\mathcal{X}}_g$ carrying a twisted $\mathscr{D}$-module structure (Corollary \ref{cor:twistedDmodonVhatXg}). We then show that $\widehat{\mathbb{V}}(V,N)$ descends to a quasi-coherent sheaf ${\mathbb{V}}(V,N)$ on ${\mathcal{X}}_g$ (Corollary \ref{cor:VonXg}) also carrying a twisted $\mathscr{D}$-module structure (Corollary \ref{cor:twistedDmodonVXg}).
We conclude with the proof of Theorem~\ref{thm:mainintroXg}.

\subsection{Spaces of twisted coinvariants associated to points of $\widehat{\mathcal{X}}_g$}
For a Heisenberg vertex algebra $V$ of non-zero level and a point $x=(Z, F, L, \overline{h}, q)$ in $\widehat{\mathcal{X}}_g$, construct spaces of twisted coinvariants as follows.

Select $N\in \mathfrak{h}_\Gamma$ such that $(N,N)=d$, as in Lemma \ref{lemma:mpaction}, and select
\mbox{$h\in \overline{h}+K$.}
Consider the algebra map
\begin{equation}
\label{eq:sigmaNh}
\sigma_{N,h}\colon \mathscr{H}_\Gamma \rightarrow \mathscr{H}_\Gamma
\end{equation}
determined on linear elements by 
\[
A\otimes u \mapsto A\otimes u - (A,N)\langle u, h \rangle\bm{1}.
\]
This is an automorphism of $\mathscr{H}_\Gamma$ (with inverse $\sigma_{N,-h}$), fixes $\bm{1}$, and is continuous since $\langle b_n, h \rangle=0$ for $n\gg 0$.
Define the \textit{Lie algebra of twisted outgoing states} associated to $x$ as
\begin{equation}
\label{eq:hGNoutx}
\mathfrak{h}_{\Gamma,N}^{\mathrm{out}}(x):= \sigma_{N,h}\left(\mathfrak{h}_\Gamma \otimes_{\mathbb{C}} \, F\right).
\end{equation}
This is a Lie subalgebra of $\mathscr{H}_\Gamma$ dependent on the choice of $N$ but independent of the choice of $h$ as above (since $K\subset F^\perp$).
Recall that the pair $(\overline{h}, q)$ is an element of the quotient \eqref{H'KtimesGQ} and is only determined up to the action of the group $G$ of points of order $2$ in $H'/K$. However, changing the pair by an element of $G$ results in shifting $h$ by some $k$ with $2k\in K\subset F^\perp$, thus leaving $\mathfrak{h}_{\Gamma,N}^{\mathrm{out}}(x)$ invariant.

The action of $\mathscr{H}_\Gamma$ on $V$ induces an action of $\mathfrak{h}_{\Gamma,N}^{\mathrm{out}}(x)$ on $V$.
The \textit{space of twisted coinvariants} of $V$ associated to $x$ and $N$ is 
\begin{equation}
\label{eq:VhatxVN}
\widehat{\mathbb{V}}(x,V,N):=V\,/\, \mathfrak{h}_{\Gamma,N}^{\mathrm{out}}(x) (V).
\end{equation}

\subsection{Sheaves of twisted coinvariants on $\widehat{\mathcal{X}}_g$}
Define the sheaf of twisted coinvariants
\[
\widehat{\mathbb{V}}(V,N)\rightarrow\widehat{\mathcal{X}}_g
\]
as follows.
Recall the sheaf $\mathfrak{h}_\Gamma^{\mathrm{out}}:= \mathfrak{h}_\Gamma \otimes_{\mathbb{C}} \mathbb{F}$ from \eqref{eq:houtGamma}, and define
\begin{equation}
\label{eq:hGNout}
\mathfrak{h}_{\Gamma,N}^{\mathrm{out}} := \sigma_{N,h}\left(\mathfrak{h}_\Gamma^{\mathrm{out}} \right)
\end{equation}
where we suppress the notation for the pull-back via $\widehat{\mathcal{X}}_g \rightarrow \widehat{\mathcal{A}}_g$, and $\sigma_{N,h}$ is the $\mathscr{O}_{\widehat{\mathcal{X}}_g}$-linear extension of the map from \eqref{eq:sigmaNh}.
The sheaf $\mathfrak{h}_{\Gamma,N}^{\mathrm{out}}$ globalizes $\mathfrak{h}_{\Gamma,N}^{\mathrm{out}}(x)$. 
One has an inclusion
\[
\mathfrak{h}_{\Gamma,N}^{\mathrm{out}} \hookrightarrow
\mathscr{H}_\Gamma \,\widehat{\otimes}_\mathbb{C}\, \mathscr{O}_{\widehat{\mathcal{X}}_g}.
\]
The action of $\mathscr{H}_\Gamma$ on $V$ extends $\mathscr{O}_{\widehat{\mathcal{X}}_g}$-linearly to an action of $\mathscr{H}_\Gamma \,\widehat{\otimes}_\mathbb{C}\, \mathscr{O}_{\widehat{\mathcal{X}}_g}$ on $V\otimes_\mathbb{C} \mathscr{O}_{\widehat{\mathcal{X}}_g}$. This induces an action of $\mathfrak{h}_{\Gamma,N}^{\mathrm{out}}$ on $V\otimes_\mathbb{C} \mathscr{O}_{\widehat{\mathcal{X}}_g}$. The sheaf of twisted coinvariants on $\widehat{\mathcal{X}}_g$ is defined as
\[
\widehat{\mathbb{V}}(V,N) := V\otimes_\mathbb{C} \mathscr{O}_{\widehat{\mathcal{X}}_g}
\, / \,
\mathfrak{h}_{\Gamma,N}^{\mathrm{out}}\left( V\otimes_\mathbb{C} \mathscr{O}_{\widehat{\mathcal{X}}_g}\right).
\]
This is a quasi-coherent sheaf on $\widehat{\mathcal{X}}_g$.

\subsection{The $\mathscr{H}_2$-action}
The Lie algebra $\mathscr{H}_2$ acts on $V$ as in \S\ref{sec:mpaction} and acts transitively on $\widehat{\mathcal{X}}_g$ from Theorem \ref{thm:spspHunif}(ii).
These two actions induce an action of $\mathscr{H}_2 \,\widehat{\otimes}_\mathbb{C}\, \mathscr{O}_{\widehat{\mathcal{X}}_g}$ on $V\otimes_\mathbb{C} \mathscr{O}_{\widehat{\mathcal{X}}_g}$ similar to~\eqref{eq:mpOonVO}.
Recall the sheaf of isotropy Lie subalgebras $\mathfrak{i}\subset \mathscr{H}_2 \,\widehat{\otimes}_\mathbb{C}\, \mathscr{O}_{\widehat{\mathcal{X}}_g}$ from Figure \ref{fig:bigAtiyahXidiag}. 

\begin{theorem}
\label{thm:houtispreservedU2H}
\begin{enumerate}[(i)]
\item The subsheaf 
\[
\mathfrak{h}_{\Gamma,N}^{\mathrm{out}}\left( V\otimes_\mathbb{C} \mathscr{O}_{\widehat{\mathcal{X}}_g}\right) \subset
V\otimes_\mathbb{C} \mathscr{O}_{\widehat{\mathcal{X}}_g}
\]
is preserved by the action of $\mathscr{H}_2 \,\widehat{\otimes}_\mathbb{C}\, \mathscr{O}_{\widehat{\mathcal{X}}_g}$ on $V\otimes_\mathbb{C} \mathscr{O}_{\widehat{\mathcal{X}}_g}$.
Hence one has an induced action of $\mathscr{H}_2 \,\widehat{\otimes}_\mathbb{C}\, \mathscr{O}_{\widehat{\mathcal{X}}_g}$ on the quotient $\widehat{\mathbb{V}}(V,N)$.

\smallskip

\item The action of $\mathscr{H}_2 \,\widehat{\otimes}_\mathbb{C}\, \mathscr{O}_{\widehat{\mathcal{X}}_g}$ on $\widehat{\mathbb{V}}(V,N)$ induces a trivial action of \sloppy{$\mathfrak{i}$} on $\widehat{\mathbb{V}}(V,N)$.
\end{enumerate}
\end{theorem}

\begin{proof}
The argument is similar to the one for Theorem \ref{thm:houtispreservedmp}.
Consider the complete topological vector space
\[
H'_\Gamma:= \mathfrak{h}_\Gamma \,\widehat{\otimes}_\mathbb{C}\, H'
\]
with symplectic form $\langle A\otimes f, B\otimes h\rangle=(A,B)\langle f, h\rangle$ and the symplectic algebra
\[
\mathfrak{sp}_{\Gamma}\left(H' \right) := \mathfrak{sp}\left(H'_{\Gamma} \right).
\]
This is the Lie algebra topologically generated by symmetric tensors $A_mB_n$ with $A,B\in\mathfrak{h}_\Gamma$ and $m,n\in \mathbb{Z}\setminus\{0\}$. 

Recall the sheaf of Lie algebras $\mathfrak{sp}_{\Gamma}^{\mathrm{out}}\left( H' \right)$ from \eqref{eq:spoutsubsheaf}
and the sheaf $\mathfrak{h}_\Gamma^{\mathrm{out}}$ from \eqref{eq:houtGamma}.
One has natural inclusions
\begin{align*}
&\mathfrak{sp}_{\Gamma}^{\mathrm{out}}\left( H' \right) \subset \mathfrak{sp}_{\Gamma}\left(H' \right) \,\widehat{\otimes}_\mathbb{C}\, \mathscr{O}_{\widehat{\mathcal{X}}_g},&
&\mathfrak{h}_\Gamma^{\mathrm{out}} \subset H'_\Gamma \,\widehat{\otimes}_\mathbb{C}\, \mathscr{O}_{\widehat{\mathcal{X}}_g}
\end{align*}
where we suppress the notation for the pull-back via $\widehat{\mathcal{X}}_g\rightarrow \widehat{\mathcal{A}}_g$.

Select a local section $h\in H' \,\widehat{\otimes}_\mathbb{C}\, \mathscr{O}_{\widehat{\mathcal{X}}_g}$ lifting the tautological section $\overline{\bm{h}}$, and let $b:= N\otimes h \in H'_\Gamma \,\widehat{\otimes}_\mathbb{C}\, \mathscr{O}_{\widehat{\mathcal{X}}_g}$. Consider the Lie algebra subsheaf
\[
\mathfrak{i}_{\Gamma, N} \subset \left(\mathfrak{sp}_{\Gamma} \left( H' \right) \ltimes H'_\Gamma\right)\widehat{\otimes}_\mathbb{C}\, \mathscr{O}_{\widehat{\mathcal{X}}_g}
\]
defined by
\begin{equation}
\label{eq:iGammax}
\mathfrak{i}_{\Gamma, N}:= 
\{X-X(b)+a\in \left(\mathfrak{sp}_\Gamma \left( H' \right) \ltimes H'_\Gamma\right) \,\widehat{\otimes}_\mathbb{C}\, \mathscr{O}_{\widehat{\mathcal{X}}_g}
 \, | \, X\in  \mathfrak{sp}_{\Gamma}^{\mathrm{out}} 
\mbox{ and } a\in \mathfrak{h}_\Gamma^{\mathrm{out}} \}.
\end{equation}
The semi-direct product $\mathfrak{sp}_{\Gamma}\left(H' \right)\ltimes H'_\Gamma$ is defined for the natural action of $\mathfrak{sp}_{\Gamma}\left( H' \right)$ on $H'_\Gamma$.
The choices of $N$ and $h$ as above determine an isomorphism 
\[
\mathfrak{sp}_{\Gamma}^{\mathrm{out}} \ltimes 
\mathfrak{h}_\Gamma^{\mathrm{out}}
 \xrightarrow{\cong}
\mathfrak{i}_{\Gamma, N}, \qquad
 (X,a)\mapsto X-X(b)+a.
\]
The Lie algebra sheaf $\mathfrak{i}_{\Gamma, N}$ depends on the choice of $N$ but is independent of the choice of $h$. 

Let $\mathscr{H}_{2, \Gamma}$
be the Lie subalgebra of $\mathscr{H}_\Gamma$ topologically generated by $\bm{1}$, $A_mB_n$, and $C_k$ with $A,B, C\in\mathfrak{h}_\Gamma$ and $m,n,k\in\mathbb{Z}\setminus \{0\}$. 
One has an inclusion of sheaves of Lie algebras
\begin{equation}
\label{eq:iGNintoH}
\mathfrak{i}_{\Gamma, N} \hookrightarrow
\mathscr{H}_{2, \Gamma} \,\widehat{\otimes}_\mathbb{C}\, \mathscr{O}_{\widehat{\mathcal{X}}_g}
\end{equation}
given by
\begin{align*}
(A\otimes f)(B\otimes \ell) - ((A\otimes f)(B\otimes \ell))(b)&\mapsto \sigma_{N,h}(A\otimes f)\,\sigma_{N,h}(B\otimes \ell), \\
C\otimes f &\mapsto \sigma_{N,h}(C\otimes f)
\end{align*}
where $A\otimes f, C\otimes f\in \mathfrak{h}_\Gamma^{\mathrm{out}}$, $B\otimes \ell\in H'_\Gamma$,
and $\sigma_{N,h}$ is the $\mathscr{O}_{\widehat{\mathcal{X}}_g}$-linear extension of the map from \eqref{eq:sigmaNh}.
Equivalently, one has
\begin{align*}
X - X(b)&\mapsto X - X(b) + \frac{1}{2}\langle X(b), b\rangle\bm{1}, \\
C\otimes f &\mapsto C\otimes f - (C,N)\langle f, h \rangle\bm{1}
\end{align*}
where $X$ stands for a product with the $\mathfrak{h}_\Gamma^{\mathrm{out}}$ factor on the left.
When $\mathrm{rank}(\Gamma) = 1$, \eqref{eq:iGNintoH} recovers \eqref{eq:isplitting}. 

The action of $\mathscr{H}_\Gamma$ on $V$ extends $\mathscr{O}_{\widehat{\mathcal{X}}_g}$-linearly to an action of $\mathscr{H}_\Gamma \,\widehat{\otimes}_\mathbb{C}\, \mathscr{O}_{\widehat{\mathcal{X}}_g}$ on $V\otimes_\mathbb{C} \mathscr{O}_{\widehat{\mathcal{X}}_g}$. This induces an action of $\mathfrak{i}_{\Gamma, N}$ on $V\otimes_\mathbb{C} \mathscr{O}_{\widehat{\mathcal{X}}_g}$. 

As in \eqref{eq:houtisspout}, we note that
\begin{equation}
\label{eq:houtisspouthout}
\mathfrak{h}_{\Gamma, N}^{\mathrm{out}} \left( V\otimes_\mathbb{C} \mathscr{O}_{\widehat{\mathcal{X}}_g}\right) = 
\mathfrak{i}_{\Gamma, N} \left( V\otimes_\mathbb{C} \mathscr{O}_{\widehat{\mathcal{X}}_g}\right).
\end{equation}
This follows by comparing the inclusion \eqref{eq:iGNintoH} with the definition \eqref{eq:hGNout}.
Thus it is enough to show that the subsheaf 
\[
\mathfrak{i}_{\Gamma, N} \left( V\otimes_\mathbb{C} \mathscr{O}_{\widehat{\mathcal{X}}_g}\right)
\subset V\otimes_\mathbb{C} \mathscr{O}_{\widehat{\mathcal{X}}_g}
\]
is preserved by the action of $\mathscr{H}_2 \,\widehat{\otimes}_\mathbb{C}\, \mathscr{O}_{\widehat{\mathcal{X}}_g}$ on $V\otimes_\mathbb{C} \mathscr{O}_{\widehat{\mathcal{X}}_g}$.

One has an inclusion of Lie algebras
\[
\mathscr{H}_{2} \rightarrow \mathscr{H}_{2, \Gamma}, \qquad
\bm{1}\mapsto d\,\bm{1}, \qquad
b_mb_n \mapsto \sum_{i=1}^d H^i_mH^i_n,
\qquad b_k \mapsto N_k
\]
with notation as in Lemma \ref{lemma:mpaction}, and thus an action of $\mathscr{H}_{2} \,\widehat{\otimes}_\mathbb{C}\, \mathscr{O}_{\widehat{\mathcal{X}}_g}$ on \sloppy{$\mathscr{H}_{2, \Gamma} \,\widehat{\otimes}_\mathbb{C}\, \mathscr{O}_{\widehat{\mathcal{X}}_g}$} 
as in \eqref{eq:mpaction}. It is enough to show that this action preserves the subsheaf of Lie algebras
\[
\mathfrak{i}_{\Gamma, N}
\subset \mathscr{H}_{2, \Gamma} \,\widehat{\otimes}_\mathbb{C}\, \mathscr{O}_{\widehat{\mathcal{X}}_g}. 
\]

When $\mathrm{rank}(\Gamma) = 1$, one has 
$\mathfrak{i}_{\Gamma, N}=\mathfrak{i}$. In this case, the statement follows from Corollary \ref{cor:spoutpreserved}(iv).
In general, for $\mathrm{rank}(\Gamma) \geq 1$, the image of the inclusion \eqref{eq:iGNintoH} shows that the stability 
of $\mathfrak{i}_{\Gamma, N}$ is implied by the stability of $\mathfrak{h}_{\Gamma,N}^{\mathrm{out}}$.
To show the stability of $\mathfrak{h}_{\Gamma,N}^{\mathrm{out}}$, consider a local section $\widehat{S}\otimes \ell\in \mathscr{H}_2 \,\widehat{\otimes}_\mathbb{C}\, \mathscr{O}_{\widehat{\mathcal{X}}_g}$ with image 
\[
(S,s)\otimes \ell\in \left(\mathfrak{sp} \left( H' \right) \ltimes H'\right) \,\widehat{\otimes}_\mathbb{C}\, \mathscr{O}_{\widehat{\mathcal{X}}_g}
\]
and $A\otimes f\in \mathfrak{h}_\Gamma^{\mathrm{out}}$.
One has
\begin{eqnarray}
\label{eq:stabhGNhout}
\begin{split}
[\widehat{S}\otimes \ell, \sigma_{N,h}(A\otimes f)] ={}& [\widehat{S}\otimes \ell, A\otimes f - (A,N)\langle f, h \rangle\bm{1}] \\
={}& A\otimes [S\otimes \ell, f] + (N,A)\langle s, f\rangle\bm{1}\ell \\
& - (A,N)[\widehat{S}\otimes \ell, \langle f, h \rangle\bm{1}]\\
={}& A\otimes [S\otimes \ell, f] - (A,N)\langle [S\otimes \ell, f], h\rangle\bm{1}\\
={}& \sigma_{N,h}\left(A\otimes [S\otimes \ell, f]\right).
\end{split}
\end{eqnarray}
In the second identity we applied \eqref{eq:SmpAf} and the bracket $[N \otimes s, A \otimes f] = (N, A)\langle s, f\rangle\bm{1}$.
In the third summand there, $\langle f, h \rangle\bm{1}$ is a section of $\mathscr{O}_{\widehat{\mathcal{X}}_g}$ and the derivation $[(S,s)\otimes \ell, ]$ is applied to it.
To make this explicit and explain the third identity, write 
\[
f= f_1\otimes f_2\in \mathbb{F}\subset H' \,\widehat{\otimes}_\mathbb{C}\, \mathscr{O}_{\widehat{\mathcal{X}}_g} \quad \mbox{and}\quad h=h_1\otimes h_2 \in H' \,\widehat{\otimes}_\mathbb{C}\, \mathscr{O}_{\widehat{\mathcal{X}}_g}.
\]
Here we assume for simplicity that $f$ and $h$ are pure tensors, the arbitrary case being similar by linearity and continuity.
Then
\begin{eqnarray*}
[\widehat{S}\otimes \ell, \langle f, h \rangle\bm{1}] &=&\langle f_1, h_1 \rangle \bm{1}\, [\widehat{S}\otimes \ell, f_2h_2]\\
&=& \langle f_1, h_1 \rangle \bm{1}\left( [\widehat{S}\otimes \ell, f_2]h_2 +f_2 [\widehat{S}\otimes \ell, h_2]\right)\\
&=& \langle (\widehat{S}\otimes \ell) (f) - S(f_1)\ell f_2 -s\ell, h\rangle\bm{1} \\
&&+ \langle f , (\widehat{S}\otimes \ell) (h) -S(h_1)\ell h_2 -s\ell\rangle\bm{1} \\
&=& \langle (S\otimes \ell) (f) - S(f_1)\ell f_2, h\rangle\bm{1} \\
&&+ \langle f , (\widehat{S}\otimes \ell) (h) -S(h_1)\ell h_2 -s\ell\rangle\bm{1} \\
&=& \langle ({S}\otimes \ell)(f), h\rangle\bm{1} 
+ \langle f , (\widehat{S}\otimes \ell) (h)\rangle \bm{1} 
- \langle f, s\rangle\bm{1}\ell\\
&=& \langle [S\otimes \ell, f], h\rangle\bm{1} + \langle s, f\rangle\bm{1}\ell.
\end{eqnarray*}
Here the second identity is the Leibniz rule, the third one is 
\eqref{eq:H2actiononH'sheaf}, 
the fourth one is \eqref{eq:mpactiononH'sheaf} and \eqref{eq:H2actiononH'sheaf},
the fifth one uses 
\[
\langle S(f_1)\ell f_2, h\rangle +\langle f , S(h_1)\ell h_2 \rangle 
= \left( \langle S(f_1), h_1\rangle +\langle f_1 , S(h_1) \rangle \right) \ell f_2 h_2 =0
\]
since $S\in \mathfrak{sp} \left( H' \right)$, see \eqref{eq:sp}, and the sixth one uses 
$({S}\otimes \ell)(f)=[S\otimes \ell, f]$,
and $\langle f , (\widehat{S}\otimes \ell)(h)\rangle=0$, which follows from $(\widehat{S}\otimes \ell) (h)\in\mathbb{F}$ due to Corollary \ref{cor:spoutpreserved}(iii) and the isotropy of $\mathbb{F}$.

One has $[S\otimes \ell, f]\in \mathbb{F}$ by the stability of (the pull-back of) $\mathbb{F}$ from Corollary \ref{cor:spoutpreserved}(i). 
Thus \eqref{eq:stabhGNhout} shows the stability of $\mathfrak{h}_{\Gamma,N}^{\mathrm{out}}$.
Hence part (i). 

Part (ii) follows from the fact that $\mathfrak{i}$ is a sheaf of Lie subalgebras of~$\mathfrak{i}_{\Gamma, N}$.
\end{proof}

\begin{remark}
\label{rmk:avvavsavha2}
As in Remark~\ref{rmk:avvavsavha}, one has
\[
\mathfrak i\bigl(V\otimes_{\mathbb C}\mathscr O_{\widehat{\mathcal X}_g}\bigr)
=\mathfrak i_{\Gamma,N}\bigl(V\otimes_{\mathbb C}\mathscr O_{\widehat{\mathcal X}_g}\bigr),
\]
by selecting $n$ there such that also $\langle b_n,h\rangle=0$. Hence combining with \eqref{eq:houtisspouthout}, it follows that the twisted
coinvariants here coincide with the twisted coinvariants defined by the action of $\mathfrak i$ from \cite{avva}.
\end{remark}

Consequently, we have the following result. Let $d:=\mathrm{rank}(\Gamma)$.

\begin{corollary}
\label{cor:twistedDmodonVhatXg}
The sheaf $\widehat{\mathbb{V}}(V,N)$ on $\widehat{\mathcal{X}}_g$ has a twisted $\mathscr{D}$-module structure
induced by an action of the Atiyah algebra $\frac{d}{2}\,\mathscr{F}_\Xi$.
\end{corollary}

\begin{proof}
Recall the exact sequence
\[
\begin{tikzcd}
\mathfrak{i} \arrow[hookrightarrow]{r} &
\mathscr{H}_2 \,\widehat{\otimes}_\mathbb{C}\, \mathscr{O}_{\widehat{\mathcal{X}}_g} \arrow[twoheadrightarrow]{r}
&\mathscr{F}_{\Xi}
\end{tikzcd}
\]
from Figure \ref{fig:bigAtiyahXidiag}.  
From Theorem \ref{thm:houtispreservedU2H}(i), one has an action of $\mathscr{H}_2 \,\widehat{\otimes}_\mathbb{C}\, \mathscr{O}_{\widehat{\mathcal{X}}_g}$ on $\widehat{\mathbb{V}}(V,N)$. From Theorem \ref{thm:houtispreservedU2H}(ii), the induced action of $\mathfrak{i}$ on $\widehat{\mathbb{V}}(V,N)$ is trivial. Hence the action of $\mathscr{H}_2 \,\widehat{\otimes}_\mathbb{C}\, \mathscr{O}_{\widehat{\mathcal{X}}_g}$
 factors to an action of the Lie algebroid $\mathscr{H}_2 \,\widehat{\otimes}_\mathbb{C}\, \mathscr{O}_{\widehat{\mathcal{X}}_g}\,/\,\mathfrak{i}$, which is an extension of $\mathscr{T}_{\widehat{\mathcal{X}}_g}$ by $\mathscr{O}_{\widehat{\mathcal{X}}_g}$. Thus $\widehat{\mathbb{V}}(V,N)$ is a twisted $\mathscr{D}$-module on $\widehat{\mathcal{X}}_g$. 

The action of the above Lie algebroid is equivalent to an action of the Atiyah algebra $k\,\mathscr{F}_{\Xi}$ on $\widehat{\mathbb{V}}(V,N)$ for some $k\in\mathbb{C}$.
From \S\ref{sec:mpaction}, the central element $\bm{1}$ in $\mathscr{H}_2$ acts on $V$ as multiplication by $d$. 
It follows from the diagram in Figure \ref{fig:bigAtiyahXidiag} that $k=\frac{d}{2}$. 
\end{proof}

\subsection{Equivariant structure}
\label{sec:equivstr8}
In preparation for the descent to ${\mathcal{X}}_g$, we have the following analogue of Theorem \ref{thm:expactionSp+indgroup}. 
The splitting \eqref{eq:sp+splitting} extends to a splitting
\[
\mathfrak{sp}^+\left(H'\right) \ltimes H'_+ \hookrightarrow \mathscr{H}_2.
\]
Thus the action of $\mathscr{H}_2 \,\widehat{\otimes}_\mathbb{C}\, \mathscr{O}_{\widehat{\mathcal{X}}_g}$ on $\widehat{\mathbb{V}}(V,N)$ from Theorem \ref{thm:houtispreservedU2H} induces an action of $\mathfrak{sp}^+\left(H'\right) \ltimes H'_+$ on $\widehat{\mathbb{V}}(V,N)$.

Let $\mathbb{O}_1^\times$ be the group scheme of invertible Taylor series with constant term $1$. This is the group scheme obtained by exponentiating $H'_+$, thus $\mathrm{Lie}\left(\mathbb{O}_1^\times\right)=H'_+$.

The group $\mathbb{O}_1^\times$ acts on $V$ through the rank-one Heisenberg subalgebra determined by $N$: 
 $\exp(u) \in \mathbb{O}_1^\times$ with $u \in H'_+$ acts as $\exp(N \otimes u)$, a finite sum on each vector since $N \otimes u$ is a sum of annihilation operators. Since $\rho(N \otimes u)\rho^{-1} = N \otimes \rho(u)$, this and the action of \S\ref{sec:equivstronVhatXThetaV} combine to an action of $\mathbb{S}\mathrm{p}^+\left(H' \right)\ltimes \mathbb{O}_1^\times$ on $V$, with $(\rho, \exp (u))$ acting as $\exp(N \otimes u)\circ\rho$; its differential is the action of $\mathfrak{sp}^+_b\left(H'\right) \ltimes H'_+$ through the splitting above.

The natural action of $\mathbb{S}\mathrm{p}^+\left(H' \right)\ltimes \mathbb{O}_1^\times$ on $H'$ induces a natural action of $\mathbb{S}\mathrm{p}^+\left(H' \right)\ltimes \mathbb{O}_1^\times$ on the fibers of $\widehat{\mathcal{X}}_g\rightarrow {\mathcal{X}}_g$. 
From \cite[(3.7) and Thm 3.2]{avva},
the transitive action of \mbox{$\mathfrak{sp} \left( H' \right)\ltimes H'$} on $\widehat{\mathcal{X}}_g$ is compatible with the transitive action of \sloppy \mbox{$\mathbb{S}\mathrm{p}^+\left(H' \right)\ltimes \mathbb{O}_1^\times$} on the fibers of  \mbox{$\widehat{\mathcal{X}}_g\rightarrow \mathcal{X}_g$}.
Also, the relative stabilizer of 
\[
\left[ \widehat{\mathcal{X}}_g \, / \left( \mathbb{S}\mathrm{p}^+\left(H' \right)\ltimes \mathbb{O}_1^\times \right) \right]\rightarrow {\mathcal{X}}_g
\]
at $x=(Z, F, L, \overline{h}, q)\in\widehat{\mathcal{X}}_g$ is
\begin{equation}
\label{eq:stabxinXg}
\left(\mathbb{S}\mathrm{p}^+\left(H' \right)\ltimes \mathbb{O}_1^\times\right)_x :=\left\{ 
(\rho, \exp(u_\rho))\,:\, \rho\in \mathbb{S}\mathrm{p}^+_F\left(H'\right)
\right\}
\end{equation}
where for each $\rho$, the element $u_\rho$ is defined as the projection in $H'_+$ of $(1-\rho)h\in F\oplus H'_+$ for some lift $h\in H'$ of $\overline{h}\in H'/K$. One has
\[
\mathrm{Lie}\left( \left(\mathbb{S}\mathrm{p}^+\left(H' \right)\ltimes \mathbb{O}_1^\times\right)_x\right) = \mathfrak{i}(x)\cap \left( \mathfrak{sp}^+_b\left(H'\right) \ltimes H'_+\right).
\]

\begin{theorem}
\label{thm:expactionSp+H'+indgroup}
\begin{enumerate}[(i)]
\item The subsheaf 
\[
\mathfrak{h}_{\Gamma,N}^{\mathrm{out}}\left( V\otimes_\mathbb{C} \mathscr{O}_{\widehat{\mathcal{X}}_g}\right) \subset
V\otimes_\mathbb{C} \mathscr{O}_{\widehat{\mathcal{X}}_g}
\]
is preserved by the action of $\mathbb{S}\mathrm{p}^+\left(H' \right) \ltimes \mathbb{O}_1^\times$ on $V\otimes_\mathbb{C} \mathscr{O}_{\widehat{\mathcal{X}}_g}$.
This induces an $(\mathbb{S}\mathrm{p}^+\left(H' \right) \ltimes \mathbb{O}_1^\times)$-equivariant structure on the quotient $\widehat{\mathbb{V}}(V,N)$. Its differential is the action of $\mathfrak{sp}^+_b\left(H'\right) \ltimes H'_+$.

\smallskip

\item For $x=(Z,F,L,\bar{h},q)$ in $\widehat{\mathcal{X}}_g$, the $(\mathbb{S}\mathrm{p}^+\left(H' \right) \ltimes \mathbb{O}_1^\times)$-equivariant structure on $\widehat{\mathbb{V}}(V,N)$ induces a trivial action of $\left(\mathbb{S}\mathrm{p}^+\left(H' \right)\ltimes \mathbb{O}_1^\times\right)_x$ on the fiber $\widehat{\mathbb{V}}(x,V,N)$ of $\widehat{\mathbb{V}}(V,N)$ at $x$.
\end{enumerate}
\end{theorem}

\begin{proof}
The statement follows similarly to Theorem \ref{thm:expactionSp+} together with the fact that 
conjugation by $\exp(N \otimes u)$ acts on $\mathfrak{h}_\Gamma \,\widehat{\otimes}_\mathbb{C}\, H'$ by 
\[
A \otimes f \mapsto A \otimes f + (A,N)\langle u,f\rangle\bm{1},
\]
hence carries $\sigma_{N,h}$ to $\sigma_{N,h+u}$ and $\mathfrak{h}_{\Gamma,N}^{\mathrm{out}}(x)$ to $\mathfrak{h}_{\Gamma,N}^{\mathrm{out}}((1, \exp (u))\cdot x)$.
\end{proof}

\subsection{Descent to ${\mathcal{X}}_g$}
Next, we descend to ${\mathcal{X}}_g$:

\begin{corollary}
\label{cor:VonXg}
The sheaf $\widehat{\mathbb{V}}(V,N)$ on $\widehat{\mathcal{X}}_g$ descends to a sheaf ${\mathbb{V}}(V,N)$ on~${\mathcal{X}}_g$.
\end{corollary}

\begin{proof}
The statement follows similarly to Corollary \ref{cor:VonAg} by making use of Theorem \ref{thm:expactionSp+H'+indgroup}.
\end{proof}

\begin{corollary}
\label{cor:twistedDmodonVXg}
The sheaf ${\mathbb{V}}(V,N)$ on ${\mathcal{X}}_g$ has a twisted $\mathscr{D}$-module structure
induced by an action of the Atiyah algebra $\frac{d}{2}\,\mathscr{F}_\Xi$.
\end{corollary}

\begin{proof}
The statement follows similarly to Corollary \ref{cor:twistedDmodonV} by making use of Corollary \ref{cor:twistedDmodonVhatXg} and Theorem \ref{thm:expactionSp+H'+indgroup} and the stabilizer \eqref{eq:stabxinXg}, with \cite[(3.6)]{avva} in place of \cite[(3.2)]{avva}.
\end{proof}


\section{Compatibility over relative Picard bundles}
\label{sec:comp}

Here we show the compatibility of the construction in \S\ref{sec:sheavesaghatxg} with the construction of twisted coinvariants on relative Picard bundles from \cite{bzf}.

\subsection{On relative Picard bundles}
Recall the extended relative degree-$(g-1)$ Picard bundle $\widehat{\mathcal{P}}_g\rightarrow \widehat{\mathcal{M}}_g$ from \S\ref{sec:Torelli}.
Let $\mathcal{M}_{g,1}$ be the moduli space of pointed curves $(C,P)$, where $C$ is a curve of genus $g$ and $P\in C$.
The forgetful map $\widehat{\mathcal{M}}_g \rightarrow \mathcal{M}_{g,1}$ is a principal $\mathbb{A}\mathrm{ut}(\mathbb{O})$-bundle (see \S\ref{sec:AutC}). Similarly, let $\mathcal{P}_{g,1}$ be the pull-back of $\mathcal{P}_g$ via $\mathcal{M}_{g,1}\rightarrow \mathcal{M}_g$. 
Then the forgetful map $\mathcal{P}_g \rightarrow \mathcal{P}_{g,1}$ is a principal $(\mathbb{A}\mathrm{ut}(\mathbb{O})\ltimes \mathbb{O}_1^\times)$-bundle. Recall from \S\ref{sec:equivstr8} that $\mathbb{O}_1^\times$ is the group scheme of invertible Taylor series with constant term $1$; it is enough to restrict to constant term $1$ since the trivializations are only determined up to a nonzero multiplicative constant.

\subsection{The twisted vertex algebra bundle}
Here we study the twisted vertex algebra bundle used in \cite[\S\S9.5.5, 18.1.15]{bzf}.

Let $V=\pi_\Gamma$ for some lattice $\Gamma$ of rank $d:=\mathrm{rank}(\Gamma)$.
Select $N\in \mathfrak{h}_\Gamma$ such that $(N,N)=d$, as in Lemma \ref{lemma:mpaction}.
This is equivalent to selecting a non-zero map (hence necessarily an embedding) $\pi\rightarrow \pi_\Gamma$ from the rank-one Heisenberg vertex algebra $\pi$ of level $d$.
In particular, $V$ has an action of $\mathrm{Der}_0(\mathbb{O}) \ltimes H'_+$ which exponentiates to an action of $\mathbb{A}\mathrm{ut}(\mathbb{O}) \ltimes \mathbb{O}_1^\times$
--- see \S\ref{sec:curvecasedescent} for $\mathrm{Der}_0(\mathbb{O})$ and $\mathbb{A}\mathrm{ut}(\mathbb{O})$.

Let $C$ be a curve of genus $g$, and $\mathscr{L}$ a line bundle of degree $0$ on $C$. Let $\widehat{\mathscr{L}}\rightarrow C$ be the principal $(\mathbb{A}\mathrm{ut}(\mathbb{O}) \ltimes \mathbb{O}_1^\times)$-bundle whose fiber at $P\in C$ consists of pairs $(t,\varphi)$ where $t$ is a local coordinate at $P$ and $\varphi$ is a trivialization of the restriction of $\mathscr{L}$ to  $D_P$ up to a nonzero multiplicative constant. Define the \textit{twisted vertex algebra bundle} $\mathscr{V}^{\mathscr{L}}\rightarrow C$ as the contracted product
\begin{equation}
\label{eq:VL}
\mathscr{V}^{\mathscr{L}}:= V\times_{\mathbb{A}\mathrm{ut}(\mathbb{O}) \ltimes \mathbb{O}_1^\times} \widehat{\mathscr{L}}.
\end{equation}
This is the descent of the trivial bundle $V\times\widehat{\mathscr{L}}\rightarrow \widehat{\mathscr{L}}$ along the principal $(\mathbb{A}\mathrm{ut}(\mathbb{O}) \ltimes \mathbb{O}_1^\times)$-bundle $\widehat{\mathscr{L}}\rightarrow C$. For $P\in C$, one has a non-canonical isomorphism of the fiber of $\mathscr{V}^{\mathscr{L}}$ at $P$:  
\[
\mathscr{V}^{\mathscr{L}}(P) \cong_{(t,\varphi)} V
\]
determined by a choice of a pair $(t,\varphi)$ as above.

The bundle $\mathscr{V}^{\mathscr{L}}$ admits the following natural filtration.
Let $V=\oplus_{i\geq 0}V_i$ be the standard grading of $V$, and let $V_{\leq k}:=\oplus_{i\leq k}V_i$.
Since the action of elements of $\mathrm{Der}_0(\mathbb{O}) \ltimes H'_+$ does not increase the degree, $V_{\leq k}$ is preserved by the action of $\mathrm{Der}_0(\mathbb{O}) \ltimes H'_+$, and thus preserved by the action of $\mathbb{A}\mathrm{ut}(\mathbb{O}) \ltimes \mathbb{O}_1^\times$ for each $k$.
Thus one defines the subbundle of $\mathscr{V}^{\mathscr{L}}$
\[
\mathscr{V}^{\mathscr{L}}_{\leq k}:= V_{\leq k}\times_{\mathbb{A}\mathrm{ut}(\mathbb{O}) \ltimes \mathbb{O}_1^\times} \widehat{\mathscr{L}}.
\]
One has an inclusion $\mathscr{V}^{\mathscr{L}}_{\leq k}\subset \mathscr{V}^{\mathscr{L}}_{\leq k+1}$.
The bundle $\mathscr{V}^{\mathscr{L}}$ is the direct limit of the bundles $\mathscr{V}^{\mathscr{L}}_{\leq k}$.
Let
\[
\mathrm{gr}_\bullet \mathscr{V}^{\mathscr{L}}:= \bigoplus_{k\geq 0}\mathrm{gr}_k \mathscr{V}^{\mathscr{L}}
\qquad
\mbox{where}
\qquad
\mathrm{gr}_k \mathscr{V}^{\mathscr{L}} := \mathscr{V}^{\mathscr{L}}_{\leq k+1}/\mathscr{V}^{\mathscr{L}}_{\leq k}.
\]
As in \cite[\S6.5.9]{bzf}, one has
\[
\mathrm{gr}_k \mathscr{V}^{\mathscr{L}} \cong \left( \Omega^{\otimes -k}_C\right)^{\oplus \dim V_k}.
\]
In particular, $\mathrm{gr}_k \mathscr{V}^{\mathscr{L}}$ is independent of $\mathscr{L}$, while the way $\mathscr{V}^{\mathscr{L}}$ is assembled from these graded pieces depends on $\mathscr{L}$.

\subsection{Twisted coinvariants on $\mathcal{P}_g$}
Sheaves of twisted coinvariants on $\mathcal{P}_g$ are constructed in \cite[\S\S9.5.5, 18.1.15]{bzf} as follows. 

For a quintuple $j:=(C,P,t,\mathscr{L}, \varphi)$ and a theta-characteristic $\eta$ on $C$, 
define 
\[
\mathscr{V}^{\mathscr{L}}_\eta := \mathscr{V}^{\mathscr{L}\otimes \eta^{-1}}
\]
and consider the composition of
 \begin{equation}
\label{eq:restsect}
 H^0\left(C\setminus P,  \mathscr{V}^{\mathscr{L}}_\eta \otimes \Omega^1_C \right) \rightarrow 
  H^0\left(D^\times_P,  \mathscr{V}^{\mathscr{L}}_\eta \otimes \Omega^1_C \right) \xrightarrow{\cong} V\,\widehat{\otimes}_\mathbb{C} \,\mathbb{C}((t))dt
 \end{equation}
where the first map is the restriction of sections and the second is the isomorphism determined by $(t,\varphi \, dt^{-1/2})$.
Furthermore, consider the linear map
 \begin{equation}
\label{eq:res}
V\,\widehat{\otimes}_\mathbb{C}\, \mathbb{C}((t))dt \rightarrow \mathrm{End}(V), \qquad A\otimes f(t)dt \mapsto \mathrm{Res}_{t=0} \,Y(A,t)f(t)dt
 \end{equation}
where $Y(A,t)$ is the vertex operator of $A\in V$ (\S\ref{sec:Hvoa}).

Composing the above maps, the \textit{Lie algebra of twisted outgoing states} $\mathscr{H}_{\Gamma,N}^{\mathrm{out}}(j)$  is defined as the image 
\[
\mathscr{H}_{\Gamma,N}^{\mathrm{out}}(j) := \mathrm{Im}\left(
 H^0\left(C\setminus P,  \mathscr{V}^{\mathscr{L}}_\eta \otimes \Omega^1_C \right) \rightarrow \mathrm{End}(V)
\right).
\]
Its action on $V$ is independent of $\eta$ (see the next Lemma \ref{lemma:Houtishout}).

The \textit{space of twisted coinvariants} of $V$ associated to $j$ is 
\begin{equation*}
\widehat{\mathbb{V}}(j,V,N):=V\,/\, \mathscr{H}_{\Gamma,N}^{\mathrm{out}}(j) (V).
\end{equation*}
These spaces globalize to a \textit{sheaf of twisted coinvariants}
\[
\widehat{\mathbb{V}}_{\widehat{\mathcal{P}}_g}(V,N)
\]
on $\widehat{\mathcal{P}}_g$. 
Then, one descends this sheaf along the principal $(\mathbb{A}\mathrm{ut}(\mathbb{O})\ltimes \mathbb{O}_1^\times)$-bundle $\widehat{\mathcal{P}}_g \rightarrow \mathcal{P}_{g,1}$. 
Finally, one shows that the resulting sheaf is constant over the fibers of $\mathcal{P}_{g,1}\rightarrow \mathcal{P}_g$, thus the sheaf further descends to a sheaf 
$\mathbb{V}_{\mathcal{P}_g}(V,N)$ on~$\mathcal{P}_g$.

\subsection{Compatibility}
Here we show the compatibility:
\begin{theorem}
\label{thm:VonJg}
The sheaf of twisted coinvariants $\mathbb{V}(V,N)$ on $\mathcal{X}_g$ restricts to the sheaf $\mathbb{V}_{\mathcal{P}_g}(V, N)$ via the inclusion ${\mathcal{P}}_g\rightarrow \mathcal{X}_g$.
\end{theorem}

Let $j$ be a point of $\widehat{\mathcal{P}}_g$ and $x$ its image via the map $\widehat{\mathcal{P}}_g\rightarrow\widehat{\mathcal{X}}_g$.
Recall $\mathfrak{h}_{\Gamma,N}^{\mathrm{out}}(x)$ from \eqref{eq:hGNoutx}.
We start by showing:

\begin{lemma}
\label{lemma:Houtishout}
One has
\[
\mathscr{H}_{\Gamma,N}^{\mathrm{out}}(j)(V) = \mathfrak{h}_{\Gamma,N}^{\mathrm{out}}(x)(V).
\]
\end{lemma}

\begin{proof}
First we consider elements of $\mathscr{H}_{\Gamma,N}^{\mathrm{out}}(j)$ that come from sections of \mbox{$\mathscr{V}^{\mathscr{L}}_{\eta, \leq 1} \otimes \Omega^1_C$} over $C\setminus P$ under the composition of the maps in \eqref{eq:restsect} and \eqref{eq:res}:
\[
\mathscr{H}_{\Gamma,N, \leq 1}^{\mathrm{out}}(j) := \mathrm{Im}\left(
 H^0\left(C\setminus P,  \mathscr{V}^{\mathscr{L}}_{\eta, \leq 1} \otimes \Omega^1_C \right) \rightarrow \mathrm{End}(V)
\right) \subset \mathscr{H}_{\Gamma,N}^{\mathrm{out}}(j).
\]
We claim that such elements coincide with elements of $\mathfrak{h}_{\Gamma,N}^{\mathrm{out}}(x)$, that is,
\begin{equation}
\label{eq:claimHoutishout}
\mathscr{H}_{\Gamma,N, \leq 1}^{\mathrm{out}}(j) = \mathfrak{h}_{\Gamma,N}^{\mathrm{out}}(x).
\end{equation}
Recall that  $V_0=\mathbb{C}\bm{1}$ and $V_1= \{A_{-1}\bm{1}\,: \, A\in \mathfrak{h}_{\Gamma} \}$.
Since $U:=C\setminus P$ is affine, the extension 
\[
0\rightarrow \Omega^1_C\,\bm{1} \rightarrow \mathscr{V}^{\mathscr{L}}_{\eta, \leq 1} \otimes \Omega^1_C
\rightarrow \mathrm{gr}_1 \mathscr{V}^{\mathscr{L}}_\eta\otimes \Omega^1_C \rightarrow 0
\]
splits over $U$, and thus
\[
H^0\left(U, \mathscr{V}^{\mathscr{L}}_{\eta, \leq 1} \otimes \Omega^1_C \right) = 
H^0\left(U,  \Omega^1_C\,\bm{1}\right) \oplus
s H^0\left(U,  \mathrm{gr}_1 \mathscr{V}^{\mathscr{L}}_\eta\otimes \Omega^1_C \right)
\]
for some splitting $s$. 
Two such splittings differ by an element of 
\[
\mathrm{Hom}_U\left(\mathrm{gr}_1 \mathscr{V}^{\mathscr{L}}_\eta\otimes \Omega^1_C, \Omega^1_C\,\bm{1}\right)
\]
which changes any lift by an element of $H^0(U,\Omega^1_C\,\bm{1})$.
The image in $\mathrm{End}(V)$ of sections of $\Omega^1_C\,\bm{1}$ over $C\setminus P$ is zero since $Y(\bm{1}, t)=\mathrm{id}_V$ and by the residue theorem. 
Thus the only elements of $\mathscr{H}_{\Gamma,N}^{\mathrm{out}}(j)$ contributed by sections of $\mathscr{V}^{\mathscr{L}}_{\eta, \leq 1} \otimes \Omega^1_C$ are those arising from sections in 
\begin{equation}
\label{eq:sigmaH0}
s H^0\left(U,  \mathrm{gr}_1 \mathscr{V}^{\mathscr{L}}_\eta\otimes \Omega^1_C \right)
\end{equation}
and these are independent of $s$.

Thus consider sections in \eqref{eq:sigmaH0} under the composition of the maps in \eqref{eq:restsect}. 
Note that $\mathrm{gr}_1 \mathscr{V}^{\mathscr{L}}_\eta\otimes \Omega^1_C \cong V_1\otimes \mathscr{O}_C$ canonically and independently of $\mathscr{L}$, since $\mathbb{O}_1^\times$ acts trivially on $V_{\leq 1}/V_0$. 
After selecting the pair $(t,\varphi \, dt^{-1/2})$, 
the image in $\mathscr{H}_{\Gamma,N}^{\mathrm{out}}(j)$ of sections in \eqref{eq:sigmaH0}  is thus of type
\begin{equation}
\label{eq:AfcjAf1}
A\otimes f + c_j(A\otimes f)\bm{1}  \qquad \mbox{with} \qquad A\otimes f \in \mathfrak{h}_{\Gamma}\otimes_\mathbb{C} (F\oplus \mathbb{C}t^0)
\end{equation}
for some map $c_j\colon \mathfrak{h}_{\Gamma}\otimes_\mathbb{C} (F\oplus \mathbb{C}t^0)\rightarrow \mathbb{C}$, which is linear since so is $s$.
Here $F$ is as in \eqref{eq:FonTorellilocus}; in particular, $t^0\not\in F$. Sections of $V_1\otimes_\mathbb{C} \mathscr{O}_C$ include the possibility that $f=1$, however the argument below shows that elements with $f=1$ could in fact be disregarded, as elements $A\otimes t^0$ act trivially on $V$ and \eqref{eq:cjisisgmatail} below implies that $c_j(A\otimes t^0)\bm{1}$ also acts trivially.

To prove the claimed identity \eqref{eq:claimHoutishout}, 
it remains to show that elements of type \eqref{eq:AfcjAf1} are in the image of the map $\sigma_{N,h}$ from \eqref{eq:sigmaNh}. Equivalently,  we show
\begin{equation}
\label{eq:cjisisgmatail}
c_j(A\otimes f) = - (A,N)\langle f, h \rangle \qquad \mbox{for} \qquad A\otimes f \in \mathfrak{h}_{\Gamma}\otimes_\mathbb{C} (F\oplus \mathbb{C}t^0).
\end{equation}

To determine this, we consider a larger set of elements in $\mathscr{H}_{\Gamma,N}^{\mathrm{out}}(j)$. 
The conformal vector $\omega\in V$ realizes an inclusion of the Virasoro vertex algebra (of central charge $d$) in $V$.
As in \cite[\S\S17.3, 19.6.5]{bzf}, the Virasoro outgoing elements lie in the isotropy of $j$. The isotropy of $j$ 
for the action of $\mathrm{Witt}\ltimes H'$ 
is $\mathfrak{i}(x)\cap \left( \mathrm{Witt}\ltimes H'\right)$ (see Fig.~\ref{fig:bigAtiyahXidiag}). Thus, 
for each $\xi\in T(C\setminus P)$, there exists a twisted outgoing element $L^y_\xi$ in $\mathscr{H}_{\Gamma,N}^{\mathrm{out}}(j)$ whose image modulo $\mathbb{C}\bm{1}$ is in $\mathfrak{i}_{\Gamma, N}(x)$ from \eqref{eq:iGammax} and whose $\mathrm{Witt}$ component is $\xi$, that is,
\[
L^y_\xi = \xi-N\otimes \xi(h)+a 
\qquad\mbox{for some $a\in \mathfrak{h}_{\Gamma}\otimes_\mathbb{C} F$, modulo $\mathbb{C}\bm{1}$.}
\]
Here we used $\xi(N\otimes h) = N\otimes \xi(h)$ from \eqref{eq:SmpAf}.
Also, $a$ arises as the lift of $\xi$ to $\mathscr{V}^{\mathscr{L}}_{\eta, \leq 2} \otimes \Omega^1_C$ over $U$ is determined only up to sections of $\mathscr{V}^{\mathscr{L}}_{\eta, \leq 1} \otimes \Omega^1_C$.
As $\mathbb{C}\bm{1}$ is central, the bracket in $\mathscr{H}_{\Gamma,N}^{\mathrm{out}}(j)$ yields
\begin{eqnarray*}
[L^y_\xi, A\otimes f ] &=& A\otimes \xi(f) -(N,A)\langle \xi(h), f\rangle\bm{1} +[a, A\otimes f]\\
&=& A\otimes \xi(f) -(A,N)\langle \xi(f), h\rangle\bm{1}
\end{eqnarray*}
where we used again \eqref{eq:SmpAf} together with $[a, A\otimes f]=0$ since $\langle F, F\oplus \mathbb{C}t^0\rangle=0$, and $\langle \xi(h), f\rangle = \langle \xi(f), h\rangle$ since $\xi\in \mathfrak{sp}\left(H' \right)$.
As $\xi(f)$ spans $F\oplus \mathbb{C}t^0$, applying $[L^y_\xi , ]$ to elements \eqref{eq:AfcjAf1} implies \eqref{eq:cjisisgmatail}.
This concludes the proof of the claimed identity \eqref{eq:claimHoutishout}.
Finally, the  statement is implied by the fact that
\[
\mathscr{H}_{\Gamma,N}^{\mathrm{out}}(j)(V) = 
\mathscr{H}_{\Gamma,N, \leq 1}^{\mathrm{out}}(j) (V).
\]
This says that the full image is generated by products of Heisenberg elements and follows as in \cite[\S\S9.3, 9.5.5]{bzf}.
\end{proof}

\begin{proof}[Proof of Theorem \ref{thm:VonJg}]
Lemma \ref{lemma:Houtishout} globalizes to $\mathscr{H}_{\Gamma,N}^{\mathrm{out}}(V) = \mathfrak{h}_{\Gamma,N}^{\mathrm{out}}(V)$, where the right-hand side is restricted to the Jacobian locus, and thus
the sheaf $\widehat{\mathbb{V}}(V,N)$ on $\widehat{\mathcal{X}}_g$ restricts to $\widehat{\mathbb{V}}_{\widehat{\mathcal{P}}_g}(V,N)$ over $\widehat{\mathcal{P}}_g$.

Moreover, as in the proof of Theorem~\ref{thm:VonTorelli}, the
$(\mathbb{S}\mathrm{p}^+\left(H' \right)\ltimes\mathbb O_1^\times)$-equivariant structure on
$\widehat{\mathbb V}(V,N)\to\widehat{\mathcal X}_g$ extends the combination of the
$(\mathbb{A}\mathrm{ut}(\mathbb O)\ltimes\mathbb O_1^\times)$-equivariant structure on
$\widehat{\mathbb V}_{\widehat{\mathcal P}_g}(V,N)$ with the invariance along the
fibers of $\mathcal P_{g,1}\to\mathcal P_g$. The only new ingredient is that the
equivariant structure is compatible with the action of
$\mathfrak{sp}^+(H')\ltimes H'_+\subset\mathscr H_2$, which restricts to the action of
$\mathrm{Der}(\mathbb O)\ltimes H'_+$.

This implies that the descent of $\widehat{\mathbb{V}}(V,N)$ along $\widehat{\mathcal{X}}_g\rightarrow \mathcal{X}_g$ extends the descent of $\widehat{\mathbb{V}}_{\widehat{\mathcal{P}}_g}(V,N)$ along $\widehat{\mathcal{P}}_g\rightarrow \mathcal{P}_g$. Hence the statement.
\end{proof}

We conclude with:

\begin{proof}[Proof of Theorem \ref{thm:mainintroXg}]
Part (i) follows from Theorem \ref{thm:VonJg} and part (ii) from Corollary \ref{cor:twistedDmodonVXg}.
\end{proof}


\section*{Acknowledgments} 
The author acknowledges partial support from the NSF award DMS-2404896 and a Simons Foundation's Travel Support for Mathematicians gift. 
This work is indebted to \cite{bzf} for the treatment of coinvariants on curves and to \cite{adc91} for extended PPAVs.
The author benefited from helpful conversations with
 Shashank Kanade on representations of the metaplectic algebra 
and with Filippo Viviani on quotient stacks.
The first version of this work was posted on the arXiv in March 2026 and was created with no assistance from AI. The first revision benefited from feedback from Anthropic's Claude, Fable 5 models, which pointed out a correction 
to \S\ref{sec:SpXTheta},
to Theorems \ref{thm:Sp+action}(ii) and \ref{thm:expactionSp+}, to Remark \ref{rmk:avvavsavha},
to \eqref{eq:sigmaNh}, the dependency of Theorem \ref{thm:mainintroXg} on a choice of $N\in \mathfrak{h}_\Gamma$, 
a correction to statements in \cite{avva} which carried over to this paper,
 and assisted the author in adding more details in several places. 
Some of its suggestions were incorrect and were discarded or corrected along the way.
 All statements and proofs have been verified by the author, who is solely responsible for
the content of the paper.


\bibliographystyle{alphanumN}
\bibliography{Biblio}

\end{document}